\documentclass{article}[11pt]
\usepackage{amsthm,amsmath,a4wide}
\usepackage{amssymb}
\input xypic
\usepackage[all,tips]{xy}

 \usepackage[usenames,dvipsnames]{pstricks}
 \usepackage{epsfig}
 \usepackage{pst-grad} 
 \usepackage{pst-plot} 

\usepackage[hidelinks]{hyperref}
\usepackage{color}

\newtheorem{thm}{Theorem}[section]
\newtheorem{lem}[thm]{Lemma}
\newtheorem{prop}[thm]{Proposition}
\newtheorem{cor}[thm]{Corollary}
\newtheorem{conj}[thm]{Conjecture}

\theoremstyle{definition}
\newtheorem{defin}[thm]{Definition}
\newtheorem{remark}[thm]{Remark}
\newtheorem{example}[thm]{Example}

\newcommand{\bth}{\begin{thm}}
\renewcommand{\eth}{\end{thm}}
\newcommand{\bpr}{\begin{prop}}
\newcommand{\epr}{\end{prop}}
\newcommand{\ble}{\begin{lem}}
\newcommand{\ele}{\end{lem}}
\newcommand{\bco}{\begin{cor}}
\newcommand{\eco}{\end{cor}}
\newcommand{\bde}{\begin{defin}}
\newcommand{\ede}{\end{defin}}
\newcommand{\bex}{\begin{example}}
\newcommand{\eex}{\end{example}}
\newcommand{\bre}{\begin{remark}}
\newcommand{\ere}{\end{remark}}
\newcommand{\bcj}{\begin{conj}}
\newcommand{\ecj}{\end{conj}}

\newcommand{\beq}{\begin{equation}}
\newcommand{\eeq}{\end{equation}}
\newcommand{\ve}{{\varepsilon}}
\newcommand{\ot}{{\otimes}}
\newcommand{\op}{{\oplus}}
\newcommand{\lb}{\label}
\newcommand{\nl}{\newline}
\newcommand{\bpf}{\begin{proof}}
\newcommand{\epf}{\end{proof}}

\newcommand{\da}{\text{-}}
\newcommand{\g}{{\frak g}}

\newcommand{\A}{{\cal A}}

\newcommand{\C}{{\cal C}}

\newcommand{\D}{{\cal D}}

\newcommand{\K}{{\cal K}}

\newcommand{\bN}{{\mathbb N}}

\newcommand{\G}{{\cal G}}
\renewcommand{\S}{{\cal S}}

\newcommand{\M}{{\cal M}}

\newcommand{\Mod}{{\cal M}{\it od}}

\newcommand{\Rep}{{\cal R}{\it ep}}

\newcommand{\Vect}{{\cal V}{\it ect}}

\newcommand{\De}{{\bold \Delta}}
\newcommand{\Ga}{{\bold \Gamma}}
\newcommand{\Fin}{\mathbf{Fin}_*} 

\newcommand{\bl}{{\{{\hspace{-2.5pt}}\{}}
\newcommand{\br}{{\}\hspace{-2.5pt}\}}}

\newcommand{\bV}{{\bf V}}

\newcommand{\bMon}{{\bf Mon}}
\newcommand{\bCom}{{\bf Com}}
\newcommand{\bSet}{{\bf Set}}

\newcommand\void[1]{}

\def\boxx{\mbox{$\scriptstyle\square$}}

\def\LL{\mathcal{L}}
\def\MM{\mathcal{M}}

\def\TT{\mathcal{S}k}

\def\EE{\mathcal{E}}

\def\FF{\mathcal{F}}
\def\FM{\mathcal{FM}}

\def\Cat{\mathbf{Cat}}

\def\Set{\mathrm{Set}}

\def\Mod{\mathrm{Mod}}

\def\sgn{\rm sgn}

\def\alt{\mathrm{alt}}

\def\LL{\mathcal{L}}

\def\BB{\mathcal{B}r}

\def\shuf{\mathcal{S}h}

\def\coend{\mathcal{C}}

\def\tr{\bold{sk}}
\def\link{\bold{ lk}}
\def\cc{\bold{c}}

\def\slp{\bold{slp}}

\def\nlp{\bold{nlp}}

\def\ll{|\mathcal{L}^{(n)}|}
\def\mm{|\mathcal{M}^{(n-1)}|}

\def\scup{{\scriptstyle \cup}}

\begin{document}
\author{Michael Batanin $^{a)}$ and Alexei Davydov $^{b)}$\footnote{corresponding author}}
\title{Cosimplicial monoids and deformation theory of tensor categories}
\maketitle
\date{}
\maketitle
\begin{center}\small
$^{a)}$ Department of Mathematics, 
Macquarie University, Sydney, NSW 2109, Australia 
\\ bataninmichael@gmail.com,\quad michael.batanin@mq.edu.au
\\ \quad  \\
$^{b)}$ Department of Mathematics, Ohio University, Athens, OH 45701, USA \\  davydov@ohio.edu 
\end{center}

\begin{abstract}
We introduce the notion of $n$-commutativity ($0\le n\le \infty$) for cosimplicial monoids  in a symmetric monoidal category $\bV$, where  $n=0$ corresponds to just cosimplicial monoids in $\bV,$ while $n=\infty$ corresponds to commutative cosimplicial monoids. When $\bV$ has a monoidal model structure we endow (under some mild technical conditions) the total object  of an $n$-cosimplicial monoid with a natural and very explicit $E_{n+1}$-algebra structure. 
\nl
Our main applications are to the deformation theory of tensor categories and tensor functors. We show that the deformation complex of a tensor functor  is a total complex of a $1$-commutative cosimplicial monoid and, hence, has an $E_2$-algebra structure similar to the $E_2$-structure on Hochschild complex of an associative algebra provided by Deligne's conjecture. We further demonstrate that the deformation complex of a tensor category is the total complex of a $2$-commutative cosimplicial monoid and, therefore, is naturally an $E_3$-algebra. We make these structures very explicit through a language of Delannoy paths and their noncommutative liftings. We investigate how these structures manifest themselves in concrete examples. 
\end{abstract}

\begin{center}2020 Mathematics Subject Classification: 18M05, 18M75, 18N50\end{center}

\begin{center}Keywords: cosimplicial algebras, Steenrod products, $E_n$-algebras, operads, tensor categories,  deformation complexes \end{center}

\tableofcontents

\section{Introduction}

There are two  categories naturally associated to any symmetric monoidal category $\bV$. Namely the category of monoids $\bMon(\bV)$ and its subcategory of commutative monoids $\bCom(\bV)$.
Generally speaking, due to the classical Eckman-Hilton argument, one can not define a natural intermediate subcategory of  monoids in between them without going to the realm of higher categories.    

Nevertheless, we argue in this paper that in some specific symmetric monoidal category one can have nontrivial intermediate subcategories of $n$-commutative monoids for every $0\le n\le \infty,$
 where   
 $n=0$ corresponds to monoids,  while $n=\infty$ corresponds to commutative monoids. These intermediate structures, have essentially set theoretical description, yet they model $E_{n+1}$-algebras, which is a higher categorical or homotopy theoretical concept.  

More precisely, let $\De$ be the category of nonempty finite ordinals with  nondecreasing maps. 
 A {\em cosimplicial monoid}   in $\bV$ is a  cosimplicial object in the category of monoids in $\bV$ that is a functor 
$E:\De \longrightarrow  \bMon(\bV).$
Similarly a {\em commutative cosimplicial monoid} is a cosimplicial object in the category $\bCom(\bV).$ 
\nl
Note that a  cosimplicial monoid (a commutative cosimplicial monoid) in $\bV$ is the same thing as a monoid (commutative monoid) in the category of cosimplicial objects in $\bV$ with respect to the pointwise symmetric monoidal structure.   

To give a precise definition of an $n$-commutativity   we need to introduce a  measure of complexity of interleaving between  two maps of finite ordinals.        
\nl
Let  $\tau:[p]\to [m],\ \pi:[q]\rightarrow [m]$ be two maps in $\De.$ {\it A shuffling of the pair $\tau,\pi$} of length $n$ is a decompositions of their images into disjoint unions:
$$Im(\tau) = A_{1}\cup ...\cup A_{s}, \quad A_{1}<...<A_{s}$$
$$Im(\pi) = B_{1}\cup ...\cup B_{t}, \quad  B_{1}<...<B_{t}, \ s+t -1 = n$$
which satisfy either $$A_{1}\leq B_{1}\leq A_{2}\leq B_{2} \leq... \  ,$$
or 
$$ B_{1}\leq A_{1}\leq B_{2}\leq A_{2} \leq...  \  .$$
The sign $A \le B$ means here that any element of $A$ is less or equal to any element of $B.$
\nl
We say that the {\it linking number} of the pair $(\tau,\pi)$  is $n$ if $n$ is the minimal number for which there exists a shuffling of length $n.$ 
\nl
 A cosimplicial monoid $E:\De\to \bMon(\bV)$  is {\em $n$-commutative} if 
for any two maps between finite ordinals : $\tau:[p]\to [n],\ \pi:[q]\to [n]$ whose linking number is less than or equal to $n,$ the images of $E(\tau)$ and $E(\pi)$ commute in $E(n)$  (Definition \ref{dnc}).

Let $\bMon(\bV)_n^\De$ be the category of $n$-commutative cosimplicial monoids. We have an infinite sequence of inclusions:
$$\bCom(\bV)^\De = \bMon(\bV)_{\infty}^\De\to \ldots \to \bMon(\bV)_n^\De  \to \ldots \to \bMon(\bV)_0^\De =\bMon(\bV)^\De.$$
 Our first result about $n$-cosimplicial monoids relates this tower to the tower of categories of algebras of the little $n$-cubes operads. Namely, if $\bV$ is equipped with a model structure satisfying some mild technical conditions and $\delta:\De\to\bV$ is a standard object of simplices in $\bV$ 
the total complex $Tot_{\delta}(E)$ of a cosimplicial $n$-commutative monoid   in  $\bV$  has a natural $E_{n+1}$-algebra structure, that is a structure of  algebra over an operad  weakly equivalent to the little $(n+1)$-cubes operad  (Theorem \ref{main}). 
 
A nice feature of this $E_{n+1}$-algebra structure is that it is  very combinatorially explicit. This is related to the aforementioned fact that cosimplicial monoids are essentially set theoretical objects. We take  advantage of this nature and describe explicitly the main structural operations (Steenrod $\cup_i$-products and Poisson bracket) on the total complex of an $n$-commutative cosimplicial algebra $E$ (Theorem \ref{bracketE}) as signed linear combinations of certain explicit operations on $E$. The terms are compositions of cosimplicial coface maps and multiplication in $E$. The combinatorics selecting the terms is controlled by liftings of  complexity  $n$ of { smooth Delannoy paths}  on a {commutative rectangular lattice}.
\nl
A Dellanoy path on the commutative $p\times q$-rectangular lattice is a path, which starts at the point $(0,0)$ and ends at the point $(p+1,q+1)$,  such that every step on this path adds $1$ or $0$ to one (or both) of the coordinates.   That is only $NE$ directions of movement are allowed. Such a path is called smooth if it does not have right angle corners. A lifting of such path $\psi$ is a path on the {noncommutative $p\times q$ lattice}  descending to $\psi$ under the quotient map making the lattice commutative. Such noncommutative path has complexity $k$ if it  changes direction exactly $k$ times. 
\nl
For $n=2$ the formula for the bracket which we obtain reproduces the Gershtenhaber-type bracket found earlier in \cite{da,ye}. 

The set of Delannoy paths is a very rich combinatorial object which appears naturally in many areas of mathematics (see \cite{BS} for a survey and a long bibliography). To the extent of our knowledge this paper is the first appearance of  Delannoy paths in homotopy theory. 
\nl
We want to stress a rare and satisfactory feature of the construction - that the resulting formulas for $E_{n+1}$-brackets are very explicit - with all the terms and their signs controlled by a rather well understood combinatorics. 
\nl
The formulas we obtain seem to be new even for $n=\infty,$ that is for  total complexes of commutative algebras - a classical setting for Steenrod operations \cite{Mandell,ms,st}.  
\\

In the second part of the paper we deal with our main examples of $n$-commutative cosimplicial monoids, the deformation complexes of tensor functors and of tensor categories introduced by the second author in \cite{da,da0} and independently by Yetter and Crane in \cite{cy,ye}. 

We show that the deformation complex of a tensor functor has a natural structure of a 1-commutative cosimplicial monoid (Section \ref{dctf}) and that the deformation complex of the identity functor of a tensor category is 2-commutative  (Section \ref{dctc}). 
We also show that the corresponding brackets control the obstructions to extending the first order deformations (Sections \ref{dttf} and \ref{dttc}). 

Deformation complexes of symmetric categories and functors exhibit features of both $E_{n+1}$- (for $n=2,1$) and $E_{\infty}$-algebras (Section \ref{dcsf}). In addition to the brackets their cohomology possesses a Hodge-type decomposition.
After looking at the combinatorics of such symmetric cosimplicial monoids  (Section \ref{scm}) we get a partial result on the interplay of these two structures (Theorem \ref{pac}). 

We illustrate our results with examples coming from symmetric categories of representations of Lie algebras (Section \ref{exampl}). 
We show that in characteristic zero the 1-bracket on the cohomology of the forgetful functor is the classical Schouten bracket.
We also show that the 2-bracket on the cohomology of the identity functor is trivial in characteristic zero (Theorem \ref{2bt}). This of course is in the total agreement with the very general prediction of M. Kontsevich on deformations of identity morphisms \cite{ko0}.
In finite characteristic the 2-bracket is non-trivial (Example \ref{nte}). We are planning to examine the case of finite characteristic systematically in a future work.
 
 \subsection*{Acknowledgment}

This paper had a rather long life. The work started in 2016, when the authors met in Max Planck Institute for Mathematics (Bonn, Germany) at the Program on Higher Structures in Geometry and Physics.
It is due to the ineffectiveness of the second author in managing his ever increasing load that it took four years to do the final polishing.
The authors  thank the Max Planck Institute for the opportunity to work together. 
The first author would also like to thank L'IHES for hospitality during January-February of 2020, crucial for  completion of the paper. 
The second author is partially supported by the Simons Foundation. 
The authors are grateful to Andr\'e Joyal and Maxim Kontsevich for illuminating discussions and explanations and Andrey Lazarev and Martin Markl for useful references.

\section{Cosimplicial monoids}

\subsection{Cosimplicial monoids and paths operads}

Let $\De$ be the category of nonempty finite ordered sets with  nondecreasing maps. 
Denote by $[n]=\{0,1,...n\}$ the ordered set of $n+1$ elements.
Denote by ${\partial}_{n}^{i}:[n]\to [n+1]$  the increasing monomorphism, which does not take the value $i\in [n+1]$. 
Denote by ${\sigma}_{n}^{i}:[n]\to [n-1]$  the nondecreasing epimorphism, which takes twice the value $i\in [n]$.
Graphically ${\partial}_{n}^{i}$ and ${\sigma}_{n}^{i}$ are 

$$\xygraph{
!{<0cm,0cm>;<.7cm,0cm>:<0cm,1cm>::}
!{(-2,2)}*{\circ} ="0" !{(-1,2)}*{...} !{(0,2)}*{\circ} ="1"  !{(1,2)}*{\circ}="2"  !{(2,2)}*{...}  !{(3,2)}*{\circ}="3"
!{(-2.5,0)}*{\circ}="d0" !{(-1.5,0)}*{...} !{(-.5,0)}*{\circ} ="d1" !{(0.5,0)}*{\circ}  !{(1.5,0)}*{\circ}="d2"  !{(2.5,0)}*{...}  !{(3.5,0)}*{\circ}="d3"
!{(-2,2.25)}*+{\scriptstyle 0}  !{(0,2.25)}*+{\scriptstyle i-1}  !{(1,2.25)}*+{\scriptstyle i}   !{(3,2.25)}*+{\scriptstyle n} 
!{(-2.5,-.23)}*+{\scriptstyle 0}  !{(-.5,-.23)}*+{\scriptstyle i-1}  !{(.5,-.23)}*+{\scriptstyle i}   !{(1.5,-.23)}*+{\scriptstyle i+1}  !{(3.5,-.23)}*+{\scriptstyle n+1} 
"0":"d0"     "1":"d1"  "2":"d2"  "3":"d3"
}\qquad\qquad\qquad
\xygraph{
!{<0cm,0cm>;<.7cm,0cm>:<0cm,1cm>::}
!{(-2,2)}*{\circ} ="0" !{(-1,2)}*{...} !{(0,2)}*{\circ} ="1"  !{(1,2)}*{\circ}="2"  !{(2,2)}*{...}  !{(3,2)}*{\circ}="3"
!{(-1.5,0)}*{\circ}="d0" !{(-.5,0)}*{...} !{(.5,0)}*{\circ} ="d1" !{(1.5,0)}*{...}  !{(2.5,0)}*{\circ}="d3"
!{(-2,2.25)}*+{\scriptstyle 0}  !{(0,2.25)}*+{\scriptstyle i}  !{(1,2.25)}*+{\scriptstyle i+1}   !{(3,2.25)}*+{\scriptstyle n} 
!{(-1.5,-.23)}*+{\scriptstyle 0}  !{(.5,-.23)}*+{\scriptstyle i}  !{(2.5,-.23)}*+{\scriptstyle n-1}   
"0":"d0"     "1":"d1"  "2":"d1"  "3":"d3"
}$$
correspondingly.

Let $\Cat$ be the category of small categories. We say that an object $m$ of a small category $A$ is {\it weakly  initial} if the Hom set $A(m,a)$ is nonempty for any $a\in A.$   
Dually $m\in A$ is {\it weakly terminal} if it is weakly initial in $A^{op}.$ 

Let $\Cat_{*,*}$ be the subcategory of $\Cat$ whose objects are small categories with distinguished weakly initial and weakly terminal objects and whose morphisms are functors which preserve these objects.    
\nl
For each pair $0\le  i < j $ the  finite linear poset
$ i <  i+1 < \ldots < j $ 
freely generates an object $\langle i, j\rangle$ of $\Cat_{*,*}$ (called an {\it interval}) with the initial object $i$ and the terminal object $j.$  We will denote by $\bar{k}: k\to k+1$ a generating morphism of this category. The interval $\langle 0, n\rangle$ will be denoted simply by $\langle n \rangle .$
 
The full subcategory of $\Cat_{*,*}$ spanned by the categories $\langle n \rangle \ , \ n\ge 1$ is called the {\em category of intervals} $\mathbf{Int}.$
\nl
The {\it Joyal's duality} is the isomorphism of categories
$$\widetilde{(\ )}:\De \to \mathbf{Int}^{op}, \qquad \widetilde{[n]} = \langle n+1\rangle. $$
We will need an explicit description of the effect of this isomorphism on morphisms of $\De.$ 
\nl
Let $\phi: [m]\to [l]$ be a morphism in $\De.$ We define a map $\widetilde{\phi}: \langle l+1\rangle \to \langle m+1\rangle$ as a functor which on the generator $\bar{i}: i\to i+1$
is equal to $m_i\to M_i +1$ provided $i \in Im(\phi)$ with $m_i = min\{j\in \phi^{-1}(i)\}$ and $M_i = max\{j\in \phi^{-1}(i)\}.$ If $i$ is not in the image of $\phi$ we put $\widetilde{\phi}(\bar{i}) = id.$

\begin{example}
Here is a picture of the Joayl dual of a map $\phi:[3]\to[2]$ (in solid arrows)
$$\xygraph{
!{<0cm,0cm>;<1cm,0cm>:<0cm,1cm>::}
!{(1,2)}*{\bullet}="1" !{(2,2)}*{\circ}="2" !{(3,2)}*{\bullet}="3" !{(4,2)}*{\circ}="4"  !{(5,2)}*{\bullet}="5" !{(6,2)}*{\circ}="6" 
!{(7,2)}*{\bullet}="7" !{(8,2)}*{\circ}="8"  !{(9,2)}*{\bullet}="9"  
!{(2,0)}*{\bullet}="d1" !{(3,0)}*{\circ}="d2" !{(4,0)}*{\bullet}="d3" !{(5,0)}*{\circ}="d4"  !{(6,0)}*{\bullet}="d5" !{(7,0)}*{\circ}="d6"  !{(8,0)}*{\bullet}="d7" 
"2":"d4"  "4":"d6"  "6":"d6"   "8":"d6" 
"d1":@{..>}"1"   "d3":@{..>}"1"  "d5":@{..>}"3"  "d7":@{..>}"9" 
}$$
with the intervals $\langle 4\rangle$ (at the top) and $\langle 3\rangle$ (at the bottom) represented by solid dots
and the Joyal dual map $\widetilde{\phi}:\langle 3\rangle\to\langle 4\rangle$ (in dotted arrows) going upwards.
\end{example}
\bigskip

Let $\bV = (\bV,\otimes,I)$ be a symmetric monoidal category. Let $\bMon(\bV)$ be the category of monoids in $\bV$ and  $\bCom(\bV)$ to be the category of commutative monoids in $\bV.$

\begin{defin} A {\em cosimplicial monoid} in $\bV$ is a  cosimplicial object in the category of monoids in $\bV$ that is a functor 
$$E:\De \longrightarrow  \bMon(\bV), \ E(n) = E([n]).$$
 A {\em commutative cosimplicial monoid} is a cosimplicial object in the category $\bCom(\bV).$ 
\end{defin}

We are going to construct a coloured operad $\MM$ in $\bSet$ whose algebras are cosimplicial commutative monoids. 
First observe that  for  any $n_1,\ldots,n_k \ge 0$ the cartesian product in $\Cat$  $\langle n_1 +1\rangle \times\cdots\times\langle n_k +1\rangle $ has unique initial and terminal objects. 

\begin{defin} The {\em paths operad} $\MM$ has natural numbers as colours. The set of operations is
$$\MM(n_1,\ldots, n_k; n)=\Cat_{*,*}(\langle n+1\rangle ,\langle n_1 +1\rangle \times\cdots\times\langle n_k +1\rangle ),$$
and the operad substitution maps being induced by cartesian product and composition in $\Cat_{*,*}.$ 
\end{defin}

\begin{example} An element of $\MM(p,q;m)$ can be understood as a path (possibly with some ``stops")  in a commutative $(p+1) \times (q +1)$ lattice which goes from $(0,0)$ to $(p+1,q+1)$ along north to east directions.

$$\xygraph{
!{<0cm,0cm>;<0.6cm,0cm>:<0cm,0.6cm>::}
!{(-1,3)}*{\circ}  !{(0,3)}*{\circ}   !{(1,3)}*{\circ}  !{(2,3)}*{\circ} !{(3,3)}*{\circ}   !{(4,3)}*{\circ}  !{(5,3)}*{\bullet}="5"
!{(-1,2)}*{\circ}  !{(0,2)}*{\circ}   !{(1,2)}*{\circ}  !{(2,2)}*{\circ} !{(3,2)}*{\circ}   !{(4,2)}*{\circ}  !{(5,2)}*{\circ}
!{(-1,1)}*{\circ}  !{(0,1)}*{\circ}  !{(1,1)}*{\circ}  !{(2,1)}*{\circ} !{(3,1)}*{\bullet}="3"  !{(4,1)}*{\bullet}="4"  !{(5,1)}*{\circ}
!{(-1,0)}*{\circ}  !{(0,0)}*{\circ}   !{(1,0)}*{\circ}  !{(2,0)}*{\circ}  !{(3,0)}*{\circ}   !{(4,0)}*{\circ}  !{(5,0)}*{\circ}
!{(-1,-1)}*{\bullet} ="1" !{(0,-1)}*{\circ}   !{(1,-1)}*{\circ}  !{(2,-1)}*{\bullet}="2" !{(3,-1)}*{\circ}   !{(4,-1)}*{\circ}  !{(5,-1)}*{\circ}
!{(2.7,1.3)}*+{\scriptstyle 2}  
"1":"2"     "2":"3"  "3":"4"  "4":"5"  
}$$
In this picture the corresponding path is $\phi:\langle 5 \rangle \to \langle 6\rangle \times \langle 4 \rangle.$ Black dots correspond to images of  objects of the interval $\langle 5 \rangle.$ The parts of the path between dots correspond to  images of the generating morphisms in $\langle 5 \rangle.$ The label $2$ near the middle dot indicates that the generator $\bar{2}$ is mapped by $\phi$ to the identity morphism of $(5,2)$ (so, the preimage $\phi^{-1}(id_{(5,2)}) = \{id_2,\bar{2},id_3\}$.) Since the number of different  identities in this preimage is $2$ we will say that the path has two stops at this point.    The dots without labels are those objects $(x,y)$ for which $\phi^{-1}(id_{(x,v)})$ is a singleton. We will say that the path has one stop in these points. This path has $0$-stops in  all other points.

\end{example}

\begin{thm}\label{bvcom} The category of algebras of $\MM$ in any cocomplete symmetric monoidal category $(\bV,\otimes,I)$ is isomorphic to the category of cosimpicial commutative monoids in $\bV.$
\end{thm}

\begin{proof} Observe that the underlying category of $\MM$ is  isomorphic to $\De$ by Joyal duality.  Any coloured operad with value in a cocomplete symmetric monoidal category induces a symmetric multitensor structure (standard Day-Sreet convolution)
on the category of covariant presheaves on underlying category. The category of algebras of the operad $\MM$ is isomorphic  to the category of commutative monoids with respect to this multitensor structure  \cite{DS2}   .

In the case of $\MM$ we have  a multitensor on $\bV^{\De}$ 
$$\xi_k: (\bV^\De)^k \to \bV^{\De}$$
given by the  coend formula
$$\xi_k(X_1,\ldots,X_k)(n) = \int^{n_1,\ldots,n_k\in\De}\MM(n_1,\ldots,n_k;n)\otimes X_1(n_1)\otimes\ldots\otimes X_k(n_k).$$ 
Now last coend is equal to 
$$\int^{n_1,\ldots,n_k\in\De}\Cat_{*,*}(\langle n+1\rangle ,\langle n_1 +1\rangle \times\cdots\times\langle n_k +1\rangle )\otimes X_1(n_1)\otimes\ldots\otimes X_k(n_k) \simeq  $$ 
$$\int^{n_1,\ldots,n_k\in\De}\Cat_{*,*}(\langle n+1\rangle,\langle n_1 +1\rangle)\times\cdots\times \Cat_{*,*}(\langle n+1\rangle,\langle n_k +1\rangle)\otimes X_1(n_1)\otimes\ldots\otimes X_k(n_k) \simeq $$
$$ \int^{n_1,\ldots,n_k\in\De}(\De([n_1],[n])\times\cdots\times \De([n_k],[n]))\otimes X_1(n_1)\otimes\ldots\otimes X_k(n_k)$$
Using the fact that the tensor product $\otimes$ commutes with colimits on both sides along with Fubini theorem for coends and Yoneda lemma we see that the last coend is isomorphic to $X_1(n)\otimes\ldots\otimes X_k(n).$ Hence, the Day-Street convolution is equal to the pointwise tensor product of cosimplicial objects. Commutative monoids with respect to this multitensor structure are exactly commutative cosimplicial  monoids in $\bV.$ 
\end{proof}

More explicitly, writing (by Joyal duality) a functor $f\in \Cat_{*,*}(\langle n+1\rangle,\langle n_1 +1\rangle\times\cdots\times\langle n_k +1\rangle)$ as a collection of functors $f_i \in \Cat_{*,*}(\langle n+1\rangle ,\langle n_i +1\rangle) \simeq \De([n_i],[n])$ one can see that an element of $\MM(n_1,\ldots, n_k; n)$ corresponds to the following operation on the underlying collection of a  commutative cosimplicial monoid $E:\De\to \bMon$
$$\xymatrix{E(n_1)\otimes\ldots\otimes E(n_k) \ar[rrr]^{E(f_1)\otimes\ldots\otimes E(f_k)} &&&  E(n)\otimes\ldots\otimes E(n) \ar[r]^(.67){\mu_n} &  E(n)}$$
where $E(f_i):E(n_i)\to E(n)$ is the cosimplicial map corresponding to $f_i$ and $\mu_n$ is the product of the monoid $E(n)$. 
\\

Let now $\MM^{(0)} = \MM\times \A ss$ be the product in the category of symmetric colored operads, where $\A ss$ is the one colored $\bSet$-operad for monoids. 
By definition
$$\MM^{(0)}(n_1,\ldots, n_k; n)=\Cat_{*,*}(\langle n+1\rangle ,\langle n_1 +1\rangle \times\cdots\times\langle n_k +1\rangle )\times \Sigma_k ,$$
and the operadic composition is induced by the operadic composition in $\MM$ by the first variable and the operadic composition on symmetric groups $\Sigma_k$ in the second variable. 

\begin{thm}\label{bvass} The category of algebras of $\MM^{(0)}$ in any cocomplete symmetric monoidal category $\bV$ is isomorphic to the category of cosimpicial  monoids in $\bV.$ The natural projection $\MM^{(0)}\to \MM$ is an operadic morphism which induces the forgetful functor from commutative cosimplicial  monoids to cosimplicial  monoids.
\end{thm}
\begin{proof} It is not hard to see that Day-Street convolution for $\MM^{(0)}$ is the symmetrisation of Day-Street convolution for $\MM$ 
(see \cite[p.61]{DS2}). 
So, the result follows.
\end{proof}

For example, if $E$ is a cosimplicial monoid in abelian groups (that is a cosimplicial ring) then an element  $\alpha = (\phi,\sigma) = (f_1,...,f_k;\sigma) \in\MM^{(0)}(n_1,\ldots, n_k; n)$ corresponds to the following operation:  $$\alpha(-):E(n_1)\otimes\ldots\otimes E(n_k)\to E(n),$$
\begin{equation}\label{action} \alpha(x_1\otimes \ldots\otimes x_k)= E(f_{\sigma^{}(1)})(x_{\sigma^{}(1)})\cdot\ldots\cdot E(f_{\sigma^{}(k)})(x_{\sigma^{}(k)})\  , \end{equation}
where $\cdot$ is the multiplication in $E(n).$

\begin{remark} It is clear from the description of algebras of $\MM$ and $\MM^{(0)}$ that
$$\MM = \Delta \otimes_{BV} \C om\qquad\qquad
\MM^{0} = \Delta \otimes_{BV} \A ss$$ 
and the map $q: \MM^{(0)}\to \MM$ is $1\otimes_{BV} \eta$ where $\eta: \A ss\to \C om$ is the canonical morphism of operads.
\end{remark}
Here 
  $\otimes_{BV} $ is the Boardman-Vogt tensor product \cite{bv}. In fact, Theorems \ref{bvcom} and \ref{bvass} admit an obvious generalisation for computing $C\otimes_{BV} \C om $ and $C\otimes_{BV} \A  ss$ for an arbitrary small category $C,$ considered as an operad with unary operations only.

\subsection{Linking numbers and $n$-commutative cosimplicial  monoids}

Let  $\tau:[p]\to [m],\ \pi:[q]\rightarrow [m]$ be two maps in $\De.$ {\it A shuffling of $\tau,\pi$ of length $n$} is 
a pair of decompositions of their images into disjoint unions:
$$Im(\tau) = A_{1}\cup ...\cup A_{s}, \quad A_{1}<...<A_{s}$$
$$Im(\pi) = B_{1}\cup ...\cup B_{t}, \quad  B_{1}<...<B_{t}, \ s+t = n+1$$
which satisfy one of the inequalities (of ordered sets)
$$A_{1}\leq B_{1}\leq A_{2}\leq B_{2} \leq... \  ,$$
or 
$$ B_{1}\leq A_{1}\leq B_{2}\leq A_{2} \leq...  \  .$$
A choice of one of the inequalities is a part of the shuffling. 
\nl
Note also that in general $|s-t|\le 1$,  so the last term in any decomposition is either $B_t$ or $A_s.$

\begin{example}\label{id} Let $\tau = \pi = id:[2]\to [2].$ Here are all shufflings of this pair.  
\begin{enumerate} 
\item 
$\quad A_1= \{0\},\quad A_2=\{1\},\quad  A_3=\{2\},\qquad\qquad B_1= \{0\},\quad B_2=\{1\},\quad  B_3=\{2\},$
$$A_{1}\ \leq\ B_{1}\ \leq\ A_{2}\ \leq\ B_{2}\ \leq\ A_3\ \leq\ B_3\ .$$

\item 
$\quad A_1= \{0\},\quad A_2=\{1\},\quad  A_3=\{2\},\qquad\qquad B_1= \{0\},\quad B_2=\{1\},\quad  B_3=\{2\},$
$$B_{1}\ \leq\  A_{1}\ \leq\  B_{2}\ \leq\  A_{2} \ \leq\  B_3\ \leq\  A_3\ .$$

\item 
$\quad A_1= \{0\},\quad A_2=\{1\},\quad  A_3=\{2\},\qquad\qquad B_1= \{0,1\},\quad  B_2=\{2\},$ 
$$A_{1}\ \leq\  B_{1}\ \leq\  A_{2}\ \leq\  B_{2} \ \leq\  A_3\ .$$

\item 
$\quad A_1= \{0,1\},\quad  A_3=\{2\},\qquad\qquad B_1= \{0\},\quad B_2=\{1\},\quad  B_3=\{2\},$
$$B_{1}\ \leq\  A_{1}\ \leq\  B_{2}\ \leq\  A_{2} \ \leq\  B_3\ \ .$$

\item 
$\quad A_1= \{0\},\quad A_2=\{1\},\quad  A_3=\{2\},\qquad\qquad B_1= \{0\},\quad B_2=\{1,2\},$ 
$$A_{1}\ \leq\  B_{1}\ \leq\  A_{2}\ \leq\  B_{2} \ \leq\  A_3\ .$$

\item 
$\quad A_1= \{0\},\quad A_2=\{1,2\},\qquad\qquad B_1= \{0\},\quad B_2=\{1\},\quad  B_3=\{2\},$
$$B_{1}\ \leq\  A_{1}\ \leq\  B_{2}\ \leq\  A_{2} \ \leq\  B_3\ \ .$$

\item 
 $\quad A_1= \{0\},\quad A_2=\{1,2\}, \qquad\qquad B_1= \{0,1\},\quad B_2=\{2\}\ ,$ 
$$A_{1}\ \leq\  B_{1}\ \leq\  A_{2}\ \leq\  B_{2}.$$

\item  
$\quad A_1= \{0,1\},\quad A_2=\{2\}, \qquad\qquad B_1= \{0\},\quad B_2=\{1,2\}\ ,$
$$B_{1}\ \leq\  A_{1}\ \leq\  B_{2}\ \leq\  A_{2}\ .$$
\end{enumerate} 

\end{example}

\begin{defin} 
We say that the {\em linking number} $\link(\tau,\pi)$ of two nondecreasing maps $\tau:[p]\to [m],\ \pi:[q]\rightarrow [m]$ is equal to $n$, if $n$ is the smallest number for which there exists a shuffling of $\tau$ and $\pi$ of length $n.$
\end{defin}

Note that  for any  shuffling of $\tau$ and $\pi$ the cardinality of the intersection of their images is bounded as follows 
$|Im(\tau)\cap Im(\pi)|\ \leq s+t-1$ and, hence, 
$$|Im(\tau)\cap Im(\pi)|\ \leq \link(\tau,\pi)\ .$$  

Observe also, that the linking number depends only on the images of the morphisms $\pi$ and $\tau.$ 
That is we have the following. 
\begin{lem}\label{epimono} Let $\tau:[p]\to [m],\ \pi:[q]\rightarrow [m]$ be two maps in $\De$ and let
$$[p]\to [p']\stackrel{\tau'}{\to} [m],\qquad [q]\rightarrow [q'] \stackrel{\pi'}{\to}[m]$$
be their respective epi-mono factorisations. Then
$$\link(\tau,\pi) = \link(\tau',\pi').$$
\end{lem}

\bex\lb{elno}
The monomorphism $\tau_{m,n}:[n]\to [m+n]$, which does not take the values $n+1,...,n+m$ and the monomorphism
$\pi_{m,n}:[m]\rightarrow [m+n]$, which does not take the values $0,...,m-1$ have the linking number one. Graphically $\tau_{m,n}$ and $\pi_{m,n}$ are
$$\xygraph{
!{<0cm,0cm>;<.7cm,0cm>:<0cm,1cm>::}
!{(-1,2)}*{\circ} ="0" !{(0,2)}*{...} !{(1,2)}*{\circ} ="1"  
!{(-2.5,0)}*{\circ}="d0" !{(-1.5,0)}*{...} !{(-.5,0)}*{\circ} ="d1" !{(0.5,0)}*{\circ}  !{(1.5,0)}*{...}  !{(2.5,0)}*{\circ}
!{(-1,2.25)}*+{\scriptstyle 0}  !{(1,2.25)}*+{\scriptstyle n}   
!{(-2.5,-.23)}*+{\scriptstyle 0}  !{(-.5,-.23)}*+{\scriptstyle n}  !{(.5,-.23)}*+{\scriptstyle n+1}  !{(2.5,-.23)}*+{\scriptstyle n+m} 
"0":"d0"     "1":"d1"  
}\qquad\qquad\qquad
\xygraph{
!{<0cm,0cm>;<.7cm,0cm>:<0cm,1cm>::}
!{(-1,2)}*{\circ} ="0" !{(0,2)}*{...} !{(1,2)}*{\circ} ="1"  
!{(-2.5,0)}*{\circ} !{(-1.5,0)}*{...} !{(-.5,0)}*{\circ}  !{(0.5,0)}*{\circ}="d0"  !{(1.5,0)}*{...}  !{(2.5,0)}*{\circ}="d1"
!{(-1,2.25)}*+{\scriptstyle 0}  !{(1,2.25)}*+{\scriptstyle m}   
!{(-2.5,-.23)}*+{\scriptstyle 0}  !{(-.5,-.23)}*+{\scriptstyle n-1}  !{(.5,-.23)}*+{\scriptstyle n}  !{(2.5,-.23)}*+{\scriptstyle n+m} 
"0":"d0"     "1":"d1"  
}$$
correspondingly. 
\nl
Note that $\tau_{m,n}$ can be written as the composite $\partial^{n+m}_{n+m-1}...\partial^{n+2}_{n+1}\partial^{n+1}_n$, while $\pi_{m,n}$ coincides with $\partial^{n-1}_{m+n-1}...\partial^1_{m+1}\partial^0_m$.
\eex

\bex\lb{elnt}
The monomorphism $\tau_{m,n}^i:[n]\to [m+n-1]$, which does not take the values $i+1,...,i+m-1$ and the monomorphism
$\pi_{m,n}^i:[m]\rightarrow [m+n-1]$, which does not take the values $0,...,i-1$ and $i+m+1,...,m+n-1$ have the linking number two. Graphically $\tau_{m,n}^i$ and $\pi_{m,n}^i$ are
$$\xygraph{
!{<0cm,0cm>;<.7cm,0cm>:<0cm,1cm>::}
!{(-2,2)}*{\circ} ="0" !{(-1,2)}*{...} !{(0,2)}*{\circ} ="1"  !{(1,2)}*{\circ}="2"  !{(2,2)}*{...}  !{(3,2)}*{\circ}="3"
!{(-2.5,0)}*{\circ}="d0" !{(-1.5,0)}*{...} !{(-.5,0)}*{\circ} ="d1" !{(.5,0)}*{...}  !{(1.5,0)}*{\circ}="d2"  !{(2.5,0)}*{...}  !{(3.5,0)}*{\circ}="d3"
!{(-2,2.25)}*+{\scriptstyle 0}  !{(0,2.25)}*+{\scriptstyle i}  !{(1,2.25)}*+{\scriptstyle i+1}   !{(3,2.25)}*+{\scriptstyle n} 
!{(-2.5,-.23)}*+{\scriptstyle 0}  !{(-.5,-.23)}*+{\scriptstyle i}     !{(1.5,-.23)}*+{\scriptstyle i+m}  !{(3.5,-.23)}*+{\scriptstyle n+m-1} 
"0":"d0"     "1":"d1"  "2":"d2"  "3":"d3"
}\qquad\qquad\qquad\quad
\xygraph{
!{<0cm,0cm>;<.7cm,0cm>:<0cm,1cm>::}
!{(-.5,2)}*{\circ} ="1"  !{(.5,2)}*{...}  !{(1.5,2)}*{\circ}="2"  
!{(-2.5,0)}*{\circ} !{(-1.5,0)}*{...} !{(-.5,0)}*{\circ} ="d1" !{(.5,0)}*{...}  !{(1.5,0)}*{\circ}="d2"  !{(2.5,0)}*{...}  !{(3.5,0)}*{\circ}
!{(-.5,2.25)}*+{\scriptstyle 0}  !{(1.5,2.25)}*+{\scriptstyle m}   
!{(-2.5,-.23)}*+{\scriptstyle 0}  !{(-.5,-.23)}*+{\scriptstyle i}     !{(1.5,-.23)}*+{\scriptstyle i+m}  !{(3.5,-.23)}*+{\scriptstyle n+m-1} 
"1":"d1"  "2":"d2"
}$$
correspondingly. 
\nl
Note that $\tau_{m,n}^i$ can be written as the composition $\partial^{i+m-1}_{n+m-2}...\partial^{i+2}_{n+1}\partial^{i+1}_n$, while $\pi_{m,n}^i$ coincides with $\partial^{n+m-1}_{n+m-2}...\partial^{i+m+2}_{n+i+1}\partial^{i+m+1}_{n+i}\partial^{i-1}_{n+i-1}...\partial^1_{n+1}\partial^0_n$.
\eex

\begin{example}\label{idid} In the Example \ref{id} the linking number $\link(id,id) = 3$. Similarly one sees that 
for $id:[n]\to [n]$ the linking number $\link(id,id)$ is $n+1$ so there exist pairs of maps with any linking number greater or equal to $1.$  
\end{example}

The above discussion on linking numbers allows us to define a sequence of notions intermediate between cosimplicial monoids and commutative cosimplicial  monoids. 
\begin{defin}\lb{dnc}
Let $n\ge 1.$ We call a cosimplicial monoid $E$ in a symmetric monoidal category $V$ {\em $n$-commutative} if 
for any morphisms $\tau:[p]\to [n],\ \pi:[q]\to [n]$ in $\De$ with $\link(\tau,\pi) \le n$
the  diagram

\begin{equation}\lb{ncom}
    \xymatrix@C = +4em{
     E(p)\otimes E(q) \ar[rr]^{E(\pi)\otimes E(\tau)} \ar[d]_{c_{E(p), E(q)}} & &   E(n)\otimes E(n) \ar[d]^{\mu}
      \\
  E(q)\otimes E(p)\ar[r]^{E(\tau)\otimes E(\pi)}& E(n)\otimes E(n) \ar[r]^{\mu_{}} & E(n)
}    
\end{equation}
 commutes.   Here $c_{E(p), E(q)}$ is the braiding in $\bV$ and $\mu$ is the product of the monoid $E(n)$ .
\end{defin}

For convenience, we also call an arbitrary cosimplicial monoid without any commutativity requirement $0$-commutative.

\begin{remark} It follows from this definition and the Example \ref{idid} that in an $n$-commutative cosimplicial monoid the components $E(m)$ for $m< n$ are commutative monoids. 
\end{remark}

\begin{remark} In \cite{de} a slightly different definition of $n$-commutative cosimplicial complexes of monoids is given. 
\end{remark}

\subsection{Operad for $n$-commutative cosimplicial monoids}

We are going to construct a tower of operads $\MM^{(0)}\to  \MM^{(1)}\to \MM^{(2)}\to \ldots \to\MM,$ where $\MM^{(n)}$ has $n$-commutative cosimplicial monoids as its category of algebras. 

Firstly  we would like to reformulate the linking number of two morphisms in $\Delta$ in terms of a property of operations in the paths operad $\MM.$  

Let $f:\langle m+1\rangle \to\langle p+1\rangle $ be a functor.
 We say that the generator $\bar{i}:i\to i+1$ in $\langle m+1\rangle $ is in the support set $supp(f)$ if $f(\bar{i})\ne id.$ 
Let  $\phi:\langle m+1\rangle \to\langle p+1\rangle \times\langle q+1\rangle $ be a path in $\MM.$ Composing with the corresponding projections we have then two functors 
 $\phi_1:\langle m+1\rangle \to\langle p+1\rangle $ and $\phi_2:\langle m+1\rangle \to\langle q+1\rangle $ in $\mathbf{Int}$ with the supports $A$ and $B$ correspondingly. 
 \begin{defin}
A {\em shuffling} of $\phi$ of length $n$ is a decomposition of the supports sets $A$ and $B$  into disjoint unions of nonempty sets
$$A = A_{1}\cup ...\cup A_{s}, \quad A_{1}<...<A_{s}$$
$$B = B_{1}\cup ...\cup B_{t}, \quad  B_{1}<...<B_{t}, $$ $$ s+t = n+1 , \ |s-t|\le 1,$$
which satisfies one of the inequalities \begin{equation}\label{first}A_{1}\leq B_{1}\leq A_{2}\leq B_{2} \leq... \  ,\end{equation}
\begin{equation} B_{1}\leq A_{1}\leq B_{2}\leq A_{2} \leq...  \  .\end{equation}
A choice of one of the inequalities is a part of the shuffling. 

\end{defin}

\begin{defin} The {\em linking number} $\link(\phi)$ of a path $\phi$ is the smallest $n$ for which there exists a shuffling of $\phi$ of length $n.$  

\end{defin}

As in  case of linking numbers of cosimplicial operators we have 
$$|supp(\phi_1)\cap supp(\phi_2)|\ \leq \link(\phi)\ .$$ 

Immediately from Joyal's duality we have the following.
\begin{lem}\label{link=link} Let $\tau :[p]\to [m]$ and $\pi:[q]\to [m]$ in $\De$  and let $\widetilde\tau:\langle m+1\rangle \to\langle p+1\rangle $  and   $\widetilde\pi:\langle m+1\rangle \to\langle q+1\rangle $   be their Joyal's dual functors.
Let also $\widetilde{\tau\pi}:\langle m+1\rangle\to \langle p+1\rangle \times\langle q+1\rangle $ be the composite 
$$ \xymatrix{\langle m+1\rangle\ar[r]^(.34)\delta & \langle m+1\rangle\times \langle m+1\rangle \ar[r]^{(\widetilde\tau,\widetilde\pi)} & \langle p+1\rangle\times \langle q+1\rangle}.$$ 
Then
$$\link(\tau,\pi)=\link(\widetilde{\tau\pi}).$$

\end{lem}

We realise $\MM^{(n)}$ as a  quotient of $\MM^{(0)}.$ For this we introduce a relation on  $\MM^{(0)}(p,q;k).$   
 Let $(\phi,\sigma)$ be an element of $\MM^{(0)}(p,q;k)$:
$$ \phi:\langle k+1\rangle \to \langle p+1\rangle \times \langle q+1\rangle  \ \ \mbox{and} \ \     \sigma \in \Sigma_2 = \{e,t\}.$$
 We  say that $(\phi,e)$ is {\em $n$-equivalent} to $(\phi,t)$ if 
 $\link(\phi) \le n.$

 Let $\MM^{(n)}$ be the quotient of $\MM^{(0)}$ by the equivalence relation generated by $n$-equivalence relation. More precisely, $\MM^{(n)}$ is the following pushout in the category of $\mathbb{N}$-coloured operads in $\Set$ :    

 \begin{equation*}
    \xymatrix@C = +4em{
    FU(\MM^{(0)})  \ar[r]^{F(b)} \ar[d]&    F(B) \ar[d]^{}
      \\
       \MM^{(0)}  \ar[r]^{q^{(n)}} & \MM^{(n)}
}    
\end{equation*}
where $F$ is the free operad functor on a collection, $U$ is the forgetful functor, $B$ is the quotient of the collection $U(\MM^{(0)})$ by the equivalence relation generated by $n$-equivalence relation, and $b:U(\MM^{(0)}) \to B$ is the quotient map.

 \begin{thm}
The category of $n$-commutative cosimplicial monoids is equivalent to the category of $\MM^{(n)}$-algebras.  
 \end{thm} 
 \begin{proof}Obvious from construction. 
 \end{proof}

\begin{remark} The description of $\MM^{(n)}$ above is not explicit and we don't know if there is a better description of this operad for $0< n < \infty.$  \end{remark}

Let  $i<j,$ $\phi$   be a  path and  let $\phi_{ij}$ be the composite
$$\langle n+1\rangle \stackrel{\phi}{\to} \langle n_1 +1\rangle\times\cdots\times \langle n_k +1\rangle \stackrel{proj_{ij}}{\longrightarrow} \langle n_i +1\rangle\times \langle n_j +1\rangle.$$ 

The following Lemma will be useful for us in understanding of $\MM^{(n)}$-action.

 \begin{lem}\label{eqshuffle}  Let $(\phi,\sigma)$   be an element of $\MM^{(0)}.$ And let $i,j$ be two consecutive numbers in   $\sigma = (\ldots, i, j , \ldots),$ such that $\link(\phi_{ij})\le  n .$ Then $(\phi  ,\sigma)$ is $n$-equivalent to the $(\phi,\sigma')$ where $\sigma'$ is obtained from $\sigma$ by permuting $i$ and $j.$
 \end{lem}
 
 \begin{proof} It is enough to prove that $(\phi,\sigma)$ and $(\phi,\sigma')$ produce the same action on an $n$-commutative  cosimplicial monoid in $\Set$. But this easily follows from the formula for this action (\ref{action}).  
\end{proof}


The following simple result relates the paths operad and the lattice paths operad of Batanin and Berger \cite{BB}. In  the following sections we will generalise this theorem to all $n>0.$  The reader is referred to the Appendix A for main facts about the lattice path operads.

\bth\label{L^1M^0}
There are operadic morphisms $p:\LL\to\MM$ and $p^{(1)}: \LL^{(1)}\to\MM^{(0)}$ which make the following square commutative:

\begin{equation}\lb{cso}
    \xymatrix@C = +4em{
      \LL^{(1)} \ar[r] \ar[d]& \MM^{(0)} \ar[d]^{}
      \\
    \LL \ar[r]^{} & \MM
}    
\end{equation}
\eth

\begin{proof}

The morphism $p$ is induced by the natural symmetric monoidal transformation $-\boxx- \to - \times -$ from funny tensor product to cartesian product in $\Cat.$ 

The morphism of operads
$$p^{(1)}: \LL^{(1)}(n_1,\ldots,n_k;n)\to \MM^{}(n_1,\ldots,n_k;n)\times \Sigma_k $$
is determined by two projections
$$\LL^{(1)}(n_1,\ldots,n_k;n)\to   \LL^{}(n_1,\ldots,n_k;n)\to \MM^{}(n_1,\ldots,n_k;n)$$ and 
$$\varpi^{(1)}:\LL^{(1)}(n_1,\ldots,n_k;n)\to  \Sigma_k$$
It is  a morphism of operads due to Lemma \ref{LAss} and it  obviously fits in  the commutative square \eqref{cso}.
\end{proof}

\subsection{Lifting of paths, complexity and linking numbers}

To generalise Theorem \ref{L^1M^0} to all filtration we will first reformulate the definition of linking numbers in terms of liftings of paths to lattice paths of corresponding complexity. 

\begin{defin} Let $\phi:\langle m+1\rangle \to\langle p+1\rangle \times\langle q+1\rangle $ be a path in $\MM.$ We will say that a lattice path $\psi:\langle m+1\rangle \to\langle p+1\rangle \boxx\langle q+1\rangle $ is a lifting of $\phi$ if its composite with the canonical projection to $\langle p+1\rangle \times\langle q+1\rangle $ is  $\phi.$

\end{defin}

\begin{lem} Let $\phi:\langle m+1\rangle \to\langle p+1\rangle \times\langle q+1\rangle $ be a path.
There is a one to one correspondence between shufflings of  $\phi$ and liftings of $\phi.$ Under this correspondence the shufflings of length $n$ correspond to liftings of complexity $n$ and vice versa.   

\end{lem}
\begin{proof}  
Given a fixed shuffling we construct a lifting out of it. Without loss of generality we assume that our shuffling is such that the  inequality (\ref{first}) for decomposition is satisfied.
We then define 
$$\psi(\bar{i}) = \left\{
\begin{array}{lcccc}
(\phi_1(\bar{i}),id)\circ (id,\phi_2(\bar{i}))& , & \bar{i}\in A & \mbox{and} &  \bar{i}\notin B \\
(id,\phi_2(\bar{i})) \circ (\phi_1(\bar{i}),id)& , & \bar{i}\notin A & \mbox{and} & \bar{i}\in B \\
(id,id)\circ (id,id)& , & \bar{i}\notin A & \mbox{and}  & \bar{i}\notin B \\
(\phi_1(\bar{i}),id)\circ (id,\phi_2(\bar{i}))& , & \bar{i}\in A_k & \mbox{and}  & \bar{i}\in B_k \\
(id,\phi_2(\bar{i})) \circ (\phi_1(\bar{i}),id)& , & \bar{i}\in A_{k+1} & \mbox{and}  & \bar{i}\in B_k
\end{array}
\right.$$

In reverse direction, given a lifting $\psi:\langle m+1\rangle \to\langle p+1\rangle \boxx\langle q+1\rangle $ of $\phi$ we define a shuffling as follows.
Two generating morphisms $\bar{i}$ and $\overline{i+j}$  in  $supp(\phi_1) =A$ belongs to the same element in decomposition of $A = A_{1}\cup ...\cup A_{s}$ 
if the composite 
$$\psi(i)\stackrel{\psi(\bar{i})}{\longrightarrow} \psi(i+1) \stackrel{\psi(\overline{i+1})}{\longrightarrow}\ldots \stackrel{\psi(\overline{i+j})}{\longrightarrow} \psi(i+j+1)$$
has no more then two corners in $\langle p+1 \rangle \boxx \langle q+1 \rangle.$   
Similarly for the set $B=supp(\phi_2).$
\end{proof}

\begin{example} Following the Example \ref{id} let $\pi = \tau = id:[2]\to [2].$ The corresponding path $\phi$ is the diagonal path 
$\delta:\langle 3\rangle \to \langle 3\rangle\times \langle 3 \rangle $  as in the following picture:

$$\xygraph{
!{<0cm,0cm>;<1cm,0cm>:<0cm,1cm>::}
!{(-1,2)}*{\circ} ="uul" !{(0,2)}*{\circ} ="uu"  !{(1,2)}*{\circ}="uur"  !{(2,2)}*{\circ}="urr"
!{(-1,1)}*{\circ} ="ul" !{(0,1)}*{\circ} ="u"  !{(1,1)}*{\circ}="ur"  !{(2,1)}*{\circ}="rur"
!{(-1,0)}*{\circ} ="l" !{(0,0)}*{\circ} ="m"  !{(1,0)}*{\circ}="r"  !{(2,0)}*{\circ}="rr"
!{(-1,-1)}*{\circ} ="dl" !{(0,-1)}*{\circ} ="d"  !{(1,-1)}*{\circ}="dr"  !{(2,-1)}*{\circ}="drr"
!{(-1.15,-.85)}*+{\scriptstyle 1}  !{(-.15,.15)}*+{\scriptstyle 1} !{(.85,1.15)}*+{\scriptstyle 1}  !{(1.85,2.15)}*+{\scriptstyle 1}
"dl":"m"     "m":"ur"  "ur":"urr"
}$$

Let us choose the shuffling of length $5$ of this path as in part 1 of Example \ref{id}:
$$A_1= \{\bar{0}\},A_2=\{\bar{1}\},  A_3=\{\bar{2}\},$$
$$B_1= \{\bar{0}\},B_2=\{\bar{1}\},  B_3=\{\bar{2}\},$$
with 
$$A_{1}\leq B_{1}\leq A_{2}\leq B_{2} \leq A_3\leq B_3.$$

Then the corresponding lifting has complexity $5$ and is presented graphically as follows: 

$$\xygraph{
!{<0cm,0cm>;<1cm,0cm>:<0cm,1cm>::}
!{(-1,2)}*{\circ} ="uul" !{(0,2)}*{\circ} ="uu"  !{(1,2)}*{\circ}="uur"  !{(2,2)}*{\circ}="urr"
!{(-1,1)}*{\circ} ="ul" !{(0,1)}*{\circ} ="u"  !{(1,1)}*{\circ}="ur"  !{(2,1)}*{\circ}="rur"
!{(-1,0)}*{\circ} ="l" !{(0,0)}*{\circ} ="m"  !{(1,0)}*{\circ}="r"  !{(2,0)}*{\circ}="rr"
!{(-1,-1)}*{\circ} ="dl" !{(0,-1)}*{\circ} ="d"  !{(1,-1)}*{\circ}="dr"  !{(2,-1)}*{\circ}="drr"
!{(-1,-.78)}*+{\scriptstyle 1}  !{(-.16,.16)}*+{\scriptstyle 1} !{(1,.78)}*+{\scriptstyle 1}  !{(.85,1.15)}*+{\scriptstyle 1}  !{(2.2,2)}*+{\scriptstyle 1}
!{(1.15,-.15)}*+{\scriptstyle 0}  !{(.15,-1.15)}*+{\scriptstyle 0}  !{(2.15,.85)}*+{\scriptstyle 0}
"dl":"d"     "d":"m" 
"m":"r"     "r":"ur"  "ur":"rur" "rur":"urr"
}$$
For the shuffling from the part 8 of Example \ref{id} the corresponding lifting is  of complexity $3$:

$$\xygraph{
!{<0cm,0cm>;<1cm,0cm>:<0cm,1cm>::}
!{(-1,2)}*{\circ} ="uul" !{(0,2)}*{\circ} ="uu"  !{(1,2)}*{\circ}="uur"  !{(2,2)}*{\circ}="urr"
!{(-1,1)}*{\circ} ="ul" !{(0,1)}*{\circ} ="u"  !{(1,1)}*{\circ}="ur"  !{(2,1)}*{\circ}="rur"
!{(-1,0)}*{\circ} ="l" !{(0,0)}*{\circ} ="m"  !{(1,0)}*{\circ}="r"  !{(2,0)}*{\circ}="rr"
!{(-1,-1)}*{\circ} ="dl" !{(0,-1)}*{\circ} ="d"  !{(1,-1)}*{\circ}="dr"  !{(2,-1)}*{\circ}="drr"
!{(-1.2,-1)}*+{\scriptstyle 1}  !{(0,.2)}*+{\scriptstyle 1} !{(1.2,1)}*+{\scriptstyle 1}  !{(2,2.2)}*+{\scriptstyle 1}
!{(-1.15,.15)}*+{\scriptstyle 0}  !{(1.15,-.15)}*+{\scriptstyle 0}  !{(.85,2.15)}*+{\scriptstyle 0}
"dl":"l"     "l":"m" 
"m":"r"     "r":"ur"  "ur":"uur"  "uur":"urr"
}$$

\end{example}

\subsection{Delannoy paths}

A path $\phi:\langle n+1\rangle \to\langle p+1\rangle \times\langle q+1\rangle $ is called {\it sujective} if for any generator $\bar{i}$ the arrow $\phi(\bar{i}) = (t,s)$ where $t$ and $s$ are either one of the generators $\bar{j}$ or an identity. In other words this is a path whose each movement adds $0$ or $1$ to each coordinate.  

\begin{lem}\label{cont} For each path $\phi:\langle n+1\rangle \to\langle p+1\rangle \times\langle q+1\rangle $ there is a canonical factorisation 
$$\langle n+1\rangle \stackrel{\phi'}{\to}\langle p'+1\rangle \times\langle q'+1\rangle \to\langle p+1\rangle \times\langle q+1\rangle $$ such that 
$\phi'$ is surjective and $\link(\phi') = \link(\phi).$
\end{lem} 
\begin{proof} We can take Joyal's dual to the projections $\phi_1$ and $\phi_2$ then take their epi-mono factorisations and product of their Joyal's duals again. The result now follows from Lemmas  \ref{epimono} and \ref{link=link}.
\end{proof}
\begin{defin}
A surjective path is {\it sharp} if 
 $\phi(\bar{i}) \ne (\bar{s},\bar{t})$ for all $i,t,s.$
 \end{defin}

A path $\phi$ is called {\it injective} if $\phi(\bar{i}) \ne (id,id)$ for any $0\le i \le n.$  That is, such path has only $1$ or $0$ stops,     
\begin{lem}\label{nond} For each path $\phi:\langle n+1\rangle \to\langle p+1\rangle \times\langle q+1\rangle $ there is a 
canonical factorisation
$$ \langle n+1\rangle {\to}\langle n'+1\rangle \stackrel{\phi'}{\to}\langle p+1\rangle \times\langle q+1\rangle $$ such that $\phi'$ is injective and 
$\link(\phi') = \link(\phi).$
\end{lem} 
\begin{proof}
Indeed, if $\phi(\bar{i}) = (id,id)$ we then send it to $id$ in $\langle n'+1\rangle .$ In other words we construct a shorter path with exactly the same stopping points where we `cut off' all nontrivial loops in $\phi.$ Since generators $\bar{i}$ with $\phi(\bar{i}) = (id,id)$ do not contribute to the supports sets of $\phi_1$ and $\phi_2$ the linking number does not change. 
\end{proof}

\begin{defin}[\cite{BS}] A surjective and injective path is called a {\em Delannoy path}. \end{defin} 


A Delannoy path  $\phi:\langle n+1\rangle \to\langle p+1\rangle \times\langle q+1\rangle $ {\it has a low corner} at the point $(s+1,t)$ if there exists $0\le i\le n$ such that
$\phi(\bar{i}) = (\bar{s},id)$ and $\phi(\overline{i+1}) = (id,\bar{t}).$    It has an {\it upper corner} at $(t+1,s)$ if   there exists $0\le i\le n$ such that
$\phi(\bar{i}) = (id,\bar{t})$ and $\phi(\overline{i+1}) = (\bar{s},id).$ A Dellannoy path which does not have low or upper corners is called {\it smooth}.   
That is, smooth paths are `opposite' to sharp paths which have all possible corners. 
\begin{example}
This is an example of a sharp Dellannoy path (on the left) and a smooth path (on the right). 
$$\xygraph{
!{<0cm,0cm>;<0.6cm,0cm>:<0cm,0.6cm>::}
!{(-1,3)}*{\circ}  !{(0,3)}*{\circ}   !{(1,3)}*{\circ}  !{(2,3)}*{\circ} !{(3,3)}*{\circ}   !{(4,3)}*{\circ}  !{(5,3)}*{\bullet}="11"
!{(-1,2)}*{\circ}  !{(0,2)}*{\circ}   !{(1,2)}*{\circ}  !{(2,2)}*{\circ} !{(3,2)}*{\circ}   !{(4,2)}*{\circ}  !{(5,2)}*{\bullet}="10"
!{(-1,1)}*{\circ}  !{(0,1)}*{\circ}  !{(1,1)}*{\circ}  !{(2,1)}*{\bullet}="6" !{(3,1)}*{\bullet}="7"  !{(4,1)}*{\bullet}="8"  !{(5,1)}*{\bullet}="9"
!{(-1,0)}*{\circ}  !{(0,0)}*{\circ}   !{(1,0)}*{\circ}  !{(2,0)}*{\bullet}="5"  !{(3,0)}*{\circ}   !{(4,0)}*{\circ}  !{(5,0)}*{\circ}
!{(-1,-1)}*{\bullet} ="1" !{(0,-1)}*{\bullet}="2"   !{(1,-1)}*{\bullet}="3"  !{(2,-1)}*{\bullet}="4" !{(3,-1)}*{\circ}   !{(4,-1)}*{\circ}  !{(5,-1)}*{\circ}
"1":"2"     "2":"3"  "3":"4"  "4":"5" "5":"6"  "6":"7"  "7":"8"  "8":"9"  "9":"10"  "10":"11"
}\qquad\qquad\qquad\qquad
\xygraph{
!{<0cm,0cm>;<0.6cm,0cm>:<0cm,0.6cm>::}
!{(-1,3)}*{\circ}  !{(0,3)}*{\circ}   !{(1,3)}*{\circ}  !{(2,3)}*{\circ} !{(3,3)}*{\circ}   !{(4,3)}*{\circ}  !{(5,3)}*{\bullet}="8"
!{(-1,2)}*{\circ}  !{(0,2)}*{\circ}   !{(1,2)}*{\circ}  !{(2,2)}*{\circ} !{(3,2)}*{\circ}   !{(4,2)}*{\circ}  !{(5,2)}*{\bullet}="7"
!{(-1,1)}*{\circ}  !{(0,1)}*{\circ}  !{(1,1)}*{\circ}  !{(2,1)}*{\circ} !{(3,1)}*{\bullet}="5"  !{(4,1)}*{\bullet}="6"  !{(5,1)}*{\circ}
!{(-1,0)}*{\circ}  !{(0,0)}*{\circ}   !{(1,0)}*{\circ}  !{(2,0)}*{\bullet}="4"  !{(3,0)}*{\circ}   !{(4,0)}*{\circ}  !{(5,0)}*{\circ}
!{(-1,-1)}*{\bullet} ="1" !{(0,-1)}*{\bullet}="2"   !{(1,-1)}*{\bullet}="3"  !{(2,-1)}*{\circ} !{(3,-1)}*{\circ}   !{(4,-1)}*{\circ}  !{(5,-1)}*{\circ}
"1":"2"     "2":"3"  "3":"4"  "4":"5" "5":"6"  "6":"7"  "7":"8" 
}$$
\end{example}

\begin{lem}\label{sharplink}
Let $\phi:\langle n+1\rangle \to\langle p+1\rangle \times\langle q+1\rangle $ be a sharp path. Then 
\begin{enumerate} \item $supp(\phi_1)\cap supp(\phi_2)\ = \emptyset $  \item $supp(\phi_1)\cup supp(\phi_2)\ = \{\bar{0},\ldots,\bar{n}\}$ if $\phi$ is Delannoy;   \item There is a unique lifting $\psi$ of $\phi$ 
\item The linking number of $\phi$ is equal to complexity of $\psi.$ 
\end{enumerate}
\end{lem}

\begin{proof} Obvious from definitions.

\end{proof}

\begin{lem}\label{sharpening} Let   $\phi:\langle n+1\rangle \to\langle p+1\rangle \times\langle q+1\rangle $ be a Dellanoy path.
Then there exists a factorisation 
$$ \langle n+1\rangle {\to}\langle n'+1\rangle \stackrel{\phi'}{\to}\langle p+1\rangle \times\langle q+1\rangle $$
such that $\phi'$ is sharp and 
$\link(\phi') = \link(\phi).$
\end{lem}

\begin{proof} Let $\psi$ be a lifting of $\phi$ of complexity $\link(\phi).$  Let $\{\bar{i}_1,\ldots,\bar{i}_k\}, \ k \le \link(\phi)$  be a subset of generators of $\langle n+1\rangle$
which get sent by $\psi$ to the the composites  of the form $(\bar{s},id)\circ(id,\bar{t})$ or $(id,\bar{t})\circ(\bar{s},id)$ (these are exactly corners of $\psi$ which are cut in $\phi$). 
Let $n'= n+k$ and let $\delta:\langle n+1\rangle \to \langle n'+1\rangle$ is defined as follows:

 $$\delta(\bar{j}) = \left\{
\begin{array}{lcc}
\bar{j}& , & 0\le j<i_1  \\
\overline{j+m}& , &i_m< j < i_{m+1}\\
\overline{i_m+m-1}\circ \overline{i_m+m}& , & j=i_m \\
\overline{j+k}& , & i_k < j < n+1
\end{array}
\right.$$
\noindent and $\psi':\langle n'+1\rangle \stackrel{\phi'}{\to}\langle p+1\rangle \boxx \langle q+1\rangle $ is the lattice path defined on generators by 
$\psi'(\overline{j+m}) = \psi(\bar{j})$ for $i_m<j<i_{m+1}$ and $\psi'(\overline{i_m+m-1}) = (\bar{s},id) \ , \ \psi'(\overline{i_m+m}) = (id, \bar{t})$ if $\psi(\bar{i}_m) = (\bar{s},id)\circ(id,\bar{t}), $ and 
$\psi'(\overline{i_m+m-1}) = (id,\bar{t}) \ , \ \psi'(\overline{i_m+m}) = (\bar{s},id)$ if $\psi(\bar{i}_m) = (id,\bar{t})\circ(\bar{s},id).$

It is obvious from the construction that we have a factorisation of the lattice path $\psi$ as 
$$ \langle n+1\rangle \stackrel{\delta}{\to}\langle n'+1\rangle \stackrel{\psi'}{\to}\langle p+1\rangle \boxx \langle q+1\rangle .$$
Then the composite:
$$\phi': \langle n'+1\rangle \stackrel{\psi'}{\to}\langle p+1\rangle \boxx \langle q+1\rangle \to \langle p+1\rangle \times\langle q+1\rangle$$
is a sharp path and we have a factorisation of $\phi$ as $\delta$ followed by $\phi'$. It is also obvious that $\link(\phi') = \link(\phi).$
  
\end{proof}



\begin{lem}\label{sharpening2} Any path  $\phi:\langle n+1\rangle \to\langle p+1\rangle\times\langle q+1\rangle $ admits a factorisation 
\begin{equation}\label{sharpfactor} \langle n+1\rangle \to \langle n'+1\rangle \stackrel{\phi'}{\to} \langle p'+1\rangle \times \langle q'+1\rangle {\to} \langle p+1\rangle \times \langle q+1\rangle \end{equation} such that $\phi'$ is sharp Delannoy and 
$\link(\phi') = \link(\phi).$  


\end{lem}

\begin{proof}
We  use  Lemmas \ref{cont}, \ref{nond} and \ref{sharpening} consecutively. 

\end{proof}

 The following Lemma describes the relations of shuffle paths and linking numbers of their images under $p:\LL\to \MM.$

 \begin{lem}\label{shuflink}
 Let $$\psi: \langle n+1\rangle \stackrel{}{\to} \langle n_1 +1\rangle\boxx\ldots\boxx\langle n_k+1\rangle$$ be a shuffle lattice path.
Then 
\begin{enumerate}
\item The projection $\pi_{ij}(p(\psi))$ is a sharp Dellanoy path for any $1\le i<j\le k;$ 
 \item $supp(\pi_i(p(\psi)))\cap\ldots\cap supp(\pi_i(p(\psi)))\ = \emptyset $  and $supp(\pi_1(p(\psi))\cup\ldots\cup supp(\pi_k(p(\psi)))\ = \{\bar{0},\ldots,\bar{n}\},$ where $\pi_i$ is the projection on $i$-th coordinate;  
\item   $\link(\pi_{ij}(p(\psi))=\cc_{ij}(\psi).$  
\end{enumerate}
\end{lem}
 
\begin{proof} 
Easily follows from Lemmas  \ref{sharplink}.  \end{proof}


We will need several Lemmas about smooth paths behaviour. 

This Lemma shows that connects our notion of linking numbers to the Steenrod-McClure-Smith notion of overlapping subdivisions. 

\begin{lem} \label{overlap} Let $\phi:\langle n+1\rangle \to\langle p+1\rangle \times\langle q+1\rangle $ be a smooth path. Then
$$|supp(\phi_1)\cap supp(\phi_2)|\ = \link(\phi)\ .$$ \end{lem}
\begin{proof}
It is clear that for a smooth path there exists a minimal shuffling for which any two consecutive terms have exactly one element in their intersection. This proves the statement.
\end{proof}

\begin{remark} This Lemma shows that for any smooth path $\phi$ with $\link(\phi) = m$ the set of generating morphisms $\{\bar{0},\ldots,\bar{n}\}$ admits an overlapping partition with $m+1$ pieces in the sense of \cite[Definition 2.3]{ms}. \end{remark}

Similar to the sharp case there  is  a useful factorisation involving smooth paths. Such a factorisation is not unique sough. 

\begin{lem}\label{smoothfactorisation} For any Delannoy path $\phi:\langle n+1\rangle \to\langle p+1\rangle \times\langle q+1\rangle $  there exists (not necessary unique)  a  factorisation
 $$\phi: \langle n+1\rangle \stackrel{\phi'}{\to}   \langle p'+1\rangle \times\langle q'+1\rangle  \to     \langle p+1\rangle \times\langle q+1\rangle $$
such that $\phi'$  is smooth and  $\link(\phi') = \link(\phi).$
  
\end{lem} We will show that there is a factorisation of the Delannoy path which decrease the number of corners by one but does not change the linking numbers. We can then repeat this procedure until we get a Delannoy path without corners.

\begin{proof} 

Indeed, suppose   $\phi:\langle n+1\rangle \to\langle p+1\rangle \times\langle q+1\rangle $ has a corner at  some point. We can assume that this is a low corner  such that $\phi(\bar{i}) = (\bar{s},id)$ and $\phi(\overline{i+1}) = (id,\bar{t}).$ We then construct a new path
 $\phi':\langle n+1\rangle \to\langle p+2\rangle \times\langle q+1\rangle $ by 
 $$\phi'(\bar{j}) = \left\{
\begin{array}{lccc}
\phi(\bar{j})& , & 0\le j\le i  & \\
(\overline{s+1},\bar{t})& , &j = i+1 & \\
(\overline{a+1},\bar{b})& , & j > i+1, & \phi(j) = (a,b), a\ne id  \\
(id_{k+1},\overline{b})& , & j > i+1, & \phi(j) = (id_k,b), 
\end{array}
\right.$$ 
For example, given a Delannoy path as on the left the procedure applied to the first left corner  produces a path as an the right
$$\xygraph{
!{<0cm,0cm>;<0.6cm,0cm>:<0cm,0.6cm>::}
!{(-1,3)}*{\circ}  !{(0,3)}*{\circ}   !{(1,3)}*{\circ}  !{(2,3)}*{\circ} !{(3,3)}*{\circ}   !{(4,3)}*{\circ}  !{(5,3)}*{\bullet}="11"
!{(-1,2)}*{\circ}  !{(0,2)}*{\circ}   !{(1,2)}*{\circ}  !{(2,2)}*{\circ} !{(3,2)}*{\circ}   !{(4,2)}*{\circ}  !{(5,2)}*{\bullet}="10"
!{(-1,1)}*{\circ}  !{(0,1)}*{\circ}  !{(1,1)}*{\circ}  !{(2,1)}*{\bullet}="6" !{(3,1)}*{\bullet}="7"  !{(4,1)}*{\bullet}="8"  !{(5,1)}*{\bullet}="9"
!{(-1,0)}*{\circ}  !{(0,0)}*{\circ}   !{(1,0)}*{\circ}  !{(2,0)}*{\bullet}="5"  !{(3,0)}*{\circ}   !{(4,0)}*{\circ}  !{(5,0)}*{\circ}
!{(-1,-1)}*{\bullet} ="1" !{(0,-1)}*{\bullet}="2"   !{(1,-1)}*{\bullet}="3"  !{(2,-1)}*{\bullet}="4" !{(3,-1)}*{\circ}   !{(4,-1)}*{\circ}  !{(5,-1)}*{\circ}
"1":"2"     "2":"3"  "3":"4"  "4":"5" "5":"6"  "6":"7"  "7":"8"  "8":"9"  "9":"10"  "10":"11"
}\qquad\qquad\qquad\qquad
\xygraph{
!{<0cm,0cm>;<0.6cm,0cm>:<0cm,0.6cm>::}
!{(-1,3)}*{\circ}  !{(0,3)}*{\circ}   !{(1,3)}*{\circ}  !{(2,3)}*{\circ} !{(3,3)}*{\circ}   !{(4,3)}*{\circ}  !{(5,3)}*{\bullet}="10"
!{(-1,2)}*{\circ}  !{(0,2)}*{\circ}   !{(1,2)}*{\circ}  !{(2,2)}*{\circ} !{(3,2)}*{\circ}   !{(4,2)}*{\circ}  !{(5,2)}*{\bullet}="9"
!{(-1,1)}*{\circ}  !{(0,1)}*{\circ}  !{(1,1)}*{\circ}  !{(2,1)}*{\bullet}="5" !{(3,1)}*{\bullet}="6"  !{(4,1)}*{\bullet}="7"  !{(5,1)}*{\bullet}="8"
!{(-1,0)}*{\circ}  !{(0,0)}*{\circ}   !{(1,0)}*{\circ}  !{(2,0)}*{\bullet}="4"  !{(3,0)}*{\circ}   !{(4,0)}*{\circ}  !{(5,0)}*{\circ}
!{(-1,-1)}*{\bullet} ="1" !{(0,-1)}*{\bullet}="2"   !{(1,-1)}*{\bullet}="3"  !{(2,-1)}*{\circ} !{(3,-1)}*{\circ}   !{(4,-1)}*{\circ}  !{(5,-1)}*{\circ}
"1":"2"     "2":"3"  "3":"4"  "4":"5" "5":"6"  "6":"7"  "7":"8"  "8":"9"  "9":"10"  
}$$
\\

Now the map $\partial_{s+1}: \langle p+2 \rangle \to \langle p+1 \rangle$ which sends $\overline{s+1}$ to the identity and other generating morphims to the generating morphisms provides a factorisation of $\phi:$
$$\langle n+1\rangle \stackrel{\phi'}{\to}\langle p+2\rangle \times\langle q+1\rangle \stackrel{\partial_{s+1}\times 1}{\longrightarrow}\langle p+1\rangle \times\langle q+1\rangle $$ 
Clearly, the linkage numbers of $\phi$ and $\phi'$ are equal. 
\end{proof}

\begin{cor} \label{corsmooth} Any path  $\phi:\langle n+1\rangle \to\langle p+1\rangle\times\langle q+1\rangle $ admits a factorisation 
\begin{equation}\label{sharpfactor} \langle n+1\rangle \to \langle n'+1\rangle \stackrel{\phi'}{\to} \langle p'+1\rangle \times \langle q'+1\rangle {\to} \langle p+1\rangle \times \langle q+1\rangle \end{equation} such that $\phi'$ is smooth path and 
$\link(\phi') = \link(\phi).$  
\end{cor} 

\begin{proof} It follows easily from Lemmas  \ref{sharpening2} and \ref{smoothfactorisation}.
\end{proof}

We will use Lemma \ref{sharpening2} and Corollary \ref{corsmooth} to obtain two slightly different descriptions of the operad $\MM^{(n)}.$ These descriptions will be very useful later.  

\begin{defin} 
A sharp $n$-commutative cosimplicial monoid is a cosimplicial monoid which satisfies the $n$-commutativity condition with respect to any pair of maps in $\De$ represented by a sharp Delannoy path. 
\nl
Similarly, a smooth $n$-commutative cosimplicial monoid is a cosimplicial monoid which satisfies the $n$-commutativity condition with respect to any pair of maps in $\De$ represented by a smooth path.
\end{defin}

 Let $(\phi,\sigma)$ be an element of $\MM^{(0)}(p,q;k).$
 We  say that $(\phi,\sigma)$ is {\em sharply $n$-equivalent} to $(\phi,\tau\sigma)$ if 
 $\phi$ is a sharp Delannoy path and $\link(\phi) \le n.$  We  say that $(\phi,\sigma)$ is {\em smoothly $n$-equivalent} to $(\phi,\tau\sigma)$ if $\phi$ is a sharp path and $\link(\phi) \le n.$  

Let $\MM^{(n)}_{sh}$ be an operad constructed as a quotient  of $\MM^{(0)}$ by the equivalence relation generated by sharp $n$-equivalence relation and $\MM^{(n)}_{sm}$ be an operad constructed as a quotient  of $\MM^{(0)}$ by the equivalence relation generated by smooth $n$-equivalence relation. 
 
\begin{thm} \label{sharpeq} 
The operads $\MM^{(n)}_{sh}, \MM^{(n)}_{sm}$ and  $\MM^{(n)}$ are canonically isomorphic.  
\end{thm} 
\begin{proof}  
It is obvious that there is a morphism 
$\MM^{(n)}_{sh}\to  \MM^{(n)}$ and so every $n$-commutative cosimplicial monoid is a sharp $n$-commutative cosimplicial monoid. To prove that  $\MM^{(n)}_{sh}\to  \MM^{(n)}$ is an isomorphism is enough to show that a sharp $n$-commutative cosimplicial monoid $E$ in $\bSet$ is also an $n$-commutative cosimplicial monoid. For that, let $$ \phi:\langle k+1\rangle \to \langle p+1\rangle \times \langle q+1\rangle $$ be a path such that $\link(\phi) = n.$ 
By Lemma \ref{sharpening2} we have a factorisation
$$\langle k+1\rangle \to \langle k'+1\rangle \stackrel{\phi'}{\to} \langle p'+1\rangle \times \langle q'+1\rangle {\to} \langle p+1\rangle \times \langle q+1\rangle $$
where $\phi'$ is sharp and $\link(\phi)=\link(\phi').$ Then  the images of two elements from $E(p)$ and $E(q)$ commute in $E(k')$ and, hence, in $E(k).$ 
\nl
The result for smooth case follows from Corollary \ref{corsmooth} by a similar argument. 
\end{proof} 

\subsection{Smooth paths with linking numbers $1$ and $2$}\label{smooth paths lk=1.2} 

It is instructive and will be useful in applications to see what Theorem \ref{sharpeq} gives us for $n$ equal one and two. 

There are exactly two  smooth paths of complexity one on a rectangle. The first one is given by a map  $\eta_{m,n}:\langle m+n+1\rangle\to \langle m+1\rangle \times \langle n+1\rangle$, sending 
\begin{equation}\label{etanm}\mbox{$ i\mapsto \left\{\begin{array}{cc} (i,0),& 0\leq i\leq m\\  (m+1,i-m),& m+1\leq i\leq m+n+1\end{array}\right.$} \end{equation}
It has a (unique, for $m+n>0$) lifting $\langle m+n+1\rangle\to \langle m+1\rangle \boxx \langle n+1\rangle$ of complexity 1. 
Graphically:
\bigskip
$$\xygraph{
!{<0cm,0cm>;<1cm,0cm>:<0cm,1cm>::}
!{(-2,2)}*{\circ} ="uull" !{(-1,2)}*{\circ} ="uul" !{(0,2)}*{\hdots} ="uu"  !{(1,2)}*{\circ}="uur"  !{(2,2)}*{\circ}="urr"
!{(-2,1)}*{\circ} ="ull" !{(-1,1)}*{\circ} ="ul" !{(0,1)}*{} ="u"  !{(1,1)}*{\circ}="ur"  !{(2,1)}*{\circ}="rur"
!{(-2,0)}*{\circ} ="ll" !{(-1,0)}*{\circ} ="l" !{(0,0)}*{} ="m"  !{(1,0)}*{\circ}="r"  !{(2,0)}*{\circ}="rr"
!{(-2,-1)}*{\circ} ="dll" !{(-1,-1)}*{\circ} ="dl" !{(0,-1)}*{\hdots} ="d"  !{(1,-1)}*{\circ}="dr"  !{(2,-1)}*{\circ}="drr"
!{(-2,.6)}*+{\vdots}  !{(2,.6)}*+{\vdots}
"dll":"dl"     "dr":"drr" 
"drr":"rr"     "rur":"urr"
"dr"-@{..>}"rr"   
}
$$
Here the path  $\eta_{m,n}$ is in dotted arrows and its lifting is in unbroken arrows with the multiplicity of the corner point being zero and the multiplicities of all other points being one.

The other path $\eta^t_{n,m}$ goes around boundary of the rectangle in the clockwise direction. Observe that the Joyal dual to the projection on the first coordinate of the path $\eta_{n,m}$ is the map $\tau_{n,m}$ in $\De$ introduced in the Example \ref{elno}. The Joyal dual to the second projection is $\pi_{n,m}$. Similarly for the path $\eta^t_{n,m}.$ By Corollary \ref{corsmooth} and Theorem \ref{sharpeq} we have the following:

\ble\lb{lno}
For  any two maps in $\De$ $\tau:[p]\to [r],\ \pi:[q]\rightarrow [r]$ with linking number one there are $n$ and order preserving maps $\phi:[p]\to [n],\ \psi:[q]\to [r-n]$ such that $\tau = \tau_{r-n,n}\phi$ and $\pi = \pi_{r-n,n}\psi$.

Thus a cosimplicial monoid is $1$-commutative if and only if it satisfies $1$-commutativity condition with respect to all maps $\tau_{n,m}$ and $\pi_{n,m}.$
\ele

A typical smooth path with linking number two has the following form 

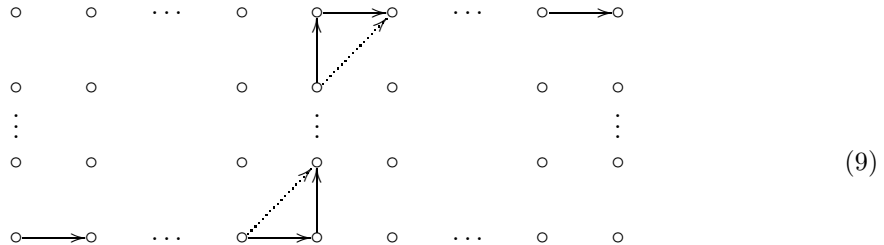
\begin{equation} \label{etanmi}
\xygraph{
!{<0cm,0cm>;<1cm,0cm>:<0cm,1cm>::}
!{(-2,2)}*{\circ} ="uull" !{(-1,2)}*{\circ} ="uul" !{(0,2)}*{\hdots} ="uu"  !{(1,2)}*{\circ}="uur"  !{(2,2)}*{\circ}="urr"
!{(3,2)}*{\circ} ="u3" !{(4,2)}*{\hdots} ="u4"  !{(5,2)}*{\circ}="u5"  !{(6,2)}*{\circ}="u6"
!{(-2,1)}*{\circ} ="ull" !{(-1,1)}*{\circ} ="ul" !{(0,1)}*{} ="u"  !{(1,1)}*{\circ}="ur"  !{(2,1)}*{\circ}="rur"
!{(3,1)}*{\circ} ="um3"   !{(5,1)}*{\circ}="um5"  !{(6,1)}*{\circ}="um6"
!{(-2,0)}*{\circ} ="ll" !{(-1,0)}*{\circ} ="l" !{(0,0)}*{} ="m"  !{(1,0)}*{\circ}="r"  !{(2,0)}*{\circ}="rr"
!{(3,0)}*{\circ} ="bm3"   !{(5,0)}*{\circ}="bm5"  !{(6,0)}*{\circ}="bm6"
!{(-2,-1)}*{\circ} ="dll" !{(-1,-1)}*{\circ} ="dl" !{(0,-1)}*{\hdots} ="d"  !{(1,-1)}*{\circ}="dr"  !{(2,-1)}*{\circ}="drr"
!{(3,-1)}*{\circ} ="b3" !{(4,-1)}*{\hdots} ="b4"  !{(5,-1)}*{\circ}="b5"  !{(6,-1)}*{\circ}="b6"
!{(-2,.6)}*+{\vdots}  !{(2,.6)}*+{\vdots}  !{(6,.6)}*+{\vdots}
"dll":"dl"     "dr":"drr" 
"drr":"rr"     "rur":"urr"
"urr":"u3"  "u5":"u6"
"dr"-@{..>}"rr"    "rur"-@{..>}"u3"
}
\end{equation} 
\noindent where the bottom diagonal map goes from  $(i,0)$ to $(i+1,1).$ Again it is not hard to see that the Joyal duals for the corresponding projections are 
$\tau_{m,n}^i$ and $\pi_{m,n}^i$  and we have:

\ble\lb{lnt}
For any two maps in $\Delta$ $\tau:[p]\to [r-1],\ \pi:[q]\rightarrow [r-1]$ with linking number two there are $i, m$ and order preserving maps $\phi:[p]\to [r-m],\ \psi:[q]\to [m]$ such that $\tau = \tau_{m,r-m}^i\phi$ and $\pi = \pi_{m,r-m}^i\psi$.

Thus a $1$-commutative cosimplicial monoid is $2$-commutative if and only if it satisfies $2$-commutativity condition with respect to all maps $\tau^i_{n,m}$ and $\pi^i_{n,m}.$
\ele

 \subsection{Lattice paths action on $n$-commutative cosimplicial  monoids}

\begin{thm} \label{LtoM}

There are  morphisms of operads $p^{(n)}: \LL^{(n)}\to \MM^{(n-1)} \, n\ge 1$ 
  making the following diagram commutative:

\begin{equation*}
    \xymatrix@C = +2em{
      \LL^{(1)}  \ar[d]^{p^{(1)}} \ar[r] & \LL^{(2)} \ar[d]^{p^{(2)}}  \ar[r] &\ldots   \ar[r]     & \LL^{(n)}\ar[d]^{p^{(n)}} \ar[r] &\ldots   \ar[r]   & \LL \ar[d]^{p^{(\infty)}}             \\
        \MM^{(0)} \ar[r]^{} & \MM^{(1)}    \ar[r] &\ldots   \ar[r]  &  \MM^{(n-1)} \ar[r] &\ldots \ar[r]  & \MM
 }
\end{equation*}

\end{thm}

\begin{proof} 
We construct a map of collections $p^{(n)}$ as follows
$$ \LL^{(n)}(n_1,\ldots,n_k;m)\stackrel{p\times\varpi^{(n)}}{\longrightarrow} \MM^{(0)}(n_1,\ldots,n_k;m) \stackrel{q^{(n-1)}}{\longrightarrow}   \MM^{(n-1)}(n_1,\ldots,n_k;m).$$
The first map in this composite is not an operadic map for $n> 1$ so we will need to prove that the composite is a map of operads. 

Let $\circ_i$ be the usual operadic circle product. We have to prove then that
$(p\times\varpi^{(n)})(\psi\circ_i \omega)$ is  $(n-1)$-equivalent to  the composite   $(p\times\varpi^{(n)})(\psi)\circ_i (p\times\varpi^{(n)})( \omega)$ for any lattice paths. Due to the Lemma \ref{sharplink} point 3, Theorem \ref{sharpeq} and Lemma \ref{40} 
it will be  enough to prove this statement for two shuffle paths $\psi$ and $\omega.$ 

Without loss of generality we can also assume that $\varpi(\psi) = id$ and $\varpi(\omega) = id.$ This is because we can apply appropriate permutations to $\psi$ and $\omega.$ 
Moreover, we will assume that  $i=1.$ This is because substitution to the $i$-th variable affects the first  movement order of the variables $i,i+1, \ldots,k$ only. So, $i=1$ is the most general case.

Now, the first movement order and pairwise linking numbers of projection depend only on the sketch of the path  $\psi\circ_i \omega.$ Before we proceed let consider an example.

Let $\omega$ of complexity $2$ has a following  presentation: $t = t_1 t_2  t_3 t_2  t_1.$  It is easy to see that all pairwise complexity indices are $2$ as well. 
And $\varpi(t) = (123).$ 

 Let $\psi$ be such that $\tr_1(\psi) = (s)_1 = s_1 s_2 s_3 s_2 s_1 s_1 s_1 s_1.$  Then again $\varpi(s) = (123)$ and  any pairwise complexity is $2.$
  
   Then after substitution to the first variable we have the following sketch:
$$(s)_1\circ t  = (t_1 ) s_2 s_3  s_2  (t_2)  (t_3) (t_2) ( t_1) = s_1 s_4 s_5  s_2  s_2 s_3 s_2 s_1  \ .$$
Then $\varpi((s)_1\circ t) = (14523) \ne (12345).$ 

But the complexities of $\cc_{24}((s)_1\circ t ) =  \cc_{25}((s)_1\circ t )  =  \cc_{34}((s)_1\circ t ) = \cc_{35}((s)_1\circ t ) = 1 <2 .$ Thus in $\MM^{(0)}$ the pairs 
$(p(\psi\circ_1 \omega), (12345))$ and $(p(\psi\circ_1 \omega), (14523))$ are $1$-equivalent buy Lemma \ref{eqshuffle}. 

We will see that such behaviour holds in general. That is if the first movement order on a pair of indices changes on opposite after substitution then the corresponding complexity index also drops and so the resulting elements of $\MM^{(0)}$ are $(n-1)$-equivalent.  

Here is a general recipe how to compute the complexity index $\cc_{ij}((s)_1\circ t )$ in $\TT.$  It is not hard to see that  for $t\in FM(d)$ and $d< i < j \le d+m-1$   
 the complexity  $\cc_{ij}(s\circ_1 (t))$ is equal to $\cc_{i-d\ j-d}(s).$ And if $1\le i < j \le d$  the complexity is $\cc_{ij}(s\circ_1 (t)) = \cc_{i j}(t).$ A nontrivial case is where 
 $1\le i \le d$ and $d+1 < j \le d+m-1.$ This corresponds to the case where $s_{j-d}\ne s_1$ is one of the variables in the sketch $s$ and $t_{i}$ is a variable from $t.$ 
 
 Then we do the following: 
 \begin{enumerate}\item Out of the expansion $(s)_1$ we construct a new sketch expansion $(s')_1$  by putting all  variables except for $s_1$ and $s_{j-d}$ to be $e$ and then do reduction by variable $s_{j-d}$ only;  
 \item  in the  shuffle path $t$ we put all variables except for $t_{i}$ to be $e$ and then obtain a word $t'$ on variables  $e$ and $t_i;$
 \item  we  insert $p$-th letter from $t'$ to the $p$-th copy of  $s_1 $ in $(s')_1;$ 
 \item we renumber, do reduction and obtain  a sketch  on two variables $s_2$ and $s_{1};$
 \item  The number of letter  of this sketch minus one is exactly  $\cc_{ij}((s)_1\circ t) .$ 
 
 \end{enumerate} 
 
 Let us see how it works in the example above. And suppose we have to compute the index $\cc_{24}((s)_1\circ t)  = \cc_{24}(s_1 s_4 s_5  s_2  s_2 s_3 s_2 s_1).$
 Then we do the following    : 
 \begin{enumerate}\item $(s')_1 =  s_1 s_2 e s_2 s_1 s_1 s_1 s_1 = s_1 s_2  s_1 s_1 s_1 s_1        ;$ 
 \item  $t'=  e t_2  e t_2  e; $  

 \item $(s')_1\circ t'  = e s_2 t_2 e t_2 e = s_2  s_1  ;$ 
 \item  $\cc_{24}((s)_1\circ t) =   \cc(s_2  s_1) = 1.$ 

 \end{enumerate} 
     
In general, let a sketch expansion $(s)_1$ of complexity $n$   have the following form:
$$\underbrace{{\bold s}_1\ldots {\bold s}_1}_l {\hat s}_2 {\hat s}_3 \ldots {\hat s}_{i}\ldots {\hat s}_{p} \ldots {\bold s}_1\ldots {\bold s}_1 \ldots, $$ 
where: 
\begin{enumerate}
\item $\varpi(s)$ = id; 
 \item There is no other $s_1$ in the interval between bold letters;
\item  ${\hat s_j}$ is the first appearance of $s_j$ in $s.$ 
\end{enumerate}

Let $t\in FM(d)$ be  such that the composite $(s)_1\circ t$ is defined.  The only possibility to have an inverse order in $\varpi((s)_1\circ t)$ 
for a variable $s_j \ , j\le d$ and $s_{2+d},\ldots,s_{p+d}$ is if $t_j$ appears in $t$ in a place $l+b$ for some  $b>0.$ We want to compute $\cc_{j i+d}((s)_1\circ t).$  Then, after making 
the first two steps described  above we obtain an expansion 
$$(s')_1 = \underbrace{{\bold s}_1\ldots {\bold s}_1}_l  {\hat s}_{i}{\bold s}_1\ldots {\bold s}_1 \ldots $$ and the word 
$$t' = \underbrace{e\ldots e}_{> l}t_j\ldots .$$   
Since the complexity $\cc_{1i}(s) \le n$ the number of variables of  $s_1$ in $s'$ can not exceed $n/2 +1 $ (if $n$ is even) or  $(n+1)/2$ (if $n$ is odd) and the number of variables  $s_i$ in $s'$ is strictly less then $n/2 +1 $ (if $n$ is even) or  less or equal to $(n+1)/2$ (if $n$ is odd).

 Then after the substitution and reduction we have a sketch from $\TT(2)$ which starts from $s_2$ (which corresponds to the variable $s_{i+d}$) followed by $s_1$ (which corresponds to $s_j$) and the number of occurrences of $s_1$ variable can not exceed the number of $s_1$ variables in $s'$ minus one. So the overall number of letters in this word can not exceed $(n/2 +1-1) + (n/2 +1 -1) = n$ (if $n$ is even) and $(n+1)/2 + (n+1)/2 -1 = n$ if $n$ is odd. So the complexity of   $(s')_1\circ t'$ can not exceed $n-1$ as was claimed.    
 
Now the result follows from iterated application of Lemma \ref{eqshuffle}.

\end{proof}

\subsection{The little $n$-cubes operad action}

Let now $(\bV,\otimes, I)$ be a symmetric monoidal model category.
Recall that  a
\emph{standard system of simplices} in $(\bV,\otimes, I)$ is a cosimplicial object $\delta$ in $\bV$ satisfying the following properties \cite[Definition A.6]{BergerMoerdijk0}:
\begin{itemize}
\item[(i)]
$\delta$ is cofibrant for the Reedy model structure on
  $\bV^{\De}$,
\item[(ii)] 
$\delta^0$ is the unit object $I$ of $\bV$ and the
  simplicial operators $[m]\to[n]$ act via weak equivalences
  $\delta(m)\to\delta(n)$ in $\bV$, and
\item[(iii)]
the simplicial realization
  functor $|-|_\delta=(-)\otimes_\De\delta:\bV^{\Delta^{op}}\to
  \bV$ is a symmetric monoidal functor whose structural maps
  \[
|X|_\delta\otimes |Y|_\delta\to|X\otimes Y|_{\delta}
\]
are weak
  equivalences for Reedy-cofibrant objects $X,Y \in \bV^{\De^{op}}$.\end{itemize}  

Recall also that the totalization of a cosimplicial object $X:\De\to \bV$ with respect to $\delta$ is given by the end
$$Tot_{\delta}(X). = \int_{[n]\in \De} X(n)^{\delta(n)}\ $$  

Let also $\cal P$ be a $\bN$-coloured operad in $\bSet$ whose underlying category is $\De.$
  The operad $\coend_{\cal P}(\delta)$ is the single colour operad whose degree $n$ component is
  $$\coend_{\cal P}(\delta)(n) = Tot_{\delta}(\otimes^n_{\cal P}(\delta,\ldots,\delta)),$$   
where $\otimes^n_{\cal P}$ is the $n$-th convolution tensor product of $\cal P.$ The operad $\coend_{\cal P}(\delta)$ acts on  $\delta$-totalization $Tot_{\delta}(X)$ of a
$\cal P$-algebra $X$ \cite{BB}. 

We therefore have the following theorem:

\bth\lb{main} \begin{enumerate} \item
The totalization $Tot_{\delta}(E)$  of an $(n-1)$-commutative cosimplicial monoid $E$ in $V$  has a natural $\coend_{\MM^{(n-1)}}(\delta)$-algebra structure;
\item By restriction this totalization is also a $\coend_{\LL^{(n)}}(\delta)$-algebra;
\item If $\LL$ is strongly $\delta$-reductive then $Tot_{\delta}(E)$ is naturally an $E_{n}$-algebra.
\end{enumerate}

\eth
\bpf This follows promptly from Theorem \ref{LtoM} and Theorem 4.8 from \cite{BB}. 
\epf

\begin{cor}\label{coraction}  The totalization $Tot_{\delta}(E)$  of an $(n-1)$-commutative cosimplicial monoid $E$ has a natural  structure of $e_{n}$-algebra in any of the following cases:
\begin{enumerate}\item $\bV$ is the category of topological spaces with Quillen model structure and $\delta_{top}(k), k\ge 0$ is the geometric realisation of the representables  $\Delta(k):\bSet^{\De^{op}}\to \bSet,  \ k\ge 0$ 
\item $\bV$ is the category of chain complexes over a commutative ring and $\delta(k), k\ge 0$ is the chain complex of simplicial chains of $C_*(\Delta(k));$

\item $\bV$ is the category of chain complexes over a commutative ring and $\delta(k), k\ge 0$ is the chain complex of normalized simplicial chains $Nor_*(\Delta(k)).$

\end{enumerate}

\end{cor}

\begin{proof} It follows from \cite[Example 4.10]{BB} and \cite[Appendix]{BBM} that in all these cases the condition of strong reductivity of $\LL$ is satisfied. 

\end{proof}

\subsection{Steenrod $\scup_i$-products and Poisson bracket}

In this section we consider case (3) of Corrollary \ref{coraction} in more details. Without loss of generality we can assume here that our base commutative ring is $\mathbb{Z},$ so we work in the category of abelian groups. 
As it was stated $\delta(k)$ in this case consists of the integer normalised simplicial chains of $\Delta(k).$ 
For convenience let us denote the operad $\coend_{\LL^{(n)}}(\delta)$ as $\ll,$ so that $\ll_c(k)$ means the $c$-th term of the cochain complex $\ll(k)$ in the operadic degree $k.$ 
Similarly, we introduce notation $\mm$ for the operad $\coend_{\MM^{(n-1)}}(\delta).$

We are going to describe more explicitly the components of  the operad $\ll.$ 
 By definition this chain complex is the normalised cosimplicial totalization of the  cosimplicial chain complex   $\otimes^k_{\LL^{(n)}}(\delta,\ldots,\delta).$  Even more explicitely $\ll_c(k)$  is the intersection of kernels of codegeneracy operators  in the product 
$$\prod_{m\ge 0} \bigoplus_{m= n_1+\ldots n_k -c} \mathcal{Z}^{(n)}(n_1,\ldots,n_k;m),$$
where $ \mathcal{Z}^{(n)}(n_1,\ldots,n_k;m)$ is the quotient of the free abelian group generated by  $\LL^{(n)}(n_1,\ldots,n_k;m)$ by images of simplicial degeneracies \cite[Section 3]{BBM}.

We can  explain how the simplicial operators work in terms of labelled lattice paths. The action of a face operator $\partial_i^s , \ 0\le i \le n_s , \ 1\le s \le k$ can be seen as follows. In the hypercube $\langle n_1 +1\rangle\boxx\ldots\boxx\langle n_k+1\rangle$ take the `layer'  $$\langle n_1 +1\rangle\boxx\ldots\boxx \underbrace{\langle i,i+1\rangle}_{s}\boxx  \ldots\boxx\langle n_k+1\rangle$$
and 'collapse' it.  This will produce a hypercube $$\langle n_1 +1\rangle\boxx\ldots\boxx \langle n_s \rangle \boxx  \ldots\boxx\langle n_k+1\rangle.$$ 
If   $\langle n_1 +1\rangle\boxx\ldots\boxx\langle n_k+1\rangle$ was labelled according to the lattice path   $\psi$  then  $\langle n_1 +1\rangle\boxx\ldots\boxx \langle n_s \rangle \boxx  \ldots\boxx\langle n_k+1\rangle$ acquires a labelling which is the result of adding the labels of the vertices with the $s$-th coordinates $i+1$ to the labels of the vertices with $s$-cordinates $i.$  

Analogously, a simplicial degeneracy works by `expantion' of the corresponding layer and labelling $0$ all new points (see \cite[Section 2.4]{BBM} for examples). 

Cosimplicial coface and codegeneracy operators do not change the hypercube and the path.  A coface just adds $1$ to the corresponding label whereas a codegeneracy sucbstracts $1$ (see \cite[Section 2.5]{BBM}).

It is not hard to see from this description of simplicial operators that an 
lattice path is  degenerate if it has a stop labelled $0$ at some internal point (but may have nonzero stops in the corners. 

We then  take a normalised cosimplicial totalization  of the corresponding cosimplicial abelian group. The differential is the usual sum of alternating coface operators and  we also take intersection of kernels of all codegeneracies. A lattice path is in this intersection if all stops at internal points  are labelled $1.$ Since we want it to be a degenerate path it is necessary and sufficient for this label to be $1$ at internal point.  But such a path also has to have $0$ labels at any corner since if it is greater than $0$ the corresponding codegeneracy can never produce a degenerate path after subtracting $1$ from this label. We will call such lattice path {\it normal}. Let $\nlp^{(n)}_m(n_1,\ldots,n_k)\subset \ll_m(k)$ be  the set of normal lattice paths of complexity  equal exactly  $n.$
 
\begin{defin} A lattice path 
$\psi: \langle m+1\rangle \stackrel{}{\to} \langle p+1\rangle\boxx\langle q+1\rangle$
is called smooth if its image $p(\psi)$ is a smooth path in $\MM.$ 
\end{defin} 
Let $\slp^{(n)}_m(p,q)$ be the set of all smooth lattice paths of complexity exactly $n.$
\begin{lem} \label{smooth-normal}
The following is true:
\begin{enumerate} 
\item Any smooth path $\phi$ such that $\link(\phi) = n$ has at least one normal lifting of complexity $n;$
\item If  $\psi$ is a lifting of $\phi$ with $\link(\phi) = n$ and $\cc(\psi)>n$ then $\psi$ is not normal. 
\item  The set $\nlp^{(n)}_m(p,q)$ is nonempty only if  $m=p+q-n+1.$ 
\item The set $\slp^{(n)}_m(p,q)$ consists of all normal liftings of smooth paths in $\MM$ with linking numbers  exactly $n.$  
\end{enumerate} 
\end{lem} 
\begin{proof} First part of the Lemma is immediate from the definitions.  

Any normal path can be completed to  a shuffle path by adding $1$ to all labels $0.$ The complexity of these two paths remains the same.  For a fixed $m\ge 0$ the set of shuffle paths of all complexities $\langle k+1\rangle \stackrel{}{\to} \langle p+1\rangle\boxx\langle q+1\rangle$ is in one to one correspondence with the  set of all $(p,q)$-shuffles, hence, it is nonempty if and only if $k=p+q+1.$   Complexity $n$ paths have exactly $n$-corners. To get back a normal path we have to substruct $1$ from the labels in the corners. Hence, for the  set of normal lattice paths $\nlp^{(n)}_m(p,q)$ to be nonempty it is necessary that   $m = p+q +1 -n.$      \end{proof}

\begin{remark} The condition $m = p+q -n +1$ does not guaranty that there exists a normal or smooth path with this $m$ thus the implication (3) can not be inverted.   
\end{remark}
\begin{remark} Due to the points (3) and (4) of the Lemma it is tempting to skip the subscript $m$ in the notations for $\nlp^{(n)}_m(p,q)$ and $\slp^{(n)}_m(p,q).$ 
But we leave these notations like they are because of convenience to remember the number $m$ in some formulas involving summations over different $p$ and $q.$ 
\end{remark}

Call a lattice path from  $\nlp^{(n)}(p,q) $ {\it even}     if its first movement order is $(12)$ and odd otherwise. Let $\nlp_{+m}^{(n)}(p,q)$ be the set of even paths and  $\nlp_{-m}^{(n)}(p,q)$ be the set of odd paths. Analogously,  one can give a similar definition for smooth lattice paths. Let $\slp_{+m}^{(n)}(p,q)$ be the set of even smooth lattice paths and  $\slp_{-m}^{(n)}(p,q)$ be the set of odd smooth lattice paths.

\begin{defin} The {\em sign} $\sgn(\psi)$ of a lattice path $\psi$ is the sign of  the shuffle permutation $\mu(\psi^\dag)$  determined by the shuffle path ${\psi}^\dag$ (see Lemma \ref{classical}).
\end{defin}

\begin{lem}\label{even=odd} For a  lattice path $\psi: \langle m+1\rangle \to \langle p+1\rangle\boxx \langle q+1\rangle $  the following is true :$${\sgn}(\psi) = (-1)^{(p-1)(q-1)}{\sgn}(\psi^{t})$$
where $\psi^{\mathrm{t}}$ is obtained from $\psi$ by the action of involution in $\LL$, i.e.
$$\langle m+1\rangle \stackrel{\psi}{\to} \langle p+1\rangle\boxx \langle q+1\rangle \stackrel{\tau}{\to} \langle q+1\rangle\boxx \langle p+1\rangle .$$  
 \end{lem}
 \begin{proof}
 We can assume from the beginning that $\psi$ is a shuffle path. We can also assume that $\psi$ is an odd path. 
 Let $0=j_0 < j_1 <\ldots < j_l <j_{l+1} = p+1 $ be the first coordinates of the points where the change of directions takes place. Let also  $0=i_0 < i_1 <\ldots < i_l <i_{l+1} = q+1 $ be the second coordinates of the points where the change of directions happens.
$$\xygraph{
!{<0cm,0cm>;<1cm,0cm>:<0cm,.8cm>::}
!{(-1,5)}*{\circ}  !{(0,5)}*{...}   !{(1,5)}*{\circ}  !{(2,5)}*{...} !{(3,5)}*{\circ}   !{(4,5)}*{...}  !{(5,5)}*{\bullet}
!{(-1,4.1)}*{\vdots}  !{(1,4.1)}*{\vdots}  !{(3,4.1)}*+{\vdots}="10"     !{(5,4.1)}*{\vdots}
!{(-1,3)}*{\circ}  !{(0,3)}*{...}   !{(1,3)}*{\bullet}="7"  !{(2,3)}*+{...}="8" !{(3,3)}*{\bullet}="9"   !{(4,3)}*{...}  !{(5,3)}*{\circ}
!{(-1,2.1)}*{\vdots}  !{(1,2.1)}*+{\vdots}="6" !{(1,2.1)}*{\vdots}="6a"  !{(3,2.1)}*{\vdots} !{(5,2.1)}*{\vdots}
!{(-1,1)}*{\bullet}="3"  !{(0,1)}*+{...}="4"  !{(1,1)}*{\bullet}="5"  !{(2,1)}*{...} !{(3,1)}*{\circ}  !{(4,1)}*{...} !{(5,1)}*{\circ}
!{(-1,0.1)}*+{\vdots}="2"  !{(-1,0.1)}*{\vdots}="2a"  !{(1,0.1)}*{\vdots}   !{(3,0.1)}*{\vdots}  !{(5,0.1)}*{\vdots}
!{(-1,-1)}*{\bullet} ="1" !{(0,-1)}*{...}   !{(1,-1)}*{\circ}  !{(2,-1)}*{...} !{(3,-1)}*{\circ}   !{(4,-1)}*{...}  !{(5,-1)}*{\circ}
!{(1,-1.3)}*{\scriptstyle j_1}  !{(3,-1.3)}*{\scriptstyle j_2}
!{(-1.3,1)}*{\scriptstyle i_1}  !{(-1.3,3)}*{\scriptstyle i_2}
"1":"2"  "2a":"3"  "3":"4"  "4":"5" "5":"6"  "6a":"7"  "7":"8"  "8":"9"  "9":"10"  
}$$
We then have the following  formula for the number of inversions in $\mu(\psi):$ 
$$ inv(\psi) = \sum_{s=0}^{l}(i_{s+1}-i_s)((p+1)-j_s)$$
Analogously we have
$$ inv(\psi^{\mathrm{t}}) = \sum_{s=0}^{l}(j_{s+1}-j_s)((q+1)-j_{s+1}).$$ 
The sum of this numbers modulo $2$ is:
$$\sum_{s=0}^l (p+1)(i_{s+1} +i_s) + \sum_{s=0}^l (q+1)(j_{s+1} +j_s) + \sum_{s=0}^l i_sj_s + j_{s+1}j_{s+1} = (p+1)(q+1).$$ 
The results follows. 
\end{proof} 

Let now $A$ be an algebra of $\LL^{(n)}.$  
\begin{defin} Let $0\le i \le n-1.$ The Steenrod product $x\scup_{i} y\in A(p+q-i) \ 0\le i\le n-1 , x\in A(p),y\in A(q)$ in the normalised totalisation of $A$ 
is given by the formula
$$x\scup_{i} y =  \sum_{\psi\in \nlp_{+m}^{(i+1)}(p,q)} (-1)^{(p-1)(q-1)}{\sgn}(\psi)\psi(x\otimes y), $$
where $\psi(x\otimes y)$ is the value of $\psi$ on $x\otimes y.$  
\end{defin}
Consider also the following operation of degree $(n-1)$ on $Tot(A):$
$$ \gamma^{(n-1)}(x\otimes y) = x\scup_{n-1} y - (-1)^n (-1)^{(p-1)(q-1)} y\scup_{n-1} x .$$
\begin{lem} The operation $\gamma^{(n-1)}$ is equal to the operation given by the following  formula:
$$\gamma^{(n-1)}(x\otimes y) =   \sum_{\psi\in \nlp_{+m}^{(n)}(p,q)} (-1)^{(p-1)(q-1)}{\sgn}(\psi)\psi (x\otimes y) + $$ $$ +\sum_{\psi\in \nlp_{-m}^{(n)}(p,q)}  (-1)^{(n-1)}{\sgn}(\psi)\psi(x\otimes y).$$ 
\end{lem}
\begin{proof}This follows promptly from Lemma \ref{even=odd}. 
\end{proof}

\bth \label{formulal} For $n \ge 2$  the  degree $1-n$ map $$\gamma^{(n-1)}:A(p)\otimes A(q)\to A(p+q-n+1)$$ 
induces  a  Poisson bracket  operation on $H^*(A).$  
\eth
\begin{proof}
The proof of this formula depends  on the combinatorics of $\LL^{(n)}.$ We will use the following fact about it proved in \cite[Theorem 3.10]{BBM}.
Let $\BB^{(n)}(k)$ be the normalised simplicial normalisation of the $k$-simplicial abelian group $\LL^{(n)}(\bullet_1,\ldots,\bullet_k,0)$ with the induced differential. One can define an operadic composition on this collection. This operad is called {\it the normalized brace operad of complexity $n$}. There is an operad map (whiskering)
$w:\BB^{n} \to \ll.$ Theorem 3.10 from \cite{BBM} states that this map is a quasi-isomorphism of operads. Moreover, for each $k\ge 0$ there is a quasiisomorphism of chain complexes
$$\BB^{(n)}(k)\stackrel{w}{\longrightarrow} \ll(k) \stackrel{i}{\longrightarrow}|\dot{\LL}^{(n)}|(k),$$
where $|\dot{\LL}^{(n)}|(k)$ is the un-normalized cosimpliocial realisation of the cosimplicial chain complex  $\otimes^k_{\LL^{(n)}}(\delta,\ldots,\delta).$
This composite admits a retraction $proj: |\dot{\LL}^{(n)}|(k)\to \BB^{(n)}(k),$ which is also a quasiisomorphism. In fact, the proof of this Theorem in \cite{BBM} consists exactly in proving that $proj$ is a quasiisomorphism. For this $|\dot{\LL}^{(n)}|(k)$ is expressed as a total complex of a double chain complex and $\BB^{(n)}(k)$ as its left column. The retraction is the projection on the left column of this double chain complex. 

Explicitly the differentials in this double chain complex look as follows. Given a lattice path of complexity $n$ $$\psi:\langle m+1\rangle\to \langle p+1\rangle \boxx \langle q+1\rangle$$
the horizontal differential  (which comes from the cosimplicial differential)  is equal to
$$(-1)^{n-1}\sum_{i=0}^{m+1}(-1)^i\delta_i$$
where $\delta_{i}(\psi)$ is the same path as $\psi$ except that the label $x$ at the point $\psi(i)$ becomes $x+1.$ 
\nl
There are also two  differentials $\partial^1$ and $\partial^2.$ 
$$\partial^1(\psi) = (-1)^{n+q-1}\sum_{i=0}^{p+1} (-1)^i\partial^1_i(\psi) \  \ \mbox{and} \  \
\partial^2(\psi) = \sum_{i=0}^{q+1} (-1)^i\partial^2_i(\psi),$$
where $\partial^2_i$ collapses the $i$-th row and $\partial^1_i$ collapses the $i$-th column.  
\nl
The vertical differential in the double chain complex is the sum $\delta + \partial,$ where $\partial = \partial^1 + \partial^2.$

To prove our Theorem it suffice to show that there is a cocycle in $\ll(2),$ mapped by $proj$ to a generator of the cohomology group $H^{n-1}(\BB^{(n)}(2)) = \mathbb{Z}.$ 
\nl
Consider the following linear combination of lattice paths for each $m\ge 0:$
\begin{equation}\label{formula2}  \lambda^{(n-1)}_m 
=  \sum_{p+q = m+n-1}( \sum_{\psi\in \nlp_{+m}^{(n)}(p,q)} (-1)^{(p-1)(q-1)}{\sgn}(\psi)(\psi)+ \sum_{\psi\in \nlp_{-m}^{(n)}(p,q)}  (-1)^{(n-1)}{\sgn}(\psi))(\psi)).     \end{equation} 
This gives an element $\lambda^{(n-1)}$ of the total complex of the double complex. 
The fact that $proj(\lambda)$ is a generator of $H^{n-1}(\BB^{(n)}(2))$ follows immediately as this projection is exactly the chain in $\BB^{(n)}(2)$ of degree $n-1$ corresponding to the generator in the cohomology group of the chains complex of the standard globular subdivision of the sphere $S^{n-1}$ (see \cite{ms}).    
\nl
Now we have to prove that $\lambda^{(n-1)}$ is a cocycle. 
\nl
Let $\psi\in \nlp_{+m}^{(n)}(p,q)$ be a normal path. Recall that the horizontal cosimplicial differential $\delta_s(\psi)$ adds $1$ to the corresponding label. Thus it creates a point labelled by $2.$ 
We have two possibilities for the location of this points. It either belongs to a horizontal piece of the path as in picture below or to a vertical piece. 
\beq\lb{dsp}
\xygraph{
!{<0cm,0cm>;<1cm,0cm>:<0cm,.8cm>::}
!{(-1,5)}*{\circ}  !{(0.5,5)}*{...}    !{(2,5)}*{\circ}    !{(3.5,5)}*{...}  !{(5,5)}*{\bullet}
!{(-1,4.1)}*{\vdots}  !{(2,4.1)}*{\vdots}     !{(5,4.1)}*{\vdots}
!{(-1,3)}*{\circ}  !{(0,3)}*{...}   !{(1.3,3)}*{}="1"  !{(2,3)}*{\bullet}="2" !{(2.7,3)}*{}="3"   !{(4,3)}*{...}  !{(5,3)}*{\circ}
!{(-1,1.5)}*{\vdots}  !{(2,1.5)}*+{\vdots} !{(5,1.5)}*{\vdots}
!{(-1,0)}*{\bullet} !{(0.5,0)}*{...}   !{(2,0)}*{\circ}     !{(3.5,0)}*{...}  !{(5,0)}*{\circ}
!{(2,-.3)}*{\scriptstyle i}  !{(5,-.3)}*{\scriptstyle p+1}  
!{(-1.3,3)}*{\scriptstyle j}  !{(-1.4,5)}*{\scriptstyle q+1}
!{(2,2.7)}*{\scriptstyle 2}
"1":"2"  "2":"3"   
}
\eeq
Assume the first (the argument in the second case is completely analogous). 
\nl
Let the coordinates of the point be $(i,j).$  It divides the path $\psi$ in to two parts. The first part is starting at $(0,0)$ and ending at $(i,j)$ and the second is going from $(i,j)$ to $(p+1,q+1).$ Both parts are normal paths. Suppose also that the complexity of the first paths is  $n_1$ and the second has complexity $n_2, \ n_1+n_2=n.$ 
By Lemma \ref{smooth-normal} applied to the first path we have that the number of non-zero stops along this path is $m_1 = (i-1)+(j-1) -n_1 +1.$ 
\nl
Hence, the $s= m_1+1 = i+j - n_1.$ And $\delta^s(\psi)$ enter the alternating sum with the coefficient $(-1)^{n-1}(-1)^{s} = (-1)^{n-n_1 + i + j -1} = (-1)^{n+ i + j -1}.$ The last equality takes place since $n_1$ is an even number. This in turn is because the path is even and approaches $(i,j)$ from the left. 
We thus described the effect of horizontal differential to $\psi.$ Hence, the overall sign of this path in the formula (\ref{formula2}) is $(-1)^{(p-1)(q-1) + n +i +j  -1}\sgn(\psi).$ 
\nl
The vertical differential applied to a path from  $\psi\in \nlp_{+m}^{(n)}(p+1,q)$ can also create a path as in \eqref{dsp} above and there is exactly one such normal path $\psi_i$. Locally, around the point $(i,j)$, it has the form     
$$\xygraph{
!{<0cm,0cm>;<1cm,0cm>:<0cm,.8cm>::}
!{(-1,5)}*{\circ}  !{(0.5,5)}*{...}    !{(2,5)}*{\circ}  !{(3,5)}*{\circ}    !{(4.5,5)}*{...}  !{(6,5)}*{\bullet}
!{(-1,4.1)}*{\vdots}  !{(2,4.1)}*{\vdots}  !{(3,4.1)}*{\vdots}     !{(6,4.1)}*{\vdots}
!{(-1,3)}*{\circ}  !{(0,3)}*{...}   !{(1.3,3)}*{}="1"  !{(2,3)}*{\bullet}="2" !{(3,3)}*{\bullet}="3" !{(3.7,3)}*{}="4"   !{(5,3)}*{...}  !{(6,3)}*{\circ}
!{(-1,1.5)}*{\vdots}  !{(2,1.5)}*+{\vdots}  !{(3,1.5)}*+{\vdots}  !{(6,1.5)}*{\vdots}
!{(-1,0)}*{\bullet} !{(0.5,0)}*{...}   !{(2,0)}*{\circ}   !{(3,0)}*{\circ}   !{(4.5,0)}*{...}  !{(6,0)}*{\circ}
!{(2,-.3)}*{\scriptstyle i} !{(3,-.3)}*{\scriptstyle i+1}  !{(6,-.3)}*{\scriptstyle p+2}  
!{(-1.3,3)}*{\scriptstyle j}  !{(-1.4,5)}*{\scriptstyle q+1}
"1":"2"  "2":"3" "3":"4"    
}$$
The differential $\partial^1_i$ sends this path to the path $\delta_s(\psi).$ We need to show that $\partial^1_i$ enters the alternating sum with the sign opposite to the one of $\delta_s(\psi).$ This sign is easy to compute. It is $(-1)^{p(q-1)}(-1)^{n+q-1}(-1)^i(-1)^{j}\sgn(\psi).$ This is because $\sgn(\psi_i) = (-1)^j\sgn(\psi)$ since the shuffle $\mu(\psi_i^\dag)$ contains extra $j$ inversions comparing to $\mu(\psi^\dag).$  Then the overall sign is $(-1)^{p(q-1)+n+q +i +j -1}\sgn(\psi)$ which is obviously opposite to  $(-1)^{(p-1)(q-1) + n +i +j  -1}\sgn(\psi).$   
\nl
We need to check that for odd paths a similar property holds. But this is reduced to the even paths again due to Lemma \ref{even=odd}. 

Finally the vertical differential in the bicomplex has components which reduce the number of corners in a lattice path or create corners with the label $1.$ It is not too hard (but rather long) to check that those components cancel each other (it is quite obvious modulo $2$ but requires some signs verification in the spirit above in general).  
\end{proof} 

Let now $E$ be an $(n-1)$-commutative cosimplicial monoid. 

\bth \label{bracketE} The Steenrod product $\scup_i$ on the normalised totalisation of $E$ is given by the formula:
$$a\scup_{i} b =  \sum_{\psi\in \slp_{+m}^{(i+1)}(p,q)} (-1)^{(p-1)(q-1)}{\sgn(\psi)}p^{(n)}(\psi) (a\otimes b),  \ 0\le i\le n-1 , a\in E(p),b\in E(q), $$
where $p^{(n)}(\psi) (a\otimes b)$ is the action of the operation $p^{(n)}(\psi)$ on $a\otimes b$ (see (\ref{action})). 
\nl
For  $n \ge 2$ the   degree $(1-n)$ bracket:
$$\beta^{(n-1)}: E(p)\otimes E(q) \to E(p+q-n+1)$$  
is given by   
$$\beta^{(n-1)}(a\otimes b) = a\scup_{n-1} b -  (-1)^{n+(p-1)(q-1)} b \scup_{n-1}a = $$ 
\begin{equation}
\label{formula}=   \sum_{\psi\in \slp_{+m}^{(n)}(p,q)} (-1)^{(p-1)(q-1)}{\sgn(\psi)}p^{(n)}(\psi) (a\otimes b)  +\sum_{\psi\in \slp_{-m}^{(n)}(p,q)}  (-1)^{(n-1)}{\sgn(\psi)}p^{(n)}(\psi)(a\otimes b) .
\end{equation} 
\eth
\begin{proof} 
Since the action of $\ll$ on $Tot(E)$ is factorised through $\mm$ by application of Lemma (\ref{smooth-normal}) we see that it is sufficient to show that if  a lattice path $\psi$ is normal but its projection is not a smooth path in $\MM$ then it is a degenerate element in $\mm$ and hence the action of $p^{(n)}(\psi)$ vanishes.
\nl
Indeed, $p^{(n)}(\psi)$ is a non-smooth path only if there are two successive points on the path $\psi$ labelled $0.$ In this case its image in $\mm$ is clearly in the image of a simplicial degeneracy which `expands' the corresponding row or column in the commutative lattice. 
\end{proof} 

\begin{example} Below we present a normal lattice path (on the left), whose image is not a smooth path. Solid dots correspond to the label $1$, empty dots to the label $0.$ 
The corresponding path (on the right) can be obtained by expansion of the first row and, hence, is degenerate.  
$$\xygraph{
!{<0cm,0cm>;<.7cm,0cm>:<0cm,.7cm>::}
!{(-2,2)}*{\circ} ="uull" !{(-1,2)}*{\circ} ="uul" !{(0,2)}*{\circ} ="uu"  !{(1,2)}*{\circ}="uur"  !{(2,2)}*{\bullet}="urr"
!{(-2,1)}*{\circ} ="ull" !{(-1,1)}*{\circ} ="ul" !{(0,1)}*{\circ} ="u"  !{(1,1)}*{\circ}="ur"  !{(2,1)}*{\bullet}="rur"
!{(-2,0)}*{\circ} ="ll" !{(-1,0)}*{\circ} ="l" !{(0,0)}*{\bullet} ="m"  !{(1,0)}*{\bullet}="r"  !{(2,0)}*{\circ}="rr"
!{(-2,-1)}*{\bullet} ="dll" !{(-1,-1)}*{\circ} ="dl" !{(0,-1)}*{\circ} ="d"  !{(1,-1)}*{\circ}="dr"  !{(2,-1)}*{\circ}="drr"
"dll":"dl"  "dl":"l"     "l":"m" 
"m":"r"     "r":"rr"  "rr":"rur" "rur":"urr"
}\qquad\qquad\qquad\qquad
\xygraph{
!{<0cm,0cm>;<.7cm,0cm>:<0cm,.7cm>::}
!{(-2,2)}*{\circ} ="uull" !{(-1,2)}*{\circ} ="uul" !{(0,2)}*{\circ} ="uu"  !{(1,2)}*{\circ}="uur"  !{(2,2)}*{\bullet}="urr"
!{(-2,1)}*{\circ} ="ull" !{(-1,1)}*{\circ} ="ul" !{(0,1)}*{\circ} ="u"  !{(1,1)}*{\circ}="ur"  !{(2,1)}*{\bullet}="rur"
!{(-2,0)}*{\circ} ="ll" !{(-1,0)}*{\circ} ="l" !{(0,0)}*{\bullet} ="m"  !{(1,0)}*{\bullet}="r"  !{(2,0)}*{\circ}="rr"
!{(-2,-1)}*{\bullet} ="dll" !{(-1,-1)}*{\circ} ="dl" !{(0,-1)}*{\circ} ="d"  !{(1,-1)}*{\circ}="dr"  !{(2,-1)}*{\circ}="drr"
"dll":"m"    
"m":"r"    "r":"rur" "rur":"urr"
}$$
\end{example} 

\begin{remark} The Theorem \ref{bracketE} shows that the total complex of an $(n-1)$-commutative monoid is somewhat special among algebras of the lattice path operad $\LL^{(n)}.$ For instance, there are more nontrivial terms in the classical formula for the Gershtenhaber bracket on Hoshcshild complex then for the analogous bracket on the deformation complex of a monoidal functor.  
\end{remark}

\subsection{Chain operations of low complexity}


 
In this section we write explicitly the Steenrod $\scup_i$-products on an $n$-cosimplicial monoid for $n=0,1$ and $2$ in terms of smooth path action. Cosimplicial monoids in this section means cosimplicial monoids in the symmetric monoidal category of modules over a commutative ring. We consider this category as a full monoidal subcategory of chain complexes 
concentrated in the degree $0,$ thus, all formulas from the previous section make sense.

The operation $\scup_0$ is defined in any cosimplicial monoid and we simply refer to it as the $\scup$-product.
The degree 1 bracket $\beta^{(1)}$ will also be denoted  $\{\da,\da\}$, while we will use $\bl\da,\da\br$ for the degree 2 bracket $\beta^{(2)}$. 
 \nl
 Consider first a very low dimensional example.
 \bex
Let $\delta:\langle 1\rangle\to \langle 1\rangle \times \langle 1\rangle.$ 
Graphically
$$\xygraph{
!{<0cm,0cm>;<1cm,0cm>:<0cm,1cm>::}
 !{(0,1)}*{\circ} ="u"  !{(1,1)}*{\circ}="ur"
 !{(0,0)}*{\circ} ="m"  !{(1,0)}*{\circ}="r"
"m":"ur" 
}$$
The path $\delta$ has just two liftings $\langle 1\rangle\to \langle 1\rangle \boxx \langle 1\rangle$
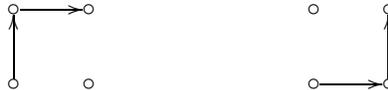
$$\xygraph{
!{<0cm,0cm>;<1cm,0cm>:<0cm,1cm>::}
 !{(0,1)}*{\circ} ="u"  !{(1,1)}*{\circ}="ur"
 !{(0,0)}*{\circ} ="m"  !{(1,0)}*{\circ}="r"
"m":"u"     "u":"ur" 
}\qquad\qquad
\qquad\qquad
\xygraph{
!{<0cm,0cm>;<1cm,0cm>:<0cm,1cm>::}
 !{(0,1)}*{\circ} ="u"  !{(1,1)}*{\circ}="ur"
 !{(0,0)}*{\circ} ="m"  !{(1,0)}*{\circ}="r"
"m":"r"     "r":"ur" 
}$$
which both have complexity $1$. This gives two operations  $(\delta,e) \ne (\delta,t)$ in $\M^{(0)}$. 
Let $E$ be a  cosimplicial monoid.
To read off the operations $E(0)\ot E(0)\to E(0)$ corresponding to $(\delta,e), (\delta,\tau)$ write $\delta:\langle 1\rangle\to \langle 1\rangle \times \langle 1\rangle$ as the composite 
$$ \xymatrix{\langle 1\rangle\ar[r]^(.34)\delta & \langle 1\rangle\times \langle 1\rangle \ar[r]^{(\widetilde\tau,\widetilde\pi)} & \langle 1\rangle\times \langle 1\rangle}$$
with $\widetilde\tau = \widetilde\pi = 1_{\langle 1\rangle}$. Clearly the Joyal dual maps are $\tau = \pi = 1_{[0]}$. 
Thus the operations $$E(0)\ot E(0)\to E(0)$$ are the direct and the reverse products in $E(0).$ The direct product coincides with the $\scup$-product in this case. 
If $E$ is a $1$-commutative cosimplicial monoid  $a\scup b = b\scup a$  since $(\delta,e)$ and $(\delta,t)$ are $1$-equivalent and so determine the same operation in $\MM^{(1)}.$  \eex

In general there exactly two  smooth  paths of linking number one on a rectangle (see formula (\ref{etanm}) in Section \ref{smooth paths lk=1.2}). 
 The path $\eta_{n,m}$ is  even and $\eta^t_{n,m}$ is odd.  Thus the cup product   
$$a\scup b =(-1)^{(n-1)(m-1)} (\partial_{n+m}...\partial_{n+2}\partial_{n+1})(a)\cdot(\partial_0)^n(b),\quad a\in E(n),\ b\in E(m)\ $$
as expected.
This is the only nontrivial Steenrod product in general (i.e. $0$-commutative) cosimplicial monoid.
\\

In a $1$-commutative monoid we will have one more product $a\cup_1b$ and, hence, the first nontrivial bracket operation.

\bex
Consider the map $\delta:\langle 2\rangle\to \langle 2\rangle \times \langle 2\rangle$, sending $0\mapsto (0,0),\ 1\mapsto (1,1),\ 2\mapsto (2,2)$. This is the only smooth path with linking number  $1$ on this square.
Graphically
$$\xygraph{
!{<0cm,0cm>;<1cm,0cm>:<0cm,1cm>::}
!{(-1,1)}*{\circ} ="ul" !{(0,1)}*{\circ} ="u"  !{(1,1)}*{\circ}="ur"
!{(-1,0)}*{\circ} ="l" !{(0,0)}*{\circ} ="m"  !{(1,0)}*{\circ}="r"
!{(-1,-1)}*{\circ} ="dl" !{(0,-1)}*{\circ} ="d"  !{(1,-1)}*{\circ}="dr"
!{(-1.15,-.85)}*+{\scriptstyle 1}  !{(-.15,.15)}*+{\scriptstyle 1} !{(.85,1.15)}*+{\scriptstyle 1}
"dl":"m"     "m":"ur" 
}$$
The map $\delta$ has four liftings $\langle 2\rangle\to \langle 2\rangle \boxx \langle 2\rangle$
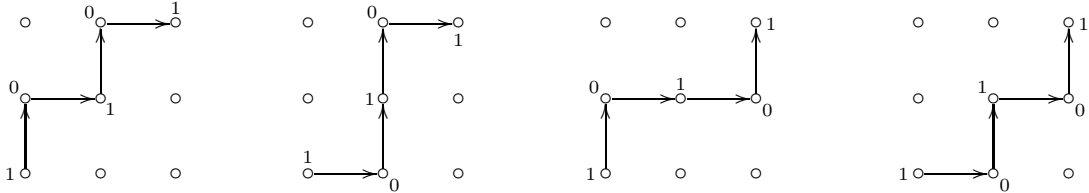
$$\xygraph{
!{<0cm,0cm>;<1cm,0cm>:<0cm,1cm>::}
!{(-1,1)}*{\circ} ="ul" !{(0,1)}*{\circ} ="u"  !{(1,1)}*{\circ}="ur"
!{(-1,0)}*{\circ} ="l" !{(0,0)}*{\circ} ="m"  !{(1,0)}*{\circ}="r"
!{(-1,-1)}*{\circ} ="dl" !{(0,-1)}*{\circ} ="d"  !{(1,-1)}*{\circ}="dr"
!{(-1.2,-1)}*+{\scriptstyle 1}   !{(.13,-.13)}*+{\scriptstyle 1} !{(1,1.2)}*+{\scriptstyle 1}
!{(-.15,1.15)}*+{\scriptstyle 0}  !{(-1.15,.15)}*+{\scriptstyle 0}
"dl":"l"     "l":"m" 
"m":"u"     "u":"ur" 
}\qquad\qquad
\xygraph{
!{<0cm,0cm>;<1cm,0cm>:<0cm,1cm>::}
!{(-1,1)}*{\circ} ="ul" !{(0,1)}*{\circ} ="u"  !{(1,1)}*{\circ}="ur"
!{(-1,0)}*{\circ} ="l" !{(0,0)}*{\circ} ="m"  !{(1,0)}*{\circ}="r"
!{(-1,-1)}*{\circ} ="dl" !{(0,-1)}*{\circ} ="d"  !{(1,-1)}*{\circ}="dr"
!{(-1,-.78)}*+{\scriptstyle 1}  !{(-.18,0)}*+{\scriptstyle 1} !{(1,.78)}*+{\scriptstyle 1}
!{(-.15,1.15)}*+{\scriptstyle 0}  !{(.15,-1.15)}*+{\scriptstyle 0}
"dl":"d"     "d":"m" 
"m":"u"     "u":"ur" 
}\qquad\qquad
\xygraph{
!{<0cm,0cm>;<1cm,0cm>:<0cm,1cm>::}
!{(-1,1)}*{\circ} ="ul" !{(0,1)}*{\circ} ="u"  !{(1,1)}*{\circ}="ur"
!{(-1,0)}*{\circ} ="l" !{(0,0)}*{\circ} ="m"  !{(1,0)}*{\circ}="r"
!{(-1,-1)}*{\circ} ="dl" !{(0,-1)}*{\circ} ="d"  !{(1,-1)}*{\circ}="dr"
!{(-1.2,-1)}*+{\scriptstyle 1}  !{(0,.2)}*+{\scriptstyle 1} !{(1.2,1)}*+{\scriptstyle 1}
!{(-1.15,.15)}*+{\scriptstyle 0}  !{(1.15,-.15)}*+{\scriptstyle 0}
"dl":"l"     "l":"m" 
"m":"r"     "r":"ur" 
}\qquad\qquad
\xygraph{
!{<0cm,0cm>;<1cm,0cm>:<0cm,1cm>::}
!{(-1,1)}*{\circ} ="ul" !{(0,1)}*{\circ} ="u"  !{(1,1)}*{\circ}="ur"
!{(-1,0)}*{\circ} ="l" !{(0,0)}*{\circ} ="m"  !{(1,0)}*{\circ}="r"
!{(-1,-1)}*{\circ} ="dl" !{(0,-1)}*{\circ} ="d"  !{(1,-1)}*{\circ}="dr"
!{(-1.2,-1)}*+{\scriptstyle 1}  !{(-.15,.15)}*+{\scriptstyle 1} !{(1.2,1)}*+{\scriptstyle 1}
!{(.15,-1.15)}*+{\scriptstyle 0}  !{(1.15,-.15)}*+{\scriptstyle 0}
"dl":"d"     "d":"m" 
"m":"r"     "r":"ur" 
}$$
which all have complexity $\geq 2$. 
Only two middle liftings have complexity $2$ so they come as  images of the elements of $\LL^{(2)}$ (which are exactly liftings $\psi_1$ and $\psi_2$ exhibited) and give two distinct elements $(\delta,e)$ and $(\delta,t)$ of $\MM^{(1)}.$ 
Of course,  $(\delta,e) = (\delta,t)$ in $\M^{(2)}$. 

Let now $E$ be a $1$-commutative cosimplicial monoid.

To read off the operations $E(1)\ot E(1)\to E(1)$ corresponding to $(\delta,e), (\delta,\tau)$ write $\delta:\langle 2\rangle\to \langle 2\rangle \times \langle 2\rangle$ as the composite 
$$ \xymatrix{\langle 2\rangle\ar[r]^(.34)\delta & \langle 2\rangle\times \langle 2\rangle \ar[r]^{(\widetilde\tau,\widetilde\pi)} & \langle 2\rangle\times \langle 2\rangle}$$
with $\widetilde\tau = \widetilde\pi = 1_{\langle 2\rangle}$. Clearly the Joyal dual maps are $\tau = \pi = 1_{[1]}$. 
Thus these two operations $$E(1)\ot E(1)\to E(1)$$ are the direct and the reverse products and the $a\scup_1 b = a b.$ 


Then the degree $1$ bracket $\beta^{(1)}(a\otimes b) =   \{ a, b \}$ on a $1$-commutative $E(1)$ coincides with the algebra commutator 
$$\{a,b\} =  a b - b  a\ .$$ 
Indeed, in this case we have only two summands  in the formula (\ref{formula}) which correspond to two middle paths. 

\eex

\bex Consider the $\beta^{(1)}:E(2)\otimes E(1) \to E(2).$ It is determined by liftings of complexity $2$ of the smooth paths
$$\rho:\langle 3\rangle\to \langle 3\rangle \times \langle 2\rangle.$$ 
There are exactly three paths like this all with a unique lifting 
$$\xygraph{
!{<0cm,0cm>;<1cm,0cm>:<0cm,1cm>::}
!{(-2,1)}*{\circ} ="ull" !{(-1,1)}*{\circ} ="ul" !{(0,1)}*{\circ} ="u"  !{(1,1)}*{\circ}="ur"  
!{(-2,0)}*{\circ} ="ll" !{(-1,0)}*{\circ} ="l" !{(0,0)}*{\circ} ="m"  !{(1,0)}*{\circ}="r"  
!{(-2,-1)}*{\circ} ="dll" !{(-1,-1)}*{\circ} ="dl" !{(0,-1)}*{\circ} ="d"  !{(1,-1)}*{\circ}="dr"  
!{(-3,.5)}*+{\rho_0\ =}
"dll":"dl"  "dl":"d"     "d":"m" 
"m":"u"     "u":"ur"  
"dl"-@{..>}"m"  "m"-@{..>}"ur"  
}\qquad\qquad\qquad\qquad
\xygraph{
!{<0cm,0cm>;<1cm,0cm>:<0cm,1cm>::}
!{(-1,1)}*{\circ} ="ul" !{(0,1)}*{\circ} ="u"  !{(1,1)}*{\circ}="ur"  !{(2,1)}*{\circ}="rur"  
!{(-1,0)}*{\circ} ="l" !{(0,0)}*{\circ} ="m"  !{(1,0)}*{\circ}="r"  !{(2,0)}*{\circ}="rr"  
!{(-1,-1)}*{\circ} ="dl" !{(0,-1)}*{\circ} ="d"  !{(1,-1)}*{\circ}="dr"  !{(2,-1)}*{\circ}="drr"  
!{(-2,.5)}*+{\rho_2\ =}
"dl":"d"     "d":"m" 
"m":"u"     "u":"ur"  "ur":"rur" 
"dl"-@{..>}"m"  "m"-@{..>}"ur"   
}
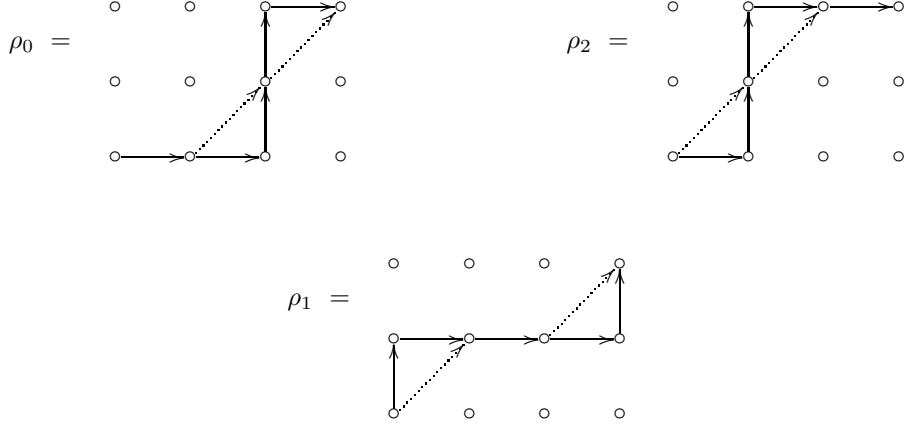
$$
\bigskip\bigskip
$$\xygraph{
!{<0cm,0cm>;<1cm,0cm>:<0cm,1cm>::}
!{(-2,1)}*{\circ} ="ull"  !{(-1,1)}*{\circ} ="ul" !{(0,1)}*{\circ} ="u"  !{(1,1)}*{\circ}="ur"  
!{(-2,0)}*{\circ} ="ll"  !{(-1,0)}*{\circ} ="l" !{(0,0)}*{\circ} ="m"  !{(1,0)}*{\circ}="r"  
!{(-2,-1)}*{\circ} ="dll" !{(-1,-1)}*{\circ} ="dl" !{(0,-1)}*{\circ} ="d"  !{(1,-1)}*{\circ}="dr"  
!{(-3,.5)}*+{\rho_1\ =}
"dll":"ll"  "ll":"l"     "l":"m" 
"m":"r"     "r":"ur" 
"dll"-@{..>}"l"  "m"-@{..>}"ur"   
}$$
 
The first path has $\varpi(\rho_0) = (12)$ and $\mu(\rho_0^\dag) = (12453)$ even. So its entry to the formula is $a\partial_0(b).$ For the second path we again have the corresponding shuffle even and its entry is $a\partial_2(b).$ Finally for the third path the $\varpi(\rho_1) = (21)$ and $\mu(\rho_1^\dag) = (41235)$ is odd so it enters with $(-1)^0(-1)\partial_1(b)a.$ Thus the bracket operation is
$$\{a, b\} = a(\partial_0(b)+\partial_2(b)) - \partial_1(b) a .$$

\eex 

\bigskip

In general  we get the following formula for  $\beta^{(1)}: E(p)\otimes  E(q)\to E(p+q-1).$ A typical even smooth lattice path of complexity  two  is shown in (\ref{etanmi}).  

It is not hard to see that the parity of the corresponding shuffle is $(-1)^{(q+1)(p-i+1)} = (-1)^{(q+1)(p+1)-i(q+1)}$.
 Let us denote its action on $a\otimes b$ by $a\circ_i b.$ 
 We also have corresponding odd paths whose action on $a\otimes b$ we  denote by $b\circ_ i a.$ 
 Then the formula (\ref{formula}) gives:  
 $$\{a, b\} = \sum_{i=1}^p (-1)^{(p-1)(q-1)}(-1)^{(q+1)(p+1)-i(q+1)}a\circ_i b - \sum_{i=1}^q (-1)^{(q+1)(p+1)-i(p+1)}b\circ_i a.$$ 
We then obtain:
\begin{equation}\label{ger} \{a, b\} = \sum_{i=1}^p (-1)^{i(q-1)}a\circ_i b - (-1)^{(p-1)(q-1)}\sum_{i=1}^q (-1)^{i(p-1)}b\circ_i a.\end{equation} 

\begin{remark} To get  a bracket on the Hochschild complex of an algebra we can apply Theorem \ref{formulal} instead
and obtain exactly the classical  formula for Gershtenhaber bracket on Hochschild complex. 
\end{remark} 




Finally, consider the case of $2$-commutative cosimplicial monoids. 

\bex
All liftings of the map $\delta:\langle 3\rangle\to \langle 3\rangle \times \langle 3\rangle$
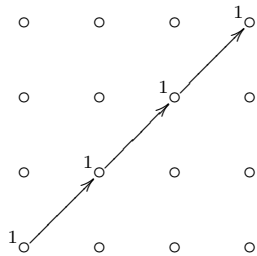
$$\xygraph{
!{<0cm,0cm>;<1cm,0cm>:<0cm,1cm>::}
!{(-1,2)}*{\circ} ="uul" !{(0,2)}*{\circ} ="uu"  !{(1,2)}*{\circ}="uur"  !{(2,2)}*{\circ}="urr"
!{(-1,1)}*{\circ} ="ul" !{(0,1)}*{\circ} ="u"  !{(1,1)}*{\circ}="ur"  !{(2,1)}*{\circ}="rur"
!{(-1,0)}*{\circ} ="l" !{(0,0)}*{\circ} ="m"  !{(1,0)}*{\circ}="r"  !{(2,0)}*{\circ}="rr"
!{(-1,-1)}*{\circ} ="dl" !{(0,-1)}*{\circ} ="d"  !{(1,-1)}*{\circ}="dr"  !{(2,-1)}*{\circ}="drr"
!{(-1.15,-.85)}*+{\scriptstyle 1}  !{(-.15,.15)}*+{\scriptstyle 1} !{(.85,1.15)}*+{\scriptstyle 1}  !{(1.85,2.15)}*+{\scriptstyle 1}
"dl":"m"     "m":"ur"  "ur":"urr"
}$$
to $\langle 3\rangle\to \langle 3\rangle \boxx\langle 3\rangle$ have complexity $\geq 3$. 
The complexity $3$ liftings are
$$\xygraph{
!{<0cm,0cm>;<1cm,0cm>:<0cm,1cm>::}
!{(-1,2)}*{\circ} ="uul" !{(0,2)}*{\circ} ="uu"  !{(1,2)}*{\circ}="uur"  !{(2,2)}*{\circ}="urr"
!{(-1,1)}*{\circ} ="ul" !{(0,1)}*{\circ} ="u"  !{(1,1)}*{\circ}="ur"  !{(2,1)}*{\circ}="rur"
!{(-1,0)}*{\circ} ="l" !{(0,0)}*{\circ} ="m"  !{(1,0)}*{\circ}="r"  !{(2,0)}*{\circ}="rr"
!{(-1,-1)}*{\circ} ="dl" !{(0,-1)}*{\circ} ="d"  !{(1,-1)}*{\circ}="dr"  !{(2,-1)}*{\circ}="drr"
!{(-1,-.78)}*+{\scriptstyle 1}  !{(-.18,0)}*+{\scriptstyle 1} !{(1,.78)}*+{\scriptstyle 1}  !{(2.2,2)}*+{\scriptstyle 1}
!{(-.15,1.15)}*+{\scriptstyle 0}  !{(.15,-1.15)}*+{\scriptstyle 0}  !{(2.15,.85)}*+{\scriptstyle 0}
"dl":"d"     "d":"m" 
"m":"u"     "u":"ur"  "ur":"rur" "rur":"urr"
}\qquad\qquad\qquad\qquad
\xygraph{
!{<0cm,0cm>;<1cm,0cm>:<0cm,1cm>::}
!{(-1,2)}*{\circ} ="uul" !{(0,2)}*{\circ} ="uu"  !{(1,2)}*{\circ}="uur"  !{(2,2)}*{\circ}="urr"
!{(-1,1)}*{\circ} ="ul" !{(0,1)}*{\circ} ="u"  !{(1,1)}*{\circ}="ur"  !{(2,1)}*{\circ}="rur"
!{(-1,0)}*{\circ} ="l" !{(0,0)}*{\circ} ="m"  !{(1,0)}*{\circ}="r"  !{(2,0)}*{\circ}="rr"
!{(-1,-1)}*{\circ} ="dl" !{(0,-1)}*{\circ} ="d"  !{(1,-1)}*{\circ}="dr"  !{(2,-1)}*{\circ}="drr"
!{(-1.2,-1)}*+{\scriptstyle 1}  !{(0,.2)}*+{\scriptstyle 1} !{(1.2,1)}*+{\scriptstyle 1}  !{(2,2.2)}*+{\scriptstyle 1}
!{(-1.15,.15)}*+{\scriptstyle 0}  !{(1.15,-.15)}*+{\scriptstyle 0}  !{(.85,2.15)}*+{\scriptstyle 0}
"dl":"l"     "l":"m" 
"m":"r"     "r":"ur"  "ur":"uur"  "uur":"urr"
}$$
The corresponding operations $E(2)\ot E(2)\to E(2)$ on a 2-commutative cosimplicial monoid 
are the direct and the reverse product and $$a\scup_2 b = (-1)^{(2-1)(2-1)}ab = - a b.$$ 
\nl
The degree 2 bracket $\beta^{(2)}:E(2)\ot E(2)\to E(2)$ on a $2$-commutative $E$ coincides with the opposite to the algebra commutator 
$$\bl a,b\br =  ba - ab\ .$$
\eex
\bigskip

\bex
There are four smooth lattice paths $\langle 4\rangle\to \langle 4\rangle \boxx \langle 3\rangle$ of complexity $3$: 
$$\xygraph{
!{<0cm,0cm>;<1cm,0cm>:<0cm,1cm>::}
!{(-2,2)}*{\circ} ="uull" !{(-1,2)}*{\circ} ="uul" !{(0,2)}*{\circ} ="uu"  !{(1,2)}*{\circ}="uur"  !{(2,2)}*{\circ}="urr"
!{(-2,1)}*{\circ} ="ull" !{(-1,1)}*{\circ} ="ul" !{(0,1)}*{\circ} ="u"  !{(1,1)}*{\circ}="ur"  !{(2,1)}*{\circ}="rur"
!{(-2,0)}*{\circ} ="ll" !{(-1,0)}*{\circ} ="l" !{(0,0)}*{\circ} ="m"  !{(1,0)}*{\circ}="r"  !{(2,0)}*{\circ}="rr"
!{(-2,-1)}*{\circ} ="dll" !{(-1,-1)}*{\circ} ="dl" !{(0,-1)}*{\circ} ="d"  !{(1,-1)}*{\circ}="dr"  !{(2,-1)}*{\circ}="drr"
!{(-3,.5)}*+{\rho_0\ =}
"dll":"dl"  "dl":"d"     "d":"m" 
"m":"u"     "u":"ur"  "ur":"rur" "rur":"urr"
"dl"-@{..>}"m"  "m"-@{..>}"ur"   "ur"-@{..>}"urr"   
}\qquad\qquad\qquad\qquad
\xygraph{
!{<0cm,0cm>;<1cm,0cm>:<0cm,1cm>::}
!{(-1,2)}*{\circ} ="uul" !{(0,2)}*{\circ} ="uu"  !{(1,2)}*{\circ}="uur"  !{(2,2)}*{\circ}="urr"  !{(3,2)}*{\circ}="urrr"
!{(-1,1)}*{\circ} ="ul" !{(0,1)}*{\circ} ="u"  !{(1,1)}*{\circ}="ur"  !{(2,1)}*{\circ}="rur"  !{(3,1)}*{\circ}="rurr"
!{(-1,0)}*{\circ} ="l" !{(0,0)}*{\circ} ="m"  !{(1,0)}*{\circ}="r"  !{(2,0)}*{\circ}="rr"  !{(3,0)}*{\circ}="rrr"
!{(-1,-1)}*{\circ} ="dl" !{(0,-1)}*{\circ} ="d"  !{(1,-1)}*{\circ}="dr"  !{(2,-1)}*{\circ}="drr"  !{(3,-1)}*{\circ}="drrr"
!{(-2,.5)}*+{\rho_2\ =}
"dl":"d"     "d":"m" 
"m":"u"     "u":"ur"  "ur":"rur" "rur":"rurr"  "rurr":"urrr"
"dl"-@{..>}"m"  "m"-@{..>}"ur"   "rur"-@{..>}"urrr"   
}$$
\bigskip\bigskip
$$\xygraph{
!{<0cm,0cm>;<1cm,0cm>:<0cm,1cm>::}
!{(-2,2)}*{\circ} ="uull"  !{(-1,2)}*{\circ} ="uul" !{(0,2)}*{\circ} ="uu"  !{(1,2)}*{\circ}="uur"  !{(2,2)}*{\circ}="urr"
!{(-2,1)}*{\circ} ="ull"  !{(-1,1)}*{\circ} ="ul" !{(0,1)}*{\circ} ="u"  !{(1,1)}*{\circ}="ur"  !{(2,1)}*{\circ}="rur"
!{(-2,0)}*{\circ} ="ll"  !{(-1,0)}*{\circ} ="l" !{(0,0)}*{\circ} ="m"  !{(1,0)}*{\circ}="r"  !{(2,0)}*{\circ}="rr"
!{(-2,-1)}*{\circ} ="dll" !{(-1,-1)}*{\circ} ="dl" !{(0,-1)}*{\circ} ="d"  !{(1,-1)}*{\circ}="dr"  !{(2,-1)}*{\circ}="drr"
!{(-3,.5)}*+{\rho_1\ =}
"dll":"ll"  "ll":"l"     "l":"m" 
"m":"r"     "r":"ur"  "ur":"uur"  "uur":"urr"
"dll"-@{..>}"l"  "m"-@{..>}"ur"   "ur"-@{..>}"urr"   
}\qquad\qquad\qquad\qquad
\xygraph{
!{<0cm,0cm>;<1cm,0cm>:<0cm,1cm>::}
!{(-1,2)}*{\circ} ="uul" !{(0,2)}*{\circ} ="uu"  !{(1,2)}*{\circ}="uur"  !{(2,2)}*{\circ}="urr"  !{(3,2)}*{\circ}="urrr"
!{(-1,1)}*{\circ} ="ul" !{(0,1)}*{\circ} ="u"  !{(1,1)}*{\circ}="ur"  !{(2,1)}*{\circ}="rur"  !{(3,1)}*{\circ}="rurr"
!{(-1,0)}*{\circ} ="l" !{(0,0)}*{\circ} ="m"  !{(1,0)}*{\circ}="r"  !{(2,0)}*{\circ}="rr"  !{(3,0)}*{\circ}="rrr"
!{(-1,-1)}*{\circ} ="dl" !{(0,-1)}*{\circ} ="d"  !{(1,-1)}*{\circ}="dr"  !{(2,-1)}*{\circ}="drr"  !{(3,-1)}*{\circ}="drrr"
!{(-2,.5)}*+{\rho_3\ =}
"dl":"l"     "l":"m" 
"m":"r"     "r":"ur"  "ur":"uur"  "uur":"urr"  "urr":"urrr"
"dl"-@{..>}"m"  "m"-@{..>}"ur"   "ur"-@{..>}"urr"   
}$$
The dotted paths represent corresponding smooth paths.  
Thus the corresponding operations $E(3)\ot E(2)\to E(3)$ on a $2$-commutative cosimplicial monoid $E$
are sending $a\otimes b\in E(3)\ot E(2)$ to 
$\ a\partial_0(b),\ a\partial_2(b),\  \partial_1(b)a,\ \partial_3(b)a$ respectively.
\nl
The degree 2 bracket $\beta^{(2)}:E(3)\ot E(2)\to E(3)$  is given by 
\beq\lb{23b}\bl a, b\br = \beta^{(2)}(a\otimes b) = a(\partial_0(b) + \partial_2(b)) - (\partial_1(b)+\partial_3(b))a\ .\eeq
\eex
\bigskip

\bex
There are eight smooth lattice paths $\langle 5\rangle\to \langle 4\rangle \boxx\langle 4\rangle$ of complexity $3$, four of each are even: 
$$\xygraph{
!{<0cm,0cm>;<1cm,0cm>:<0cm,1cm>::}
!{(-2,2)}*{\circ} ="uuull" !{(-1,2)}*{\circ} ="uuul" !{(0,2)}*{\circ} ="uuu"  !{(1,2)}*{\circ}="uuur"  !{(2,2)}*{\circ}="uurr"
!{(-2,1)}*{\circ} ="uull" !{(-1,1)}*{\circ} ="uul" !{(0,1)}*{\circ} ="uu"  !{(1,1)}*{\circ}="uur"  !{(2,1)}*{\circ}="urr"
!{(-2,0)}*{\circ} ="ull" !{(-1,0)}*{\circ} ="ul" !{(0,0)}*{\circ} ="u"  !{(1,0)}*{\circ}="ur"  !{(2,0)}*{\circ}="rur"
!{(-2,-1)}*{\circ} ="ll" !{(-1,-1)}*{\circ} ="l" !{(0,-1)}*{\circ} ="m"  !{(1,-1)}*{\circ}="r"  !{(2,-1)}*{\circ}="rr"
!{(-2,-2)}*{\circ} ="dll" !{(-1,-2)}*{\circ} ="dl" !{(0,-2)}*{\circ} ="d"  !{(1,-2)}*{\circ}="dr"  !{(2,-2)}*{\circ}="drr"
!{(-3,0)}*+{\rho_{40}\ =}
"dll":"dl"  "dl":"d"     "d":"m" 
"m":"u"     "u":"ur"  "ur":"rur" "rur":"urr"  "urr":"uurr"
"dl"-@{..>}"m"  "m"-@{..>}"ur"   "ur"-@{..>}"urr"   
}\qquad\qquad\qquad\qquad
\xygraph{
!{<0cm,0cm>;<1cm,0cm>:<0cm,1cm>::}
!{(-1,2)}*{\circ} ="uul" !{(0,2)}*{\circ} ="uu"  !{(1,2)}*{\circ}="uur"  !{(2,2)}*{\circ}="urr"  !{(3,2)}*{\circ}="urrr"
!{(-1,1)}*{\circ} ="ul" !{(0,1)}*{\circ} ="u"  !{(1,1)}*{\circ}="ur"  !{(2,1)}*{\circ}="rur"  !{(3,1)}*{\circ}="rurr"
!{(-1,0)}*{\circ} ="l" !{(0,0)}*{\circ} ="m"  !{(1,0)}*{\circ}="r"  !{(2,0)}*{\circ}="rr"  !{(3,0)}*{\circ}="rrr"
!{(-1,-1)}*{\circ} ="ml" !{(0,-1)}*{\circ} ="md"  !{(1,-1)}*{\circ}="mdr"  !{(2,-1)}*{\circ}="mdrr"  !{(3,-1)}*{\circ}="mdrrr"
!{(-1,-2)}*{\circ} ="dl" !{(0,-2)}*{\circ} ="d"  !{(1,-2)}*{\circ}="dr"  !{(2,-2)}*{\circ}="drr"  !{(3,-2)}*{\circ}="drrr"
!{(-2,0)}*+{\rho_{13}\ =}
"dl":"d"     "d":"md" 
"md":"m"  "m":"u"   "u":"ur"  "ur":"rur" "rur":"rurr"  "rurr":"urrr"
"dl"-@{..>}"md"  "m"-@{..>}"ur"   "rur"-@{..>}"urrr"   
}
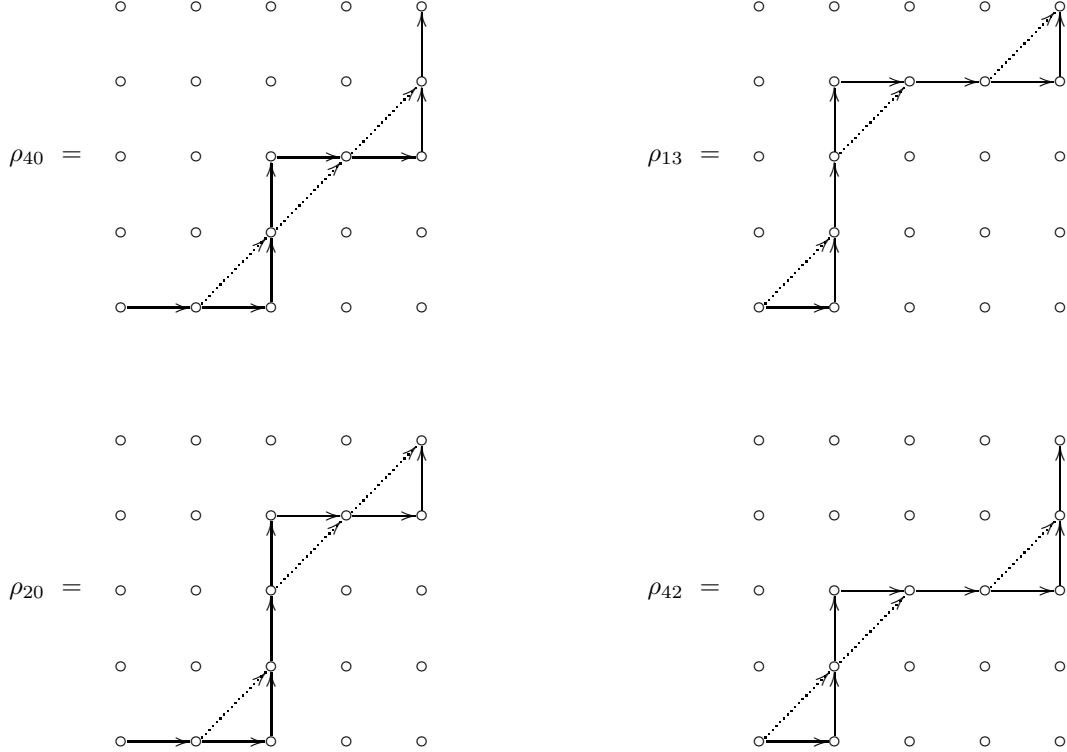
$$
\bigskip\bigskip\bigskip
$$\xygraph{
!{<0cm,0cm>;<1cm,0cm>:<0cm,1cm>::}
!{(-1,3)}*{\circ} ="uuull" !{(0,3)}*{\circ} ="uuul" !{(1,3)}*{\circ} ="uuu"  !{(2,3)}*{\circ}="uuur"  !{(3,3)}*{\circ}="uurr"
!{(-1,2)}*{\circ} ="uul" !{(0,2)}*{\circ} ="uu"  !{(1,2)}*{\circ}="uur"  !{(2,2)}*{\circ}="urr"  !{(3,2)}*{\circ}="urrr"
!{(-1,1)}*{\circ} ="ul" !{(0,1)}*{\circ} ="u"  !{(1,1)}*{\circ}="ur"  !{(2,1)}*{\circ}="rur"  !{(3,1)}*{\circ}="rurr"
!{(-1,0)}*{\circ} ="l" !{(0,0)}*{\circ} ="m"  !{(1,0)}*{\circ}="r"  !{(2,0)}*{\circ}="rr"  !{(3,0)}*{\circ}="rrr"
!{(-1,-1)}*{\circ} ="dl" !{(0,-1)}*{\circ} ="d"  !{(1,-1)}*{\circ}="dr"  !{(2,-1)}*{\circ}="drr"  !{(3,-1)}*{\circ}="drrr"
!{(-2,1)}*+{\rho_{20}\ =}
"dl":"d"     "d":"dr" 
"dr":"r"     "r":"ur"  "ur":"uur" "uur":"urr"  "urr":"urrr"  "urrr":"uurr"
"d"-@{..>}"r"  "ur"-@{..>}"urr"   "urr"-@{..>}"uurr"   
}\qquad\qquad\qquad\qquad
\xygraph{
!{<0cm,0cm>;<1cm,0cm>:<0cm,1cm>::}
!{(-1,3)}*{\circ} ="uuull" !{(0,3)}*{\circ} ="uuul" !{(1,3)}*{\circ} ="uuu"  !{(2,3)}*{\circ}="uuur"  !{(3,3)}*{\circ}="uurr"
!{(-1,2)}*{\circ} ="uul" !{(0,2)}*{\circ} ="uu"  !{(1,2)}*{\circ}="uur"  !{(2,2)}*{\circ}="urr"  !{(3,2)}*{\circ}="urrr"
!{(-1,1)}*{\circ} ="ul" !{(0,1)}*{\circ} ="u"  !{(1,1)}*{\circ}="ur"  !{(2,1)}*{\circ}="rur"  !{(3,1)}*{\circ}="rurr"
!{(-1,0)}*{\circ} ="l" !{(0,0)}*{\circ} ="m"  !{(1,0)}*{\circ}="r"  !{(2,0)}*{\circ}="rr"  !{(3,0)}*{\circ}="rrr"
!{(-1,-1)}*{\circ} ="dl" !{(0,-1)}*{\circ} ="d"  !{(1,-1)}*{\circ}="dr"  !{(2,-1)}*{\circ}="drr"  !{(3,-1)}*{\circ}="drrr"
!{(-2,1)}*+{\rho_{42}\ =}
"dl":"d"     "d":"m" 
"m":"u"     "u":"ur"  "ur":"rur" "rur":"rurr"  "rurr":"urrr"  "urrr":"uurr"
"dl"-@{..>}"m"  "m"-@{..>}"ur"   "rur"-@{..>}"urrr"   
}$$
Another four odd paths can be obtained from the above by rotation by $180^\circ$ and the orientation change (which also coincides with the symmetry along the diagonal in this case).
The corresponding even operations $E(3)\ot E(3)\to E(4)$ on a $2$-commutative cosimplicial monoid $E$
are sending $a\otimes b\in E(3)\ot E(3)$ to 
$$\partial_2(a)\partial_0(b),\quad \partial_4(a)\partial_0(b),\quad \partial_4(a)\partial_2(b),\quad \partial_1(a)\partial_3(b)$$ respectively.
\nl
The degree 2 bracket $\beta^{(2)}:E(3)\ot E(3)\to E(4)$ is given by 
$$\bl a, b\br = \partial_2(a)\partial_0(b) + \partial_4(a)\partial_0(b) + \partial_4(a)\partial_2(b) - \partial_1(a)\partial_3(b) +$$
\beq\lb{33b}+ \partial_2(b)\partial_0(a) + \partial_4(b)\partial_0(a) + \partial_4(b)\partial_2(a) - \partial_1(b)\partial_3(a)\ .\eeq
\eex

\subsection{Symmetric cosimplicial monoids and Hodge decomposition}\lb{scm}

Let $\Fin$ be the skeletal category of pointed finite sets. The objects of it are finite ordinals $[n] = \{0,\ldots,n\},$ where we consider $0$ as a based point, and morphisms are any maps which preserve base points. There is a functor $D:\mathbf{Int}\to \Fin$ defined by: $$D(\langle n\rangle) = [n-1]$$ and for an interval map $\phi:\langle n\rangle \to \langle m \rangle:$  $$D(\phi)(i) = \phi(i) \ \mbox{if} \ \phi(i)\ne m \ , \ \mbox{and} \ D(\phi)(i) = 0 \ \mbox{otherwise}.$$   
Let $\Ga = \Fin^{op}$ be the Segal category.
    We have then the following functor $\mathbf{C}^{op}$ :
$$\De\stackrel{\widetilde{(-)}}{\to}\mathbf{Int}^{op}\stackrel{D^{op}}{\to}\Ga.$$ 
\bre
 This functor is the opposite to the functor $\mathbf{C}$ described by Loday in \cite[p.221]{lo}. Hence, the name.
 \ere 
  
\bde
A {\em symmetric cosimplicial monoid} is a functor $E:\Ga \to  \bMon(\bV).$ 
\ede
\bpr
A symmetric cosimplicial monoid is a cosimplicial monoid $E:\De\to \bMon(\bV)$ (the {\em underlying cosimplicial monoid}) equipped with the following extra structure. 
The components $E(n)$ come equipped with symmetric group actions $S_n\to Aut_{mon}(E(n))$ by monoid automorphisms, such that 
$$t_n...t_1\partial_0 = \partial_{n+1}\qquad\qquad t_i\partial_i = \partial_i,$$
and
$$ \partial_i t_j = \left\{\begin{array}{cc} t_j\partial_i & i>j+1\\ 
t_{i-1}t_i\partial_{i-1} & i=j+1\\
t_{i+1}t_i\partial_{i+1} & i=j\\
t_{j+1}\partial_i & i<j\\ 
\end{array}\right.\qquad \qquad
\sigma_i t_j = \left\{\begin{array}{cc} t_j\sigma_i & i>j+1\\ 
\sigma_{i-1} & i=j+1\\
\sigma_{i+1} & i=j\\
t_{j-1}\sigma_i & i<j\\ 
\end{array}\right.\quad ,$$
where $t_i $ is the transposition $(i, i+1).$
\epr
\bpf The restriction along $\mathbf{C}^{op}$ provide the cosimplicial structure. The rest of the structure is a presentation of $\Ga$ in terms of generators and relations.

\epf

\bre One can also consider braided cosimplicial monoids as functors $E:\mathbf{\Upsilon}\to \bMon(\bV)$ where $\mathbf{\Upsilon} = \mathbf{Vin}_*^{op}$ is the opposite to the category of pointed vines \cite{lavers}. The presentation of braided cosimplicial monoids is similar to the symmetric case. It is sufficient to replace symmetric groups by braid groups.   
\ere

\bde A symmetric (braided) cosimplicial monoid is {\em $n$-commutative} if its underlying cosimplicial monoid is $n$-commutative.
\ede

Let now $\bV$ be a symmetric tensor category over a field of characteristic zero.  
\nl
Dualising the arguments from \cite[sect. 6.4]{lo} we can see that the symmetric group actions on the components $E(n)$ of a symmetric cosimplicial monoid give rise to the so-called {\em Hodge decomposition} of the cohomology
$$H^n(E)\ =\ H^{n,0}(E)\ \op\ ...\ \op\ H^{n,n}(E)\ .$$ 
The decomposition is compatible with the cup-product
$$\cup: H^{m,s}(E)\ot H^{n,t}(E) \to H^{m+n,s+t}(E)\ .$$
In particular, the top components $H^{n,n}(E)$ is the cohomology of the sub-complex of $C^*(E)$ of its anti-symmetric elements.
Recall that an element $a\in E(n)$ is {\em anti-symmetric} if $t_i(a) = -a$ for $i=1,...,n-1$, equivalently $\sigma(a) = {\sgn(\sigma)}a$ for an arbitrary permutation $\sigma\in S_n$.

Here we give a more explicit description of $H^{n,n}(E)$. 
For an element $a\in E(n)$ of a symmetric cosimplicial monoid denote 
$$\hat a_i = t_{i}...t_1\partial_0(a)\in E(n+1)\ .$$ 
For example $\hat a_0 = \partial_0(a)$ and $\hat a_n = \partial_{n+1}(a)$. 
\bde 
We say that an element $a\in E(n)$ of a symmetric cosimplicial monoid is {\em poly-primitive} if 
$$\partial_i(a) = \hat a_{i-1} + \hat a_i ,\qquad i=1,...,n\ .$$ 
\ede
Denote by $P^n(E)\subset E(n)$ the subspace of anti-symmetric poly-primitive elements. 
\bth\lb{pac}
Let $E$ be a symmetric cosimplicial monoid in a symmetric tensor category $\bV$ over a field of characteristic zero.
Then
\begin{enumerate} \item
The  top component of the Hodge decomposition is the subspace of anti-symmetric poly-primitive elements
$$H^{n,n}(E)\ =\ P^n(E)\ .$$
\item If $E$ is  $k$-commutative for $k\ge 2$
then the top degree component of the bracket $\beta^{(k)}$ 
$$P^m(E)\otimes P^n(E)\to P^{m+n-k}(E)$$ 
is zero.
\end{enumerate} 
\eth
\bpf
First note that poly-primitive elements of a symmetric cosimplicial monoid are cocycles of its cochain complex.
Indeed
$$\partial(a) = \partial_0(a) - \partial_1(a) + ... + (-1)^n\partial_n(a) + (-1)^{n+1}\partial_{n+1}(a) = $$
$$\hat a_0 - (\hat a_0+\hat a_1) + ... + (-1)^n(\hat a_{n-1}+\hat a_n) + (-1)^{n+1}\hat a_n = 0\ .$$
In particular, $P^n(E)\subset Z^{n,n}(E)$.
\nl
Now note that $P^n(E)\cap B^n(E) = 0$.
Indeed, observe that $\alt_n\partial=0$, where 
$\alt_n:E(n)\to E(n)$ is the anti-symmetrisation $$\alt_n = \frac{\sum_{\sigma\in S_n}{\sgn}(\sigma)\sigma}{n!}.$$
Since $t_i\partial_i=\partial_i$ we have $\alt_n\partial_i=0$ for $i=1,...,n$. 
Since $t_n...t_1\partial_0 = \partial_{n+1}$ we have that $\alt_n\partial_0= (-1)^n \alt_n t_n...t_1\partial_0 = (-1)^n \alt_n \partial_{n+1}$, which  gives $\alt_n\partial=0$.
Since  $a = \alt_n(a)$ for any $a\in P^n(E)$, writing $a=\partial(b)$ for $a\in P^n(E)\cap B^n(E)$, we get $a = \alt_n(a) = \alt_n\partial(b)=0$. 
\nl
Finally the anti-symmetrisation of a cocycle is poly-primitive. That follows from the formula
$$(\partial_0-\partial_1+t_1\partial_0)\sum_{\sigma\in S_n}{\sgn}(\sigma)\sigma = \left(\sum_{\sigma\in S_{n+1}, \sigma^{-1}(1)<\sigma^{-1}(2)}{\sgn}(\sigma)\sigma\right)\sum_{i=0}^{n+1}(-1)^i\partial_i\ .$$
For example,
$$(\partial_0-\partial_1+t_1\partial_0)(1-t_1) = (1 + t_2t_1-t_2)(\partial_0-\partial_1+\partial_2-\partial_3)\ .$$
Thus the anti-symmetrisation induces a map $Z^n(E)\to P^n(E)$ giving an identification $P^n(E) = H^{n,n}(E)$. 

For a proof of the second part of the Theorem observe that for a $k$-commutative $E$ we have
$\alt_{m+n-k}\beta^{(k)}(a, b) = 0, \quad a\in P^m(E),\ b\in P^n(E) . $ 
Here are  some low dimensional examples of this behavior.  
\nl
Let $a,b\in P^2$. Then $t_1(ab) = t_1(a)t_1(b) = ab$ and so $\alt_2(ab) = 0$. Thus 
$$\alt_2\bl a, b\br = \alt_2[a,b] = \alt_2(ab-ba) = 0\ .$$
\nl
For $a\in P^2,\ b\in P^3$ the 2-bracket \eqref{23b} is the commutator
$$\bl a, b\br = [\hat a_0+\hat a_1+\hat a_2,b]\ .$$
Since $t_2(\hat a_0b) = \hat a_0b$ we have $\alt_3[\hat a_0,b]=0$. Similarly for the other terms
$$\alt_3\bl a, b\br = \alt_3[\hat a_0+\hat a_1+\hat a_2,b] = \alt_3[\hat a_0,b]+\alt_3[\hat a_1,b]+\alt_3[\hat a_2,b] = 0\ .$$
Let $a,b\in P^3$. The 2-bracket \eqref{33b} takes the form
$$\bl a, b\br = \hat a_1\hat b_0 + \hat a_2\hat b_0 + \hat a_3\hat b_0 + \hat a_3\hat b_1 + \hat a_3\hat b_2 - \hat a_0\hat b_2 - \hat a_0\hat b_3 - \hat a_1\hat b_2 - \hat a_1\hat b_3 + $$
$$+ \hat b_1\hat a_0 + \hat b_2\hat a_0 + \hat b_3\hat a_0 + \hat b_3\hat a_1 + \hat b_3\hat a_2 - \hat b_0\hat a_2 - \hat b_0\hat a_3 - \hat b_1\hat a_2 - \hat b_1\hat a_3\ .$$
Since $t_{24}(\hat a_1\hat b_0) = \hat a_1\hat b_0$ we have $\alt_4(\hat a_1\hat b_0) = 0$. Similarly for the other terms. 

In general, the $k$-bracket $ \beta^{(k)}(a, b)$ can be written as a sum of terms $\hat a_I\hat b_J$ (or $\hat b_J\hat a_I$). 
There is always $i$ and $j$ such that $t_{ij}(\hat a_I) = - \hat a_I$ and $t_{ij}(\hat b_J) = - \hat b_J$. Thus $t_i(\hat a_I\hat b_J) = \hat a_I\hat b_J$ and $\alt_{m+n-k}(\hat a_I\hat b_J) = 0.$

\epf


\begin{remark}
Part 2 of Theorem \ref{pac} states that in characteristic $0$ the higher brackets are always trivial on the top component of the Hodge decomposition.  It will be useful for calculations in section \ref{liee}. It does not mean, in general, that the higher bracket is identically zero because it can be nontrivial on other Hodge components.
\end{remark}

%
%

\section{Deformation cohomology of tensor categories}

For a tensor functor $F:\C\to\D$ define its $n$-th tensor power by
$$F^{\ot n}:\C\times ...\times\C\to \D,\qquad F^{\ot n}(X_1,...,X_n) = F(X_1 \ot (X_2\ot( ...(X_{n-1} \ot X_{n})...))$$
For $n=0$ denote
$$F^{\otimes 0}:\Vect_k \to\D,\qquad F^{\otimes 0}(V) = V\otimes I $$
where $I\in\D$ is the unit object.

\subsection{Deformation complex of a tensor functor}\lb{dctf}

Following \cite{da,da0} consider the structure of a cosimplicial complex on the collection $E(F)(n) = End(F^{\otimes n}), \ n\ge 0$ of endomorphism algebras of tensor powers of a monoidal functor $F$.
\nl
More precisely the image of the coface map
$${\partial}_{i} : End(F^{\otimes n}) \rightarrow End(F^{\otimes n+1}) \qquad i = 0,...,n+1$$
of an endomorphism $a \in End(F^{\otimes n})$ has the following specialisation on objects $X_{1},...,X_{n+1}\in\C$:
$${\partial}_{i}(a )_{X_{1},...,X_{n+1}} = \left\{
\begin{array}{lcc}
\phi(1_{F(X_{1})} \otimes a_{X_{2},...,X_{n+1}})\phi^{-1} & , & i = 0 \\
F(\alpha_i)^{-1}(a_{X_{1},...,X_{i} \otimes X_{i+1},...,X_{n+1}})F(\alpha_i)& , & 1 \leq i \leq n \\
\psi(a_{X_{1},...,X_{n}} \otimes 1_{F(X_{n+1})})\psi^{-1} & , & i = n+1
\end{array}
\right.$$
Here $\phi$ is the tensor structure constraint  
$$F(X_1) \ot F(X_2\ot( ...(X_{n} \ot X_{n+1})...)\ \to\ F(X_1 \ot (X_2\ot( ...(X_{n} \ot X_{n+1})...)\ ;$$
$\alpha_i$ (for $1\leq i\leq n$) is the unique composition of associativity constraints  
$$X_1 \ot (X_2\ot( ...(X_{n} \ot X_{n+1})...)\ \to\  X_1 \ot (X_2\ot( ... \otimes (X_{i} \otimes X_{i+1})\ot(...(X_{n} \ot X_{n+1})...)\ ;$$ 
and $\phi$ is the unique composition of associativity and tensor structure constraints 
$$F(X_1\ot (X_2\ot( ...(X_{n-1} \ot X_{n+1})...)\ \to\  F(X_1 \ot (X_2\ot( ...(X_{n-1} \ot X_{n})...)\ot F(X_{n+1})\ .$$
\nl
The specialisation of the image of the codegeneration map
$${\sigma}_{i} : End(F^{\otimes n}) \rightarrow End(F^{\otimes n+1}) \qquad i = 0,...,n-1$$
is
$${\sigma}_{i}(a)_{X_{1},...,X_{n-1}} = {a}_{X_{1},...,X_{i},I,X_{i+1},...,X_{n-1}}.$$
The zero component of this complex as the  endomorphism algebra $End_\D(I)$ of the unit object $I$ of the category $\D$, which can be regarded as the endomorphism algebra of the functor $F^{\otimes 0}$.
\nl
The coface maps
$${\partial}_{i}:End_\D(I) \rightarrow End(F),\qquad i=0,1$$
have the form
\beq\lb{dz}{\partial}_{0}(a)_X = \rho_{F(X)} (a\otimes 1_{F(X)}){\rho}_{F(X)}^{-1}, \qquad {\partial}_{1}(a)_X = \lambda_{F(X)} (1_{F(X)}\otimes a){\lambda_{F(X)}}^{-1}\ ,
\eeq
here $\rho_{F(X)}:I\ot {F(X)}\to {F(X)}$ and $\lambda_{F(X)}:{F(X)}\ot I\to {F(X)}$ are the structural isomorphisms of the unit object $I$.

The components $ End(F^{\otimes n})$ of the cosimplicial object $E(F)$ are monoids under the composition.
It is straightforward to verify the cosimplicial maps ${\sigma}_i$ and ${\partial}_j$ are homomorphisms of monoids.
Thus we have the following.
\bpr
The maps ${\sigma}_i$ and ${\partial}_j$ make $E(F)$ a cosimplicial monoid.
\epr

\begin{defin} The (normalised) total cochain complex $(\EE^*(F),\partial) = Tot_{\delta}(E(F))$ is called  the {\em (normalised) deformation complex of the tensor functor} $F$.
Its cohomology $H^*(F)$  is  the {\em deformation cohomology of the tensor functor} $F$.
\end{defin}

\bex
The space of 1-cocycles $Z^1(F)$ coincides with the space
$$Der(F) = \{a\in End(F)|\  F_{X,Y}^{-1}a_{X\ot Y}F_{X,Y} = 1_{F(X)}\ot a_Y+ a_X\ot 1_{F(Y)},\quad X,Y\in \C\}$$
of {\em derivations} (or {\em primitive} endomorphisms) of $F$.
\nl
The subspace of 1-coboundaries $B^1(F)\subset Z^1(F)$ corresponds to the subspace $Der_{inn}(F)$ of {\em inner derivations} of $F$.
The first cohomology $H^1(F)$ is the space $OutDer(F) = Der(F)/Der_{inn}(F)$ of {\em outer derivations} of $F$.
\eex

\bth
The cosimplicial monoid $E(F)$  is 1-commutative.
\eth
\bpf
By lemma \ref{lno} we need to show that the images $E(\tau_{n,m})(a), E(\pi_{n,m})(b)$ commute for any $a\in E(F)(n)$ and $b\in E(F)(m)$.
By example \ref{elno} 
$$E(\tau_{m,n})(a) = \partial^{n+m}_{n+m-1}...\partial^{n+2}_{n+1}\partial^{n+1}_n(a) = a\ot1_m$$
and 
$$E(\pi_{m,n})(a) = \partial^{n-1}_{m+n-1}...\partial^1_{m+1}\partial^0_m(b) = 1_n\ot b\ .$$
Clearly $a\ot1_m$ commutes with $1_n\ot b$ by the naturality of the tensor product.
\epf

The corollary \ref{coraction} implies the following.
\bco
The deformation complex $\EE^*(F)$ of a tensor functor $F$ is an $E_2$-algebra.
\eco

The $\cup$-product on the deformation complex $\EE^*(F)$ takes the form
$$(a\cup b)_{X_1,...,X_m,X_{m+1},...,X_{m+n}} = (-1)^{(m-1)(n-1)}\phi(a_{X_1,...,X_m}\ot b_{X_{m+1},...,X_{m+n}})\phi^{-1}\ ,\qquad a\in \EE^m(F), b\in \EE^n(F) \ .$$
Here $\phi= F_{X_1\ot ... \ot X_{m},X_{m+1} \ot ...\ot X_{m+n}}$ is the 
coherence isomorphism
$$F(X_1 \ot  ... \ot X_{m})\ot F(X_{m+1} \ot ... \ot X_{m+n})\to  F(X_1 \ot  ... \ot X_{m+n})\ .$$
The $\cup$-product induces an associative multiplication on the cohomology
$$\cup:H^m(F)\ot H^n(F) \to H^{m+n}(F)\ .$$
The induced $\cup$-product on cohomology is super-commutative
$$b\cup a = (-1)^{|a||b|}a\cup b\ .$$
The Steenrod $\cup_1$-product is equal to 
$$a\cup_1 b = a\circ b =  \sum_{i=1}^m(-1)^{(n-1)i}a\circ_ib\ ,$$
where 
$$(a\circ_ib)_{X_1,...,X_m,X_{m+1},...,X_{m+n-1}} = $$
$$= (1_{F(X_1\ot...\ot X_{i-1})}\ot b_{X_{i},...,X_{i+m-1}}\ot 1_{F(X_{i+m}\ot...\ot X_{m+n})})a_{X_1,...,X_{i-1},X_{i}\ot...\ot X_{i+m-1},X_{i+m},...,X_{m+n-1}}\ .$$
Thus according to the formula (\ref{ger}) the bracket operation on chain level is given by:
$$\{a,b\} =  a\circ b - (-1)^{(n-1)(m-1)}b\circ a\ ,\qquad a\in E^m(F), b\in E^n(F) \ .$$
which induces the 1-bracket on the cohomology
$$\{\ ,\ \}:H^m(F)\ot H^n(F) \to H^{m+n-1}(F)\ .$$

\bex
The bracket $\{ ,\}$ on $\EE^1(F)$ coincides with the commutator  
$$\{a,b\} =  ab - ba\ .$$
It induces the Lie algebra structure on the space 
$$Z^1(F)=Der(F) = \{\alpha\in End(F)|\ \alpha_{X\ot Y} = 1_X\ot \alpha_Y + \alpha_X\ot 1_Y\}$$ 
of {\em tensor derivations} of a tensor functor $F$.
The subspace 
$$B^1(F)=InnDer(F) = \{\alpha_X = a1_X - 1_Xa|\ a\in End(I)\}$$ 
of {\em inner derivations} is a Lie ideal in $Der(F)$ and the cohomology $H^1(F)=OutDer(F)$ is a Lie algebra. 
\eex

Let now $F:\C\to\D$ and $G:\D\to\K$ be tensor functors. The composition of natural endomorphisms defines a pairing 
of cosimplicial monoids
\beq\lb{pari}E(G)(n)\ot E(F)(n)\to E(G\circ F)(n)\ .\eeq
The pairing gives rise to the mixing cup product on the cohomology
$$\cup:H^m(G)\ot H^n(F) \to H^{m+n}(G\circ F)\ .$$
The associativity of composition implies that the pairing \eqref{pari} is determined by the homomorphisms of cosimplicial monoids
\beq\lb{cihc}E(G),\ E(F)\to E(G\circ F)\ .\eeq
As the result the mixing cup product is also determined by the ring homomorphisms 
\beq\lb{cih}H^*(G),\ H^*(F) \to H^*(G\circ F)\ .\eeq
The homomorphisms of cosimplicial monoids \eqref{cihc} make the homomorphisms \eqref{cih} compatible with the brackets, i.e. homomorphisms of Gerstenhaber algebras. 
Moreover, since the pairing \eqref{pari} implies that the images of \eqref{cihc} commute, we have 
\beq\lb{bti}\{H^*(G),\ H^*(F)\} = 0\eeq
in $H^*(G\circ F)$. 

In particular, when $F=G=Id$   the bracket operation on deformation complex of an identity functor is cohomologically trivial. We will see in the next section that in this case we have a secondary bracket which can be nontrivial.

\subsection{Deformation complex of a tensor category}\lb{dctc}

\begin{defin} 
The {deformation complex  $\EE^*(\C)$ of the tensor category} $\C$  is the deformation complex of the identity functor $Id_\C:\C\to\C.$ \end{defin} 
The cosimplicial complex $E(Id_\C) = E(\C)$  has a higher amount of commutativity.
\bth
The cosimplicial complex $E(\C) $ of a tensor category $\C$ is 2-commutative.
\eth
\bpf
By lemma \ref{lnt} it is sufficient to show  that the images $E(\tau)(a), E(\pi)(b)$ commute for any maps $\tau = \tau_{m,n}^i$ and $\pi = \pi_{m,n}^i$ and any $a\in E(\C)(n)$ and $b\in E(\C)(m).$
By example \ref{elnt} 
$$E(\tau_{m,n}^i)(a) = \partial^{i+m-1}_{n+m-2}...\partial^{i+2}_{n+1}\partial^{i+1}_n(a)$$
and 
$$E(\pi_{m,n}^i)(b) = \partial^{n+m-1}_{n+m-2}...\partial^{i+m+2}_{n+i+1}\partial^{i+m+1}_{n+i}\partial^{i-1}_{n+i-1}...\partial^1_{n+1}\partial^0_n(b) = 1_{i}\ot b\ot 1_{n-i-2}\ .$$
By naturality of $a$ the evaluation 
$$(\partial^{i+m-1}_{n+m-2}...\partial^{i+2}_{n+1}\partial^{i+1}_n(a))_{X_1,...,X_{m+n-1}} =  a_{X_1,...,X_i,X_{i+1}\ot...\ot X_{i+m},X_{i+m+1},...,X_{m+n-1}}$$ 
commutes with the evaluation 
$$(1_{i}\ot b\ot 1_{n-i-2})_{X_1,...,X_{m+n-1}} = 1_{X_1}\ot...\ot 1_{X_{i}}\ot b_{X_{i+1},...,X_{i+m}}\ot 1_{X_{i+m+1}}\ot...\ot 1_{X_{m+n-1}}\ .$$
\epf

Now the corollary \ref{coraction} implies the following.
\bco
The deformation complex $\EE^*(\C)$ of a tensor category $\C$ is an $E_3$-algebra.
\eco

\bex
The degree 2 bracket $\bl-,-\br:Z^2(\C)\ot Z^1(\C)\to Z^1(\C)$ is zero.
The degree 2 bracket $\bl-,-\br:Z^2(\C)\ot Z^2(\C)\to Z^2(\C)$ is the opposite to the commutator with respect to the product in $\EE^2(\C)$
$$\bl a,b\br = ba-ab\ .$$
\eex

\subsection{Deformation complex of a symmetric functor}\lb{dcsf}

Let $F:\C\to\D$ be a tensor functor.
Assume that the category $\C$ is braided. 
For $i=1,...,n-1$ define a monoid automorphism $t_i:End(F^{\otimes n})\to End(F^{\otimes n})$ by assigning to $a\in End(F^{\otimes n})$ the composite
$$\xymatrix{
F(X_1\ot...\ot X_n) \ar[d]_{F(1c_{X_{i+1},X_i}^{-1}1)} &&& F(X_1\ot...\ot X_n) \\
F(X_1\ot...\ot X_{i+1}\ot X_i\ot... \ot X_n) \ar[rrr]^{a_{X_1,..., X_{i+1}, X_i,... , X_n}} &&& F(X_1\ot...\ot X_{i+1}\ot X_i\ot... \ot X_n) \ar[u]_{F(1c_{X_{i+1},X_i}1)}  
}$$
A standard argument shows that $t_i$ and $t_j$ commute, whenever $|i-j|>1$ and that $t_it_{i+1}t_i = t_{i+1}t_it_{i+1}$. Thus we have an action of the braid group $B_n$ on $End(F^{\otimes n})$.
Moreover, these actions have the following properties with respect to the cosimplicial structure.
\ble
$$t_i\partial_i = \partial_i,\qquad \partial_i t_j = \left\{\begin{array}{cc} t_j\partial_i & i>j+1\\ 
t_{i-1}t_i\partial_{i-1} & i=j+1\\
t_{i+1}t_i\partial_{i+1} & i=j\\
t_{j+1}\partial_i & i<j\\ 
\end{array}\right.\qquad 
\sigma_i t_j = \left\{\begin{array}{cc} t_j\sigma_i & i>j+1\\ 
\sigma_{i-1} & i=j+1\\
\sigma_{i+1} & i=j\\
t_{j-1}\sigma_i & i<j\\ 
\end{array}\right.$$
\ele
\bpf
Follows directly from the definition. For example, the identity $t_i\partial_i = \partial_i$ is implied by the naturality of natural transformations in $End(F^{\otimes n})$. The identity $\partial_2 t_1 = t_1 t_2 \partial_1$ is a consequence of one of the hexagon axions for the braiding:
$$(\partial_2 t_1)(a)_{X_1,X_2,X_3} = t_1(a)_{X_1,X_2\ot X_3} = F(c_{X_2\ot X_3,X_1})a_{X_2\ot X_3,X_1}F(c_{X_2\ot X_3,X_1})^{-1} = $$
$$F(c_{X_2,X_1}\ot 1)F(1\ot  c_{X_3,X_1})a_{X_2\ot X_3,X_1}F(1\ot  c_{X_3,X_1})^{-1}F(c_{X_2,X_1}\ot 1)^{-1} = $$
$$F(c_{X_2,X_1}\ot 1)(t_2\partial_1)(a)_{X_2,X_1, X_3}F(c_{X_2,X_1}\ot 1)^{-1} = (t_1 t_2 \partial_1)(a)_{X_1,X_2,X_3}\quad .$$
The identity $\sigma_1 t_1 = \sigma_2$ follows from the unit normalisation condition for the braiding:
$$(\sigma_1 t_1)(a)_X = t_1(a)_{I,X} = F(c_{X,I})a_{X,I}F(c_{X,I})^{-1} = a_{X,I} = \sigma_2(a)_X\ .$$
\epf

Let $F:\C\to\D$ be a tensor functor into a braided category $\D$. 
For $i=1,...,n-1$ define a monoid automorphism $t_i:End(F^{\otimes n})\to End(F^{\otimes n})$ by assigning  $a\in End(F^{\otimes n})$ to the composite
$$\xymatrix{
F(X_1\ot...\ot X_n) \ar[d] &&& F(X_1\ot...\ot X_n) \\
F(X_1)\ot...\ot F(X_n) \ar[d]_{1\ot c_{F(X_{i+1}),F(X_i)}^{-1}1} &&& F(X_1)\ot...\ot F(X_n) \ar[u] \\
F(X_1)\ot...\ot F(X_{i+1})\ot F(X_i)\ot... \ot F(X_n) \ar[d] &&& F(X_1)\ot...\ot F(X_{i+1})\ot F(X_i)\ot... \ot F(X_n) \ar[u]_{1\ot c_{F(X_{i+1}),F(X_i)}1}  \ar[u]\\
F(X_1\ot...\ot X_{i+1}\ot X_i\ot... \ot X_n) \ar[rrr]^{a_{X_1,..., X_{i+1}, X_i,... , X_n}} &&& F(X_1\ot...\ot X_{i+1}\ot X_i\ot... \ot X_n) \ar[u] 
}$$
As above the automorphisms $t_i$ satisfy to the defining relations of the braid group $B_n$. 
Among compatibilities between the braid group action and the cosimplicial structure in this case we have the following.
\ble
Let $F:\C\to\D$ be a tensor functor into a braided category $\D$. 
Then the following is true in $End(F^{\otimes n})$
$$t_n...t_1\partial_0 = \partial_{n+1}\quad .$$
\ele
\bpf
This follows from the naturality and the hexagon axioms for braiding:
$$t_n...t_1\partial_0(a)_{X_1,...,X_{n+1}} = $$
$$(1\ot c_{F(X_{n+1}),F(X_1)})... (c_{F(X_{n+1}),F(X_n)}\ot 1)\partial_0(a)_{X_{n+1},X_1,...,X_{n}}(c_{F(X_{n+1}),F(X_n)}\ot 1)^{-1}...(1\ot c_{F(X_{n+1}),F(X_1)})^{-1} =
$$
$$c_{F(X_{n+1}),F(X_1\ot...\ot X_n)}(1\ot a_{X_1,...,X_n}) c_{F(X_{n+1}),F(X_1\ot...\ot X_n)}^{-1} =  a_{X_1,...,X_n}\ot 1 = \partial_{n+1}(a)_{X_1,...,X_{n+1}}\ .$$
\epf

Note that for a braided tensor functor $F:\C\to\D$ between braided categories the two braid group actions on $E(F)(n)$ coincide. Thus we have the following.
\bpr
The cosimplicial monoid $E(F)$ of a symmetric tensor functor $F:\C\to\D$ is a symmetric cosimplicial monoid (in the sense of section \ref{scm}).
\epr

\subsection{The change of scalars}

Here we discuss the change of scalars of a tensor category from the point of view of enriched category theory (see \cite{ke}). 

Let $A$ be a commutative $k$-algebra. 
For a $k$-linear category $\C$ denote by $\C_A$ the category with the same objects as $\C$ and with hom spaces $\C_A(X,Y) = \C(X,Y)\ot_kA$.  
For a $k$-linear functor $F:\C\to \D$ denote by $F_A:\C_A\to\D_A$ the $A$-linear functor $F_A(X) = F(X)$ with the effcct on morphisms $F_{X,Y}\ot 1:\C(X,Y)\ot_kA\to \C(F(X),F(Y))\ot_kA$.

\bpr
The canonical homomorphism
$End(F)\ot_k A\to End(F_A)$ 
is an isomorphism.
\epr
\bpf
Let $\alpha:F_A\to F_A$ be a natural transformation. Let $l:A\to k$ be a $k$-linear map. 
For an object $X\in \C$ consider the composite $\alpha(l)_X=(I\ot l)(\alpha_X)\in \C(X,X)$, 
where $\alpha_X\in \C(X,X)\ot_kA$ is the specialisation of $\alpha$. 
The components $\alpha(l)_X\in \C(X,X)$ are specialisations of natural transformations $\alpha(l):F\to F$. 
Indeed, the naturality condition $(f\ot 1)\alpha_X = \alpha_Y(f\ot 1)$ for a morphism $f:X\to Y$ gives $f \alpha(l)_X = \alpha(l)_Yf$ (assuming that $l(1)$ is invertible). 
\epf

\bco
The canonical homomorphism
$E^*(F)\ot_k A\to E^*(F_A)$ 
is an isomorphism.
\eco

The change of scalars allows us to look at sets of tensor structures as functors $\A lg_k\to \S et$, which we will following \cite{gr} consider as {\em functor of points} of moduli spaces of such structures.
From this point of view deformations of tensor structures are tangent spaces to the corresponding functor of points. 

Denote by $A_2 = A[\ve|\ve^2=0]$ the algebra of dual numbers over $k$.

\noindent The {\em tangent space} $T_xX$ of a functor of points $X:\A lg_k\to \S et$ at a point $x\in X(k)$ is the fibre $X(f)^{-1}(x)$ of the map $X(f):X(A_2)\to X(k)$ corresponding to the homomorphism $f:A_2\to k$ sending $\ve$ to zero.

\noindent Denote by $A_3 = k[\ve|\ve^3=0]$. Denote by $g:A_3\to k$ the homomorphism sending $\ve$ to zero.
\nl
The {\em (first) tangent cone} $Q_xX\subset T_xX$ of a functor of points $X:\A lg_k\to \S et$ at a point $x\in X(k)$ is the image in $T_xX$ of
the fibre $X(g)^{-1}(x)$ under the map $X(A_3)\to X(A_2)$ corresponding to the homomorphism $A_3\to A_2$ sending $\ve$ to $\ve$.

\subsection{Deformation theory of tensor functors}\lb{dttf}

Here we show how 1- and 2-brackets on low dimensional cohomology appear in deformation theory of tensor functors.

Recall that an automorphism $a:F\to F$ of a tensor functor $F:\C\to\D$ is {\em tensor} if 
$$a_{X\ot Y} = a_X\ot a_Y,\qquad X,Y\in \C\ .$$ 
Tensor automorphisms are closed under the composition and form a group $Aut_\ot(F)$.
\bex
Let $c\in Aut_\D(F(I))$ be an automorphism of the identity object of $\D$. Then $\partial_0(c)\partial_1(c)^{-1}$ is a tensor automorphism of $F$, which we call an {\em inner automorphism} (here $\partial_0(c)$ and $\partial_1(c)$ are as in \eqref{dz}). 
Note that inner automorphisms form a normal subgroup $InnAut_\ot(F)$ in $Aut_\ot(F)$. 
We denote the quotient group $Out_\ot(F)$. 
\eex

An endomorphism $a:F\to F$ of a tensor functor is a {\em tensor derivation} if 
$$\alpha_{X\ot Y} = \alpha_X\ot 1 + 1\ot \alpha_Y,\qquad X,Y\in \C\ .$$ 
It is straightforward to see that the commutator of tensor derivations is a tensor derivation. Denote by $Der(F)$ the Lie algebra of tensor derivations of $F$. 
\bex
Let $c\in End_\D(F(I))$ be an endomorphism of the identity object of $\D$. Then $\partial_0(c) -\partial_1(c)$ is a tensor derivation of $F$, which we call an {\em inner derivation}. 
Note that inner derivation form a Lie ideal $InnDer(F)$ in $Der(F)$. 
We denote the quotient Lie algebra $OutDer(F)$. 
\eex

Here for a tensor functor $F:\C\to\D$ we consider the following group valued functor of points $\A lg_k\to \G rp$
$$A\mapsto Aut_\ot(F_A), \ InnAut_\ot(F_A),\ Out_\ot(F_A) $$
which allow us to look at 
$$Aut_\ot(F),\qquad InnAut_\ot(F),\qquad Out_\ot(F)$$
as proalgebraic groups. 
\bpr
The Lie algebras  
$$Z^1(F)=Der(F),\qquad B^1(F)=InnDer(F),\qquad H^1(F)=OutDer(F)$$ are the tangent Lie algebras of the proalgebraic groups 
$$Aut_\ot(F),\qquad InnAut_\ot(F),\qquad Out_\ot(F)$$
correspondingly. 
\nl
The commutator Lie bracket is the degree 1 bracket on $Z^1(F)$. 
\epr
\bpf
A tensor automorphism of $F$ over $A_2=k[\ve|\ve^2=0]$, which descents to the identity under the reduction $A_2\to k$ has the form $1 + \ve\alpha$, with $\alpha$ being a tensor derivation of $F$.  
The standard computation over $A_3=k[\ve|\ve^3=0]$ shows that the group commutator of $1 + \ve\alpha+...$ and $1 + \ve\beta+...$ has the form $1 + \ve^2[\alpha,\beta]$. 
\epf

For a functor $F:\C\to\D$ between tensor categories denote by $Tens(F)$ the set of isomorphism classes of tensor structures on $F$.
The change of scalars makes it a functor of points
$$A\mapsto Tens(F_A)\ .$$

\bpr\lb{ttf}
The second cohomology $H^2(F)$ is the tangent space to the moduli space $Tens(F)$ of tensor structures of $F$.
\nl
The tangent cone to the moduli space $h^2(F)$ is the set of solutions
$$\{\phi,\phi\} = 0\ .$$
\epr
\bpf
Write a modification to the tensor constraint of the identity functor as 
$$1 + \ve\phi + \ve^2\phi^{(2)} +...\ . $$
The tangent equation (the coefficient equation for $\ve$) to the pentagon equation is 
$$\partial_0(\phi) + \partial_2(\phi) - \partial_1(\phi) - \partial_3(\phi) = 0 ,$$
which coincides with the coboundary condition $\partial(\phi) = 0$. 
\nl
The second tangent equation (the coefficient equation for $\ve^2$) to the pentagon equation is 
$$\partial_2(\phi)\partial_0(\phi) - \partial_1(\phi)\partial_3(\phi) = \partial(-\phi^{(2)})\ ,$$
which coincides with $\{\phi,\phi\} = \partial(-2\phi^{(2)})$. 
\epf

\subsection{Deformation theory of tensor categories}\lb{dttc}

Recall from \cite{da4} that a tensor autoequivalence $F:\C\to\C$ is {\em soft} if it is isomorphic as a plain functor to the identity functor $Id_\C$. 
It is straightforward to see (e.g. \cite{da}) that the composition of soft autoequivalences is soft and that a quisi-inverse to a soft autoequivalence is also soft. 
Denote by $Aut^1_\ot(\C)$ the group of isomorphism classes of soft autoequivalences of $\C$. 
The change of scalars makes it a group valued functor of points
$$A\mapsto Aut^1_\ot(\C_A)\ .$$

\bpr
The Lie algebra $(H^2(\C), \bl\ ,\ \br)$ is the tangent Lie algebra of the proalgebraic group $Aut^1_\ot(\C)$.
\epr
\bpf
By the proposition \ref{ttf} the tangent space at the identity of $Aut^1_\ot(\C)$ is $H^2(\C)$.
Note that the group commutator in $Aut^1_\ot(\C)$ of two soft autoequivalences with the tensor constraints
$$1 + \ve\phi + ...\qquad, \qquad 1 + \ve\psi + ...$$
is a soft autoequivalence with the tensor constraint
$1 + \ve^2\bl\phi,\psi\br$, where $\bl\phi,\psi\br$ is the degree 2 bracket on $H^2(\C)$.
\epf

For a $k$-linear category denote by $Tens(\C)$ the set of equivalence classes of tensor structures on $\C$. 
The change of scalars makes it a functor of points
$$A\mapsto Tens(\C_A)\ .$$

\bpr
The third cohomology $H^3(\C)$ is the tangent space to the moduli space $Tens(\C)$.
\nl
Let $k$ be a field of characteristic not 2.
The tangent cone to the moduli space $h^3(\C)$ is the set of solutions
$$\bl\alpha,\alpha\br = 0\ .$$
\epr
\bpf
Write a modification to the associativity constraint as 
$$1 + \ve\alpha + \ve^2\alpha^{(2)} +...\ . $$
The tangent equation to the pentagon equation is 
$$\partial_0(\alpha) + \partial_2(\alpha) + \partial_4(\alpha) - \partial_1(\alpha) - \partial_3(\alpha) = 0\ ,$$
which coincides with the coboundary condition $\partial(\alpha) = 0$. 
Let now 
$$1 + \ve\alpha +...\qquad 1 + \ve\beta +...$$
be two modifications to the associativity constraint. Write a modification to the tensor constraint of the identity functor as 
$1 + \ve\phi + ...\ . $
The tangent equation to the tensor constraint equation is 
$$\partial_0(\phi) + \partial_2(\phi) - \partial_1(\phi) - \partial_3(\phi) = \alpha - \beta\ ,$$
which coincides with the coboundary condition $\partial(\phi) = \alpha - \beta$. 
\nl
The second tangent equation to the pentagon equation is 
$$\partial_2(\alpha)\partial_0(\alpha) + \partial_4(\alpha)\partial_0(\alpha) + \partial_4(\alpha)\partial_2(\alpha) - \partial_1(\alpha)\partial_3(\alpha) = \partial(-\alpha^{(2)})\ ,$$
which coincides with $\bl\alpha,\alpha\br = \partial(-2\alpha^{(2)})$. 
\epf


\bpr
The obstruction for $\phi\in H^2(C)$ to deform the tensor structure on the identity functor $Id_\C$ compatible with the deformation of the associativity of $\C$ corresponding to $\alpha\in H^3(\C)$ is
$$\bl\phi,\alpha\br = \{\phi,\phi\}\ .$$
\epr
\bpf
Write a modification to the associativity constraint as 
$$1 + \ve\phi + \ve^2\phi^{(2)} + ...\qquad 1 + \ve\alpha + ...\ . $$
The second tangent equation to the tensor constraint condition is 
$$(\partial_0(\phi) + \partial_2(\phi))\alpha + \partial_0(\phi)\partial_2(\phi) =
\alpha(\partial_1(\phi)+\partial_3(\phi)) + \partial_3(\phi)\partial_1(\phi) - \partial(\phi^{(2)})\ ,$$
which coincides with $\bl\phi,\alpha\br = \{\phi,\phi\}+\partial(...)$. 
\epf

\section{Examples}\lb{exampl}

\subsection{Modules over bialgebra} 
Here we reproduce the computations from \cite{da} for deformation complexes of the forgetful and the identity functors on the category of modules over a bialgebra.
We then write the 1- and 2-brackets on them explicitly. 

Let $H$ be a bialgebra with coproduct $\Delta :H\rightarrow H\otimes H$ and counit $\varepsilon :H\rightarrow k$. 
Denote by $H\da\Mod$ the tensor category of $H$-modules.


Let $F:H\da\Mod\to \Vect$ be the forgetful functor.

We start by computing the algebras of endomorphisms of tensor powers of $F$ (see \cite{da}).
\ble\lb{etp}
The algebra of endomorphisms $E(F)(n) = End(F^{\otimes n})$ of the $n$-th power of the forgetful functor is isomorphic to the tensor power $H^{\ot n}$ of the bialgebra.
\nl
The isomorphism is exhibited by two mutually inverse maps:
$$H^{\otimes n}\to End(F^{\otimes n}) ,$$
which associates to an element $x\in H^{\otimes n}$ the endomorphism of multiplication by $x$, and
$$End(F^{\otimes n})\to H^{\otimes n} ,\qquad a\mapsto a_{H,...,H} ,$$
which sends an endomorphism to its specialisation on the regular $H$-module $H$.
\ele
\bpf
Clearly the map $H^{\otimes n}\to End(F^{\otimes n}) ,$ defined in the statement of the lemma is a homomorphism of algebras.
All we need to do is to prove that this is an isomorphism.
\nl
Recall a well-known fact that the forgetful functor has the right adjoint
$$\Vect\to H\da\Mod ,\qquad V\mapsto H\otimes V$$
computing the free $H$-module on a vector space $V$.
Similarly the $n$-th power of the forgetful functor considered as a functor from the Deligne's tensor power
$$R_n:\xymatrix{H^{\ot n}\da\Mod = (H\da\Mod)^{\boxtimes n} \ar[r]^(.75){F^{\ot n}} & \Vect}$$
has the right adjoint
$$\Vect\to H^{\ot n}\da\Mod ,\qquad V\mapsto H^{\ot n}\otimes V\ .$$
The algebra of its endomorphisms is
$$End(R_n) = End_{H^{\ot n}}(H^{\ot n}) = (H^{\ot n})^{op}\ .$$
The adjunction identifies the endomorphism algebra $End(F^{\otimes n})$ with the opposite of the endomorphism algebra $End(R_n)$.
Finally it is straightforward to see that the isomorphism
$$End(F^{\otimes n}) \to End(R_n)^{op} = H^{\ot n}$$
is the specialisation on the regular $H$-module $H$.
\epf

\begin{prop}
The cosimplicial complex $E(F)$ 
of the forgetful functor $F:H\da\Mod\to \Vect$ is isomorphic to the bar complex
$H^{\ot *}$ of $H$ with coface maps ${\partial}^{i}_{n}:H^{\ot n-1}\longrightarrow H^{\ot n}$ given by
$${\partial}_i(h_{1}\otimes ...\otimes h_{n}) =
\left\{
\begin{array}{ccc}
1\otimes h_{1}\otimes ...\otimes h_{n}&,& i=0\\
h_{1}\otimes ...\otimes\Delta (h_{i})\otimes ...\otimes h_{n}&,& 1\leq i\leq n\\
h_{1}\otimes ...\otimes h_{n}\otimes 1&,& i=n+1
\end{array}
\right.
$$
and codegeneration
$${\sigma}_i(h_{1}\otimes ...\otimes h_{n+1}) = h_{1}\otimes ...\otimes\varepsilon (h_{i})\otimes ...\otimes h_{n+1}$$
\epr
\bpf
Direct computation of the effect of coface and codegeneration maps on endomorphisms given by multiplication with elements of $H^{\ot n}$ provides the result.
\epf

The deformation complex of the forgetful functor $F:H\da\Mod\to \Vect$ is the {\em co-Hochschild complex} of $H$ \cite[Appendix]{da2}, i.e.
the complex $C^*(H) = (H^{\otimes *},\partial)$ with the differential $\partial:H^{\otimes n}\to H^{\otimes n+1}$ defined by
\begin{equation}\label{tangcoh}
\partial(a) = 1\otimes a +\sum_{i=1}^n(-1)^i(id^{\otimes i-1}\otimes\Delta\otimes id^{\otimes n-i-1})(a) + (-1)^{n+1}(a\otimes 1).
\end{equation}

The cup product on the co-Hochschild complex $C^*(H)$ is:
$$\scup:C^m(H)\otimes C^n(H)\to C^{m+n}(H),\qquad a\scup b = a\ot b.$$
The homotopy for commutativity is 
$$a\circ b = \sum_{i=1}^m(-1)^{(n-1)i}a\circ_ib\ ,$$
where 
$$a\circ_ib = (id^{\ot i-1}\ot\Delta^{(n-1)}\ot id^{\ot m-i})(a)(1^{\ot i-1}\ot b\ot  1^{\ot m-i})\ ,$$
where $\Delta^{(n-1)}:H\to H^{\ot n}$ is the iterated coproduct and $a\in H^{\ot m}$, $b\in H^{\ot n}$.
\nl
The 1-bracket 
$$\{a,b\} =  a\circ b - (-1)^{(n-1)(m-1)}b\circ a\ ,\qquad a\in H^{\ot m}, b\in H^{\ot n} \ .$$

\bex
The first cohomology
$$H^1(F) = \{a\in H|\ \Delta(a) = 1\ot a + a\ot 1\} = Prim(H)$$ coincides with the space of {\em primitive} elements of $H$.
The 1-bracket on $H^1(F)$ is the commutator bracket.
\eex


The following was proved in \cite{da}. 
\bpr
The cosimplicial complex $E(H\da\Mod)$ 
is isomorphic to the subcomplex of the bar complex of $H$, which consists of $H$-invariant elements (the subcomplex of centralisers $C_{H^{\otimes n}}(\Delta (H))$ of the images of diagonal embeddings).
\epr
\bpf
The isomorphism from lemma \ref{etp}
$$H^{\otimes n}\to End(F^{\otimes n})\ ,$$
sending an element $x\in H^{\otimes n}$ to the endomorphism of multiplication by $x$, induces an isomorphism 
$$C_{H^{\otimes n}}(\Delta (H))\to End(Id_{H\da\Mod}^{\otimes n})\ .$$
\epf

The cohomology $H^*(H\da\Mod)$ can be computed via the equavariant spectral sequence
\beq\lb{ess}E_2^{p,q} = H^p(H,H^q_{ch}(H))\ \Longrightarrow\ H^{p+q}(H\da\Mod)\eeq
with the second leave occupied by Sweedler's cohomology $H^p(H,H^q_{ch}(H))$ with the coeficients in the co-Hochschild cohomology of $H$ considered with the adjoint $H$-action. 

\subsection{Lie algebras}\lb{liee}

Let $\g$ be a Lie algebra. Let $U(\g)$ be its universal enveloping algebra. Denote by $\Rep(\g) = U(\g)\da\Mod$ the tensor category of $\g$-representations.

The following was shown in \cite{dr}. 
\bpr\lb{cof}
Let the characteristic of the ground field $k$ be zero. Let $F:\Rep(\g)\to\Vect$ be the forgetful functor. 
Then the natural homomorphism
\beq\lb{em}\Lambda^{*}({\frak g}) = \Lambda^{*}(H^{1}(F))\rightarrow H^{*}(F),\eeq
induced by the multiplication in $H^{*}(F)$ is an isomorphism.
\epr

Note that the inverse to \eqref{em} sends a cocycle $x\in Z^n(F)\subset U(\g)^{\ot n}$ to its anti-symmetrisation 
$$\alt_n(x) = \sum_{\sigma\in S_n}(-1)^\sigma \sigma(x)\ .$$
It is a fact independent of the characteristic of $k$ that  the first cohomology coincides with the space of primitive elements of the universal enveloping algebra
$$H^{1}(F) = Prim(U(\g)) = \{x\in U(\g)|\ \Delta(x) = x\ot 1 + 1\ot x\}$$ 
and that $\alt_n(x)\in \Lambda^n(H^{1}(F))$ for $x\in Z^n(F)$. 
Moreover, the map 
$$\xymatrix{Z^n(F)\ \ar[r]^(.4){alt_n} &\ \Lambda^n(H^{1}(F))}$$
is surjective. Indeed, for $x_i\in Prim(U(\g))$ the indecomposable tensor $x_1\ot...\ot x_n$ is a cocycle and its anti-symmetrisation is $x_1\wedge...\wedge x_n$. 
\bre
Note that the forgetful functor $F:\Rep(\g)\to\Vect$ is symmetric and thus, by section \ref{scm}, its cohomology possesses a Hodge decomposition.
The proposition \ref{cof} says that $H^{n}(F) = H^{n,n}(F) = \Lambda^{n}({\frak g})$. 
\ere
\bre
Let $k$ be the field of characteristic $p$. 
Then the $p$-th power of any primitive element is primitive: $x^p\in Prim(U(\g))$ for $x\in Prim(U(\g))$. 
This shows that $Prim(U(\g))$ contains the direct sum of infinitely many copies of $\g$. 
\nl
Moreover for $x\in Prim(U(\g))$ the following is well-defined
$$\partial\left(\frac{x^p}{p}\right) =\ \frac{1}{p}\ \sum_{i=1}^{p-1}\ C^p_i\ x^i\ot x^{p-i}$$
and is a symmetric 2-cocycle, i.e. 
$$\alt_2\left( \partial\left(\frac{x^p}{p}\right) \right) = 0\ .$$
\ere

Recall (from e.g. \cite{dri}) the {\em Schouten} bracket on $\Lambda^{*}({\frak g})$:
$$\{x_{1}\wedge ...\wedge x_{s},y_{1}\wedge ...\wedge y_{t}\} = $$
$$\sum_{i,j}(-1)^{i+j}[x_{i},y_{j}]\wedge x_{1}\wedge ...\wedge\widehat{x_{i}}\wedge ...\wedge x_{s}\wedge y_{1}\wedge ...\wedge\widehat{y_{j}}\wedge ...\wedge y_{t},$$
where $\widehat{z}$ means that $z$ does not occur in the product.

\bre\lb{sbo}
The Schouten bracket $\{x,y_{1}\wedge ...\wedge y_{t}\}$ is the result of the adjoint action of $x$ on $y_{1}\wedge ...\wedge y_{t}$:
$$\{x,y_{1}\wedge ...\wedge y_{t}\} = [x,y_{1}]\wedge ...\wedge y_{t} + y_{1}\wedge[x,y_2]\wedge ...\wedge y_{t} +...+ y_{1}\wedge ...\wedge[x,y_{t}]\ .$$
\ere

\bpr
Let the characteristic of the ground field $k$ be zero.
Then the 1-bracket on $H^{*}(F)= \Lambda^{*}{\frak g}$ coincides with the Schouten bracket.
\epr
\bpf
Follows from the biderivation property of the bracket and the fact that on $H^1(F) = \g$ the bracket is the Lie bracket of $\g$. 
\epf

\bex
The Schouten bracket on $r\in \Lambda^2(\g)$ with itself has the form
$$\{r,r\} = [r_{23},r_{13}] + [r_{23},r_{12}] + [r_{13},r_{12}]\ .$$
Here $r_{23}=1\ot r$ etc.
The Maurer-Cartan equation $\{r,r\}=0$ is known as the {\em classical Yang-Baxter equation} \cite{dr2}. 
\eex

The following was proved in \cite[Appendix]{da2} (see also \cite{del}).
\bpr\lb{cif}
Let the characteristic of the ground field $k$ be zero.
Then the homomorphism 
$$\Lambda^{*}({\frak g})^\g\ \to\ H^{*}(\Rep(\g))$$
induced by the multiplication in $H^{*}(F)$, where $F:\Rep(\g)\to\Vect$ is the forgetful functor,  is an isomorphism.
\epr
\bpf
The spectral sequence \eqref{ess} takes the form
$$E_2^{p,q} = H^p(\g,\Lambda^q(\g))\ \Longrightarrow\ H^{p+q}(\Rep(\g))\ ,$$
where $H^p(\g,\Lambda^q(\g))$ is the cohomology of $\g$ with coefficients in the adjoint module $\Lambda^q(\g)$. 
The $\g$-invariant splitting $\Lambda^q(\g)\to Z^q(F)$ sending $x_1\wedge ...\wedge x_q$  to $\frac{1}{q!}x_1\wedge ...\wedge x_q$, 
guaranties the collapse of the equivariant spectral sequence on the first page.
\epf
\bre\lb{hdi}
Proposition \ref{cif} says that the Hodge decomposition of the cohomology of the identity functor on $\Rep(\g)$ is  $H^{n}(\Rep(\g)) = H^{n,n}(\Rep(\g)) = \Lambda^{n}({\frak g})^\g$. 
\ere

\bco
The cohomology $H^{*}(\Rep(\g))$ coincides with the kernel of the 1-bracket on $H^{*}(F)$, where $F:\Rep(\g)\to\Vect$ is the forgetful functor. 
\eco
\bpf
Since $H^{*}(F)$ is multiplicatively generated by $H^1(F)$ the kernel of the 1-bracket on $H^{*}(F)$ is the common kernel of the derivations $\{x,-\}$ on $H^{*}(F)$ for all $x\in\g=H^1(F)$. According to remark \ref{sbo} this common kernel is nothing but the subspace of $\g$-invariants of $H^{*}(F)$. 
\epf

\bth\lb{2bt}
Let $k$ be a field of characteristic zero.
The 2-bracket on $H^{*}(\Rep(\g))$ is zero.
\eth
\bpf
Follows from the remark \ref{hdi} and Theorem \ref{pac}.
\epf

In finite characteristic the 2-bracket on $H^{*}(\Rep(\g))$ is far from being trivial.
\bex\lb{nte}
Let $k$ be a field of characteristic 3.
Let $\g = \langle x, y, z\rangle$ be the 3-dimensional Heisenberg algebra over $k$:
$$[x,y] = z,\qquad [x,z] = [y,z] = 0\ .$$
The 2-cocycles $x\wedge z,\ y\wedge z$ are $\g$-invariant.
Their 2-bracket
$$\bl x\wedge z, y\wedge z\br = [x\wedge z, y\wedge z] = [x,y]\ \ot\ z^2 + z^2\ \ot\ [x,y] = z\ \ot\ z^2 + z^2\ \ot\ z = \partial\left(\frac{z^3}{3}\right)$$
has a non-trivial cohomology class in $H^2(\Rep(\g))$. 
This computation has an interesting similarity with the example 6.6. of \cite{da3}.
\eex

\appendix \label{LP}

\section{Lattice paths, shuffles and sketches } \label{latticepaths}

\subsection{Lattice paths operad and its filtration}

Here we define the lattice path operad, introduced in \cite{BB} .
 \nl
 Recall that the category $\Cat$ has exactly two closed symmetric monoidal structures: the cartesian structure and the so called funny product structure.
 Funny tensor product $A\boxx B$ of two small categories has the cartesian product of objects sets of $A$ and $B$ as objects, but morphisms  are generated by the expressions $(f,id)$ and $(id,g)$ where  $f:a\to a'$ in $A$ and $g:b\to b'$ in $B.$ We then factorise by relations  
 $$(f,id)\circ(id,g)\circ (id,g') = (f,id)\circ(id,g\circ g') \ \  \mbox{and} \ \ (f',id)\circ(f,id)\circ (id,g) = (f'\circ f,id)\circ (id,g)$$
 and similarly on the other side.
 The result is that in $A\boxx B$ there are  two different morphisms $(f,id)\circ (id,g)$ and $(id,g)\circ (f,id)$ from $(a,b)$ to $(a',b')$ unless one of $f$ or $g$ is the identity. 
 From this definition it is clear that  there is a natural morphism $A\boxx B \to A\times B,$ which identifies $(f,id)\circ (id,g)$ and $(id,g)\circ (f,id).$  
 \begin{remark} It is often easier to understand tensor product through its internal $Hom.$ The internal $Hom$-functor for the product $\boxx$ is given by by the category whose objects are functors from $A$ to $B$ and whose morphisms is the set of all transformations (not necessary natural) from $F$  to $G.$
 \end{remark}

Observe that if $a\in A$ and $b\in B$ are terminal (or weakly terminal) objects then $(a,b)\in A\boxx B$ is, in generally,  only weakly terminal. Similarly for initial objects. 
So,  the tensor product $\boxx$  restricts to the category $\Cat_{*,*}.$ 
  
The lattice paths operad  $\LL$ is a symmetric coloured operad in $\bSet$ with natural numbers as its set of colours  and whose space of operations 
$$\LL(n_1,\ldots, n_k; n)=\Cat_{*,*}(\langle n+1\rangle,\langle n_1 +1\rangle\boxx\cdots\boxx \langle n_k +1\rangle),$$
and the operad substitution maps being induced by tensor and composition in $\Cat_{*,*}.$ 
The underlying category of $\LL$ is $\De$ since by Joyal's duality  
$$\LL(n; m) = \Cat_{*,*}(\langle n+1\rangle,\langle m+1\rangle ) = \De(m,n)\ .$$

Recall that similarly to the cartesian product the funny product  $\boxx$  admits two natural projections:
$$A\stackrel{pr_A}{\longleftarrow} A\boxx B \stackrel{pr_B}{\longrightarrow}B. $$
Hence, for any two $1\le i<j\le k$ there  is a projection
$$\pi_{ij}:\langle n_1 +1\rangle\boxx\cdots\boxx \langle n_k +1\rangle \longrightarrow \langle n_i +1\rangle\boxx \langle n_j +1\rangle$$ 
   from a hypercube     to a  square.
Let   $\psi$   be a lattice path and  let $\psi_{ij}$ be the composite
$$\langle n+1\rangle \stackrel{\psi}{\to} \langle n_1 +1\rangle\boxx\cdots\boxx \langle n_k +1\rangle \stackrel{\pi_{ij}}{\longrightarrow} \langle n_i +1\rangle\boxx \langle n_j +1\rangle.$$ 
 
 \begin{defin} A lattice path $\psi: \langle n+1\rangle {\to} \langle p +1\rangle\boxx \langle q +1\rangle$ has $c$ corners  if the morphism in $\langle p +1\rangle\boxx \langle q +1\rangle$ given by the composite 
   $$\langle1\rangle \to \langle n+1\rangle \stackrel{\psi}{\to} \langle p +1\rangle\boxx \langle q +1\rangle$$ 
   has exactly $c$ generators of the form $(f,id)\circ(id,g)$ or $(id,f)\circ(g,id)$  in which both $f$ and $g$ are not equal to the identities.   
 \end{defin} 
 \begin{remark} It is easy to see that this number $c$ does not depend on the presentation of the composite above. This number is exactly the number of changes of directions  of the lattice path. 
 
 \end{remark}

   \begin{defin} The complexity index of the lattice path $\psi$ is 
   $$\cc(\psi) = \max_{i<j}\cc_{ij}(\psi),$$ 
   where $\cc_{ij}(\psi)$   is number of corners of the lattice path $\psi_{ij}.$ 

   \end{defin}

\begin{defin}[\cite{BB}] The operad
$\LL^{(n)}$ is the suboperad of  $\LL$ which consists of the lattice paths of complexity less or equal to $n.$ 
\end{defin}

Thus the operad $\LL$ has an exhaustive  filtration by suboperads 
$$\De = \LL^{(0)} \subset \LL^{(1)} \subset \ldots \LL^{(n)}\subset  \ldots  \subset \LL.$$

We will also need  the following description of the lattice paths introduced in \cite{BBM}.  A lattice path from $\LL^{(n)}(n_1,\ldots,n_k;m)$:
 $$\psi: \langle m+1\rangle \stackrel{}{\to} \langle n_1 +1\rangle\boxx\ldots\boxx\langle n_k+1\rangle$$
 has its 'shape' in $\LL^{(n)}(n_1,\ldots,n_k;0)$ as the result of the composition 
 $$\psi: \langle 1\rangle \to     \langle m+1\rangle \stackrel{}{\to} \langle n_1 +1\rangle\boxx\ldots\boxx\langle n_k+1\rangle.$$ 
 To reconstruct this path back it is enough to add positive integer labels (called {\it multiplicity} ) to each vertex of the lattice.  We add the label $p>1$ to such a vertex $v$ if the full preimage $\psi^{-1}(id_v)$ contains exactly $p-1$ generators $\{\bar{0},\ldots,\bar{m}\}$ or, equivalently, exactly $p$ identity morphisms. We assign a label $1$ if this preimage contains only aone identity and $0$ if such a preimage is empty. 
 Informally we think about the labelling as the time the path `spends'  at $v$ along its way from minimum point to the maximum. 
 \begin{remark}
 According to this definition the label of the endpoints of the path is greater or equal to $1.$ This is different from the agreement adopted in \cite{BBM}. Their labelling is obtained from ours by subtracting $1$ from the endpoints labels. This is because in \cite{BBM} it was convenient to take into account the number of internal points along the path. 
 
 Yet another small difference is that we label all points on the lattice, not only the points along the path. Of course, if $v$ is not on the path its label is $0$ so the information is exactly the same. But we prefer to label all points  because it is a little bit  easier to see the action of  simplicial operators on a lattice path  this way.  
  
 \end{remark}      

\subsection{First movement order and paths of complexity one}

In this section we investigate a natural map from the set of lattice paths to the symmetric groups, which we call the first movement order. 

We identify an element of the symmetric group $\Sigma_k$ with a linear order on the set $\{1,\ldots,k\}.$ A lattice path
$$\psi:\langle n+1\rangle \to \langle n_1 +1\rangle\boxx\ldots\boxx\langle n_k+1\rangle $$ 
determines such a linear order by a simple rule: an element $i\in \{1,\ldots,k\}$ is less then $j\in \{1,\ldots,k\}$ if the move in the direction $i$ appears in $\psi$ before a move in the direction $j.$ 
\nl
More formally we can define this as a map 
$$\varpi:\LL(n_ 1,\ldots,n_k;n)\to \LL(0,\ldots,0;0).$$
Indeed, the set $\LL(0,\ldots,0;0)$ is the set of nondecreasing paths on a unit box from $(0,\ldots,0) \to (1,\ldots,1)$ which go on the edges. Such a path is completely determined by the choice of linear order on the set  of directions $\{1,\ldots,k\}.$ 

\begin{example}
A lattice path $\psi:\langle 1 \rangle \to \langle 1 \rangle \boxx \langle 1 \rangle \boxx \langle 1 \rangle $ with $\varpi(\psi) = (321)$  
$$\xygraph{
!{<0cm,0cm>;<.7cm,0cm>:<0cm,.7cm>::}
!{(0,2)}*{\circ} ="lu"  !{(2,2)}*{\circ}="ru"
!{(-1,1)}*{\circ} ="ul"   !{(1,1)}*{\circ}="ur"
 !{(0,0)}*{\circ} ="ld"  !{(2,0)}*{\circ}="rd"
!{(-1,-1)}*{\circ} ="dl"   !{(1,-1)}*{\circ}="dr"
"dl":"ld"     "ld":"lu"  "lu":"ru"   
"ul":@{..}"ur"  "ul":@{..}"lu"  "ul":@{..}"dl"  "dl":@{..}"dr"  "dr":@{..}"ur"  "dr":@{..}"rd"  "rd":@{..}"ru"  "ur":@{..}"ru"  "ld":@{..}"rd"
}$$
Here the path is the sequence of vertices $(0,0,0),\ (0,0,1),\ (0,1,1),\ (1,1,1)$ of the cube.
\end{example}

It is straightforward that $\LL(0,\ldots,0;0)$ is the single colour suboperad of $\LL$ isomorphic to $\A ss.$ 
\nl
Formally the map $\varpi$ is induced by precomposition of the unique morphism of intervals $\langle 1 \rangle \to \langle n+1 \rangle $ and composition with the product
$\langle n_1+1\rangle \boxx \ldots \boxx \langle n_k+1 \rangle \to \langle 1\rangle \boxx \ldots \boxx \langle 1 \rangle,$ where 
$ \langle n_i+1\rangle \to  \langle 1 \rangle$ is the interval map, which sends any $0< a \le n_i+1$ to $1$  (or, in terms of generators, it sends $\bar{0}$ to $\bar{0}$ and any other 
generators to the identity of $1$.)

The following obvious lemma is useful.
\begin{lem}\label{xivalue} The value of the first movement order map $\varpi$ on the composite 
$$\langle n+1\rangle \to \langle n'+1\rangle \stackrel{\psi}{\to} \langle n'_1 +1\rangle\boxx\ldots\boxx\langle n'_k+1\rangle \stackrel{\psi_1\boxx\ldots\boxx\psi_k}{\longrightarrow} \langle n_1 +1\rangle\boxx\ldots\boxx\langle n_k+1\rangle $$ 
is equal to $\varpi(\psi).$
\end{lem}

\begin{lem}\label{LAss} The restriction of the map $\varpi$  to the lattice paths of complexity $1$ is a map of coloured operads
$$\varpi^{(1)}:\LL^{(1)}\to \A ss.$$
\end{lem}
\begin{proof} It is not hard to see that for any $\psi: \langle n+1\rangle \to \langle p +1\rangle\boxx\langle q +1\rangle $ of complexity $1$
there is a unique map $\langle 1\rangle \to  \langle 1\rangle \boxx \langle 1\rangle $ making the following diagram commutative:
\begin{equation}
    \xymatrix@C = +4em{
      \langle n+1\rangle \ar[r] &\langle p +1\rangle\boxx\langle q +1\rangle 
      \\
\langle 1\rangle \ar[r]^{}\ar[u] &\langle 1\rangle\boxx\langle 1\rangle \ar[u]}    
\end{equation}
In fact this bottom lattice path can be obtained as the composite 
$$\langle 1\rangle \to \langle n+1\rangle \to \langle p +1\rangle\boxx\langle q +1\rangle \to \langle 1\rangle\boxx\langle 1\rangle .$$ 
It follows easily then that $\varpi$ restricted to the lattice paths of complexity $1$ respects operadic composition. 
\end{proof}

Similarly define $\varpi^{(c)}$ as a restriction of $\varpi$ to the lattice paths of complexity no more then $c.$ The following example shows that for $c>1$ the corresponding map $\varpi^{(c)}:\LL^{(c)}\to \A ss $ is not a map of operads. 
\begin{example} Let $c=2$ and $a:\langle 2 \rangle \to \langle 2 \rangle \boxx \langle 2 \rangle$ be the lattice path of complexity $2$ as on  picture below:
$$\xygraph{
!{<0cm,0cm>;<1cm,0cm>:<0cm,1cm>::}
!{(-1,1)}*{\circ} ="ul" !{(0,1)}*{\circ} ="u"  !{(1,1)}*{\circ}="ur"
!{(-1,0)}*{\circ} ="l" !{(0,0)}*{\circ} ="m"  !{(1,0)}*{\circ}="r"
!{(-1,-1)}*{\circ} ="dl" !{(0,-1)}*{\circ} ="d"  !{(1,-1)}*{\circ}="dr"
!{(-1,-.78)}*+{\scriptstyle 1}  !{(-.18,0)}*+{\scriptstyle 1} !{(1,.78)}*+{\scriptstyle 1}
!{(-.15,1.15)}*+{\scriptstyle 0}  !{(.15,-1.15)}*+{\scriptstyle 0}
"dl":"d"     "d":"m" 
"m":"u"     "u":"ur" 
} $$
Then $\varpi^{(2)}(a) = (12)$ the identity permutation.   Consider the operadic multiplication of $a$ and  two unary operations $b:\langle 2 \rangle \to  \langle 1 \rangle$ where $b(1) = 0$
and $id:\langle 2 \rangle \to  \langle 2 \rangle.$ 
The result is the lattice path $c$
$$\langle 2 \rangle \stackrel{a}{\longrightarrow } \langle 2 \rangle \boxx \langle 2 \rangle \stackrel{b\square id}{\longrightarrow}   \langle 1 \rangle \boxx \langle 2 \rangle      $$
given by the picture
$$\xygraph{
!{<0cm,0cm>;<1cm,0cm>:<0cm,1cm>::}
!{(0,1)}*{\circ} ="u"  !{(1,1)}*{\circ}="ur"
!{(0,0)}*{\circ} ="m"  !{(1,0)}*{\circ}="r"
!{(0,-1)}*{\circ} ="d"  !{(1,-1)}*{\circ}="dr"
 !{(-.18,0)}*+{\scriptstyle 1} !{(1,.78)}*+{\scriptstyle 1}
!{(-.15,1.15)}*+{\scriptstyle 0}  !{(.15,-1.15)}*+{\scriptstyle 0}
     "d":"m" 
"m":"u"     "u":"ur" 
} $$
Clearly $\varpi^{(2)}(c) = (21)$ because the first movement is by the second coordinate. 
On the other hand the result of multiplication in $Ass$  of $\varpi^{(2)}(a) = (12)$ and $\varpi^{(2)}(b) = \varpi^{(2)}(id) = 1$ is the  permutation $(12) \ne \varpi^{(2)}(c).$   
\end{example}


\subsection{Shuffles and lattice paths}

\begin{defin} A lattice path $\psi:\langle n+1\rangle \to \langle n_1 +1\rangle\boxx\ldots\boxx\langle n_k+1\rangle $ is {\it a shuffle path}   if for any $0\le i \le n$ the morphism
$\psi(\bar{i})$ is one of the generators of the form $(id,\ldots,id,\bar{j},id,\ldots,id),$ where $ 0\le j \le n_s$ for an $1\le s\le k.$   
\end{defin}

\begin{lem} Shuffle-paths form a suboperad $\shuf$ of the lattice path operad.  
\end{lem} 

Relations to classical shuffles.  
\begin{lem}\label{classical}  Any shuffle-path $\psi$ determines a permutation $\mu(\psi)$ of $\{0,1,\ldots,n\}$ as follows:
 $\mu(\psi)(i) = j+ (n_1+1)+(n_2+1)+\ldots+(n_{s-1}+1),$ where $\bar{j}$ is on $s$-th place in $\psi(\bar{i})= (id,\ldots,id,\bar{j},id,\ldots,id).$ 
This formula establishes a one-to-one correspondence between  $(n_1+1,\ldots,n_k+1)$-shuffles and elements from 
 $\shuf(n_1,\ldots,n_k;n_1+\ldots+n_k +k-1).$    
 \end{lem}
 \begin{proof} 
 This is classical \cite{lo}.  
 \end{proof}

 \begin{lem}
 \begin{enumerate}\item For any factorisation of a lattice path 
$$ \langle n+1\rangle \to \langle n'+1\rangle \stackrel{\psi'}{\to}  \langle n_1 +1\rangle\boxx\ldots\boxx\langle n_k+1\rangle $$ 
 $\cc(\psi') = \cc(\psi)$ and $\varpi(\psi)=\varpi(\psi').$
 \item Any lattice path  $\psi:\langle n+1\rangle \to \langle n_1 +1\rangle\boxx\ldots\boxx\langle n_k+1\rangle$ admits a unique factorisation 
\begin{equation}\label{shufactor} \langle n+1\rangle \to \langle n^\dag+1\rangle \stackrel{\psi^\dag}{\to}  \langle n_1 +1\rangle\boxx\ldots\boxx\langle n_k+1\rangle ,\end{equation} where $\psi^\dag$ is a shuffle path.  
\end{enumerate}
\end{lem} 
 
\begin{proof} 
First part of the Lemma is obvious.  
\nl
For the second part we construct $\langle n+1\rangle \stackrel{\alpha}{\to} \langle n^\dag+1\rangle$ as follows. Let identify the ordinal $[n]$ with the naturally ordered set 
 $\{\bar{0},\ldots,\bar{n}\}.$ For each $\bar{i}\in supp(\psi)$ we have a unique presentation $\psi(\bar{i}) = p_0^i \circ p_1^i\circ\ldots\circ p_{l_i}^i$ for certain generators from  
$ \langle n_1 +1\rangle\boxx\ldots\boxx\langle n_k+1\rangle.$  We then consider an ordinal $[l_i]$ and take $[n^\dag] = [l_0]\ast[l_1]\ast\ldots\ast[l_n]$, where $\ast$ means the ordinal sum (concatenation) of ordinals. We consider each $[l_i]$ as a subordinal of $[n^\dag]$ in a natural way.  We then have a map $\beta: [n^\dag]\to [n]$ which sends each element of the subordinal $[l_i]$ to $\bar{i}\in [n].$     Let $\alpha:\langle n+1\rangle \stackrel{}{\to} \langle n^\dag+1\rangle$ be its Joyal's dual. By definition, $\alpha (\bar{i})$ is the composite of the generators from $\langle l_i +1 \rangle.$ We then define $\psi^\dag$ on such a generator as equal to the corresponding $p^i_j.$ 
\nl
Thus we have a required factorisation. Uniqueness follows from the fact that $\psi^\dag$ is already determined by the `shape' of $\psi$ i.e. by the composite
$$ \langle 1\rangle \to \langle n+1\rangle \stackrel{\psi}{\to}  \langle n_1 +1\rangle\boxx\ldots\boxx\langle n_k+1\rangle .$$ 
\end{proof}
  
\begin{lem} \label{40} 
Let $\psi:\langle n+1\rangle \to \langle n_1 +1\rangle\boxx\ldots\boxx\langle n_k+1\rangle$ be a lattice path and $\psi_1,\ldots, \psi_k$ be a set of composable lattice paths, that is the composition $\psi(\psi_1,\ldots,\psi_k)$ is defined in the lattice path operad. Then there exists a shuffle path $\psi^\ddag$ such that the composite $\psi^\ddag(\psi^\dag_1,\ldots,\psi^\dag_k)$ is defined and such that: 
\begin{enumerate} \item $\varpi(\psi(\psi_1,\ldots,\psi_k)) =    \varpi( \psi^\ddag(\psi^\dag_1,\ldots,\psi^\dag_k)  ) ;$ and
\item $\varpi(\psi) (\varpi(\psi_1),\ldots,\varpi(\psi_k)) = \varpi(\psi^\ddag)(\varpi(\psi^\dag_1),\ldots,\varpi(\psi^\dag_k)),$  
\end{enumerate}
where in the last raw the composite is computed in the operad $\A ss.$
\end{lem} 
\begin{proof}  Consider the composite $\psi(\psi_1,\ldots,\psi_k):$
 $$\langle n+1\rangle \stackrel{\psi}{\to} \langle n_1 +1\rangle\boxx\ldots\boxx\langle n_k+1\rangle \stackrel{\psi_1\boxx\ldots\boxx\psi_k}{\longrightarrow} 
  \langle m_{1,l_1} +1\rangle\boxx\ldots\boxx\langle m_{k,l_k}+1\rangle.$$ 
  We factorise $\psi$ and $\psi_i, \ 1\le i\le k$ as in (\ref{shufactor}) to get $\psi^\dag$ and $\psi_i^\dag.$ Then we factorise the composite
 \begin{equation} \label{comp} \langle n^\dag +1\rangle  \stackrel{\psi^\dag}{\longrightarrow}     \langle n_1 +1\rangle\boxx\ldots\boxx\langle n_k+1\rangle {\longrightarrow}   
  \langle n^\dag_1 +1\rangle\boxx\ldots\boxx\langle n^\dag_k+1\rangle \end{equation}
  as a morphism of intervals  followed  by a shuffle path $\psi^\ddag.$ 
  
From the uniqueness of factorisation follows that $(\psi(\psi_1,\ldots,\psi_k))^\dag =    \psi^\ddag(\psi^\dag_1,\ldots,\psi^\dag_k)$ and therefore 
$\varpi(\psi(\psi_1,\ldots,\psi_k)) =    \varpi( \psi^\ddag(\psi^\dag_1,\ldots,\psi^\dag_k)  )$ from Lemma  \ref{xivalue}. 

Now, $\varpi(\psi) = \varpi(\psi^\dag)$ and  $\varpi(\psi_i) = \varpi(\psi_i^\dag).$ The value of $\varpi$ on the composite (\ref{comp}) is  equal to $\varpi(\psi^{\dag})$ and also to $\varpi(\psi^{\ddag})$ by 
Lemma \ref{xivalue} again. Therefore, $$\varpi(\psi) (\varpi(\psi_1),\ldots,\varpi(\psi_k)) = \varpi(\psi^\ddag)(\varpi(\psi^\dag_1),\ldots,\varpi(\psi^\dag_k)).$$ 
 \end{proof} 

\subsection{Shuffle paths and their sketches}
  
The lattice paths has another presentation as strings of integers with a number of vertical bars between them \cite[Section 2.2]{BB}. The shuffle paths correspond to the strings where there is exactly one bar between each pair of consecutive entries. This presentation can be reformulated as follows. Let  $\FM(k)$ be a subset of  elements of the free monoid $FM(p_1,\ldots,p_k)$ on $k$ elements which contain each generator at least once.  Such an element is a word $p$ of the variables $p_1,\ldots,p_k.$ Let $\FF(n_1,\ldots,n_k; m)\subset \FM(k)$ be the subset of words in which  a variable $p_i$ appears $n_i +1$ times if $n_1+\ldots+n_k = m+1-k$ and an empty set otherwise. These sets form a $\bSet$-operad $\FF$ whose composition $\circ_i$ can be described as follows. Let $p\in \FF(p_1,\ldots,p_k), \ q\in \FF(q_1,\ldots,q_m), \ 1\le i \le k.$ We can write $p$  as  a string $p_{i_1}p_{i_2} p_{i_3}\ldots $ and similarly for $q.$  Observe that in this presentation we put all generators in degree $1$ that is we write $p_i \ldots p_i$ ($d$-times) for $p_i^d.$    We suppose that the number of occurrence of the variable $p_i$ in $p$ is equal to the length of the string $q.$ Then the string $p\circ_i q$ is obtained by replacing $j$-th occurrence of the  variable  $p_i$ in $p$ by the $j$-th element from $q$ and then renumbering of the variables.  

\begin{lem}\label{string} 
There is an isomorphism between shuffle paths operad $\shuf$ and the operad $\FF.$   
\end{lem} 
\begin{proof} According to \cite{BB} the path $$\psi: \langle n+1\rangle \stackrel{}{\to} \langle n_1 +1\rangle\boxx\ldots\boxx\langle n_k+1\rangle$$ determines a string of integers $i_0,\ldots,i_n,$ where $i_p = i$ if  $\psi(\bar{p})= (id,\ldots,id,\bar{j},id,\ldots,id)$ and $j$ is on $i$-th place.  We interpret  this string as a word from $\FM(k).$  
\end{proof}

We now introduce a sequence of sets  $\TT(k), k\ge 1$ whose  elements we call {\it sketches}. Sketches help us to handle the combinatorics of complexity and of the first movement order. The sequence $\TT(k)$ does not form an operad in $\bSet$ but, in fact, can be equipped with an operad structure after linearisation. This linearised version was introduced by McClure and Smith in \cite{ms} under the name {\it surjection operad}. Sketches form a linear basis of this operad and can be identified with {\it nondegenerate surjections} \cite[Definition 2.13]{ms}. The facts below are not completely new and can be found in \cite{ms}. We present them here for the readers convenience.  

Let $\TT(k)\subset FI(s_1,\ldots,s_k)$ be the subset of elements of the free monoids on $k$ idempotent generators which contains each generator at least once.
 They are, of course,  equivalence classes of words in the variables $s_1,\ldots,s_k,$ where the equivalence relation is generated by $s_i = s_is_i$ for all $1\le i\le k .$  We say that such a word is a reduced form with respect to a variable $s_i$  if it does not contain a repeated subword of the form $s_is_i.$  It is a reduced form if it is a reduced form with respect to all variables. Obviously, every element of $\TT(k)$ is uniquely represented by a reduced  word and we will suppose by default that the reduced words are exactly the elements of $\TT(k).$  We can then define a length of a sketch ${\bold l}(s)$ as the length of its reduced form.

For each $1\le i<j \le k$ there is a projection $\pi_{ij}:\TT(k)\to \TT(2)$ which is computed by substituting the unit $e$ to the word $s$ for all variables not equal to $s_i$ and $s_j.$
For example: 
$$\pi_{23}(s_1s_3s_1s_3s_4s_1s_2s_3s_1s_2) = es_3es_3 ee s_2 s_3 e s_2 = s_2 s_1 s_2 s_1.$$

We define the complexity index of an element $s\in \TT(2)$ as ${\bold l}(s) -1.$ For an element $s\in \TT(k)$ and $1\le i<j \le k$ we define the complexity  index $\cc_{ij}(s)$ as the complexity index of the corresponding projection on $\pi_{ij}(s).$
The complexity index $\cc(s)$ of the sketch $s$ is the maximum of the pairwise complexity indices.   

Finally, we define the first movement order $\varpi(s)$ of a sketch $s\in \TT(k)$ as a linear order on $\{1,\ldots,k\}$ in which variables appear first in the word $s.$

 An {\it expansion} $(s)_i$ of a sketch $s\in \TT(n)$ at a variable $s_i$ is a string of variables $s_1,\ldots,s_n$ such that:
 \begin{enumerate}\item
  As an element of $FI(s_1,\ldots,s_n)$ it is equal to $s$
  \item It is in a reduced form with respect to all variables except for $s_i.$ 
  \end{enumerate}
An example of an expansion  of   $s = s_{1}  s_{2}  s_{1} s_{2} s_{1}  $ at $s_1$ is 
   \begin{equation}\label{ex} s = s_1s_{1}  s_{2}  s_{1}s_1s_1 s_{2} s_{1}.  \end{equation}
 
For a shuffle path $\psi:\langle n+1\rangle \to \langle n_1 +1\rangle\boxx\ldots\boxx\langle n_k+1\rangle $ we then associate a sketch $\tr(\psi)\in \TT(k)$ and its expansions 
$\tr_i(\psi) =  (\tr(\psi))_i$ for each $1\le i\le k$ as follows:
     $\psi$ determines an element of $\FM(k)$ as in Lemma \ref{string} .  Then $\tr(\psi)$  is the image of this element under  a natural  reduction map $\FM(k)\to \TT(k).$ To get $\tr_i(\psi)$ we reduce the same word by all variables except for $i.$

Finally, we can substitute a shuffle path  $t\in \FM(d)$  to an expansion $(s)_i$ of $s$ provided the length of $t$ is equal to the number of occurrences of $s_i$ in $(s)_i.$    For this we replace the $j$-th occurrence of $s_i$ by the   $j$-th term of $t$ (in natural order). Then change $t$ to $s$ and renumber  by adding $i$ to $t_j$ and $i+c+d$ to the variable $s_{i+c}$ for $c>0.$  We then apply $\tr$ to the resulting shuffle path and produce a sketch. We denote this operation $(s)_i\circ t.$

For example, for an expansion $(s)_1$  from the example (\ref{ex}) and  a shuffle path $t = t_1t_2t_1t_3t_1t_2t_3.$ 
the result of substitution  $(s)_1\circ t$  is:
$$(t_1)(t_{2})  s_{2}  (t_{1})(t_3)(t_2) s_{2}( t_{3}) = s_1s_{2}  s_{4}  s_{1}s_3s_2 s_{4} s_{3}                                           .$$



\begin{lem} For two shuffle path $\psi$ and $\omega$ the following is true:
\begin{enumerate}\item $\varpi(\psi) = \varpi(\tr(\psi)) ;$
\item  $\cc_{ij}(\psi) = \cc_{ij}({\tr(\psi)}) = \link(\pi_{ij}(p(\psi));$
\item $\tr(\psi\circ_i \omega) = \tr_i(\psi)\circ \omega$ if $\psi\circ_i \omega$ is defined.
\end{enumerate}
\end{lem}
\begin{proof} Obvious from definitions.
\end{proof}
 




\end{document}